\documentclass[11pt,UTF-8,reqno]{amsart}
\usepackage{enumerate, bbm}
\usepackage{amssymb,url,color, booktabs}
\usepackage{mathrsfs}
\usepackage{cite}
\usepackage[utf8]{inputenc}

\usepackage{color}
\usepackage[colorlinks=true]{hyperref}
\hypersetup{
    linkcolor=blue,          
    citecolor=red,        
    filecolor=blue,      
    urlcolor=cyan
}

\usepackage{color}
\definecolor{MyDarkBlue}{cmyk}{0.8,0.3,0.8,0.4}
\definecolor{yellow}{rgb}{0.99,0.99,0.70}
\definecolor{white}{rgb}{1.0,1.0,1.0}
\definecolor{black}{rgb}{0.00,0.00,0.00}

\numberwithin{equation}{section}

\newcommand{\be}{\begin{eqnarray}}
\newcommand{\ee}{\end{eqnarray}}
\newcommand{\ce}{\begin{eqnarray*}}
\newcommand{\de}{\end{eqnarray*}}
\newtheorem{theorem}{Theorem}[section]
\newtheorem{lemma}{Lemma}[section]

\newtheorem{definition}{Definition}[section]
\newtheorem{proposition}{Proposition}[section]

\newtheorem{corollary}{Corollary}[section]

\newtheorem{remark}[theorem]{Remark}

\def\[{{\Big[}}
\def\]{{\Big]}}
\def\<{{\langle}}
\def\>{{\rangle}}
\def\({{\big(}}
\def\){{\big)}}

\def\bb2{{\boldsymbol{2}}}

\def\={&\!\!=\!\!&}
\def\RR{\mathbb{R}}

\def\b1{{\mathbbm 1}}

\def\geq{\geqslant}
\def\leq{\leqslant}

\def\le{\leqslant}

\def\[{{\Big[}}
\def\]{{\Big]}}
\def\<{{\langle}}
\def\>{{\rangle}}

\def\={&\!\!=\!\!&}
\def\bt{\begin{theorem}}
\def\et{\end{theorem}}
\def\bl{\begin{lemma}}
\def\el{\end{lemma}}
\def\br{\begin{remark}}
\def\er{\end{remark}}
\def\bd{\begin{definition}}
\def\ed{\end{definition}}
\def\bc{\begin{corollary}}
\def\ec{\end{corollary}}

\def\geq{\geqslant}
\def\leq{\leqslant}

\def\le{\leqslant}

\def\<{\langle} \def\>{\rangle}

\def \eref#1{\hbox{(\ref{#1})}}

\allowdisplaybreaks

\begin{document}

\title[Mean Field SPDEs]
{Mean Field  Stochastic Partial Differential Equations with Nonlinear Kernels$^\dagger$}

\thanks{$\dagger$
This work is supported by National Key R\&D program of China (No. 2023YFA1010101). The research of W. Hong is also supported by  NSFC (No.~12401177) and  Basic Research Program of Jiangsu (No.~BK20241048). The research of S. Li is also supported by NSFC (No.~12371147). The research of W. Liu is also supported by NSFC (No.~12171208, 12090011,12090010) and the PAPD of Jiangsu Higher Education Institutions. }


\maketitle
\centerline{\scshape Wei Hong, Shihu Li,   Wei Liu}

\medskip

\vspace{1mm}
 {\footnotesize
\centerline{   School of Mathematics and Statistics, Jiangsu Normal University, Xuzhou 221116, China,}}

\begin{abstract}
This work focuses on the mean field stochastic partial differential equations with nonlinear kernels. We first prove the existence and uniqueness of strong and weak solutions for mean field stochastic partial differential equations in the variational framework, then establish the convergence (in certain Wasserstein metric) of the empirical laws of interacting systems to the law of solutions of mean field equations, as the number of particles tends to infinity. The main challenge lies in addressing the inherent interplay  between
 the high nonlinearity of operators and the non-local effect of  coefficients that depend on the measure. In particular, we do not need to assume any exponential moment control condition of solutions, which extends the range of the applicability of our results.

As applications, we first study a class of finite-dimensional interacting particle systems with polynomial kernels, which are commonly encountered in fields such as the data science and the machine learning. Subsequently, we present several illustrative examples of infinite-dimensional interacting systems with nonlinear kernels, such as the stochastic climate models, the stochastic Allen-Cahn equations, and the stochastic Burgers type equations.


\end{abstract}
\maketitle \rm

\tableofcontents

\section{Introduction}
\subsection{Background and motivation}
 Mean field equations are a type of stochastic equations whose coefficients depend on the law of solutions, which were initially developed by McKean \cite{M66}. They are also known as distribution-dependent stochastic equations or McKean-Vlasov stochastic equations in the literature (cf.~\cite{BR1,BR2}), which typically arise as the limit of mean field interacting particle systems. For instance, we consider the dynamics of an $N$ ($N\geq 1$) particles system
governed  by the following stochastic differential equations (SDEs)
\begin{equation}\label{I03}
d X_t^{N,i} = b(X_t^{N,i},\mathcal{S}_t^N)dt+\sigma(X_t^{N,i},\mathcal{S}_t^N)d W_t^i,~\mathcal{S}_t^N:=N^{-1}\sum_{j=1}^N\delta_{X_t^{N,j}},
\end{equation}
where $X_t^{N,i}$ represents the position of $i$-th particle ($1\leq i\leq N$) at time $t$,  $\{W_t^i\}_{1\leq i\leq N}$ are independent standard Wiener processes.
Notably, the mean field interaction here is described as the dependency of coefficients on the empirical measure of the system. Under some regularity conditions, e.g.~globally Lipschitz (cf.~Sznitman \cite{S1}), on coefficients $b$ and $\sigma$, it is well-known that the empirical laws  of the particles system (\ref{I03}) will converge
weakly to the law of solutions of following mean field SDE
\begin{equation*}\label{I04}
d X_t = b(X_t, \mathscr{L}_{X_t})dt+\sigma(X_t, \mathscr{L}_{X_t})d W_t,
\end{equation*}
where $\mathscr{L}_{X_t}$ is the time marginal law of $X_t$. This kind of macroscopic limit behavior for the interacting particle system is usually referred as the mean field limit (also called the propagation of chaos)  in the literature, which is widely used in the analysis of large-scale systems in various fields, including statistical physics,  neuroscience and population dynamics. We also refer interested readers to \cite{BLPR,E1,GHM,HSS,HLLS23,HW1,TT,WFY} and references therein for the study of mean field SDEs.

\subsubsection{Finite-dimensional case}
  Classical results on mean field limit assume that the interacting kernels are globally Lipschitz and bounded. However, it is noted that such hypotheses are most often too strong in applications of practice.  Handling certain singular interacting kernels or only locally Lipschitz interacting kernels with polynomial growth has been a longstanding and challenging question. As mentioned in Section 3.1.2 in \cite{CD22a}, there is no real hope for better results at this level of generality, many directions have been explored to weaken the hypotheses in specific cases. Motivated by this, in the present work, we focus on the mean field limit for interacting particle systems with nonlinear kernels. Since in many fields, such as data science and machine learning, the particle systems of interest often involve polynomial interacting kernels and only satisfy certain {\it{local monotonicity}} condition.

  \vspace{1mm}
  One  motivating example is the following interacting particle system $(X^{N,1},\ldots,$ $X^{N,N})\in\mathbb{R}^N$ given by
\begin{equation}\label{in05}
dX_t^{N,i}=-\nabla \Phi(X_t^{N,i})\text{Var}[\mathcal{S}^N_t]dt
+\sqrt{2}dW_t^i,
\end{equation}
where $\Phi(\cdot)$ is a real value function,  $ \text{Var}[\cdot]$ is
 the empirical variance between particles, i.e.,
 $$\text{Var}[\mathcal{S}^N_t]=\frac{1}{N}\sum_{j=1}^{N}\big(X_t^{N,j}-\bar{X}_t^N\big)^{2},~~~\bar{X}^N_t:=\frac{1}{N}\sum_{i=1}^{N} X_t^{N,i}.$$
The interacting particle system \eref{in05} is constructed from the \emph{Metropolis Adjusted Langevin Algorithm} (MALA), which has important applications in data science (cf.~\cite{DL21,GHLS20} and the reference therein). Since the variance function is a quadratic quantity, the Lipschitz assumption  in McKean's theorem does not hold.
The formal mean field limit of (\ref{in05}) as $N\to \infty$ is given by the following Fokker-Planck equation
\begin{equation*}
\partial f_t(x)=\Delta f_t(x)-\nabla\cdot \Big(f_t(x) \nabla \Phi(x)\text{Var}[f_t]\Big),
\end{equation*}
where
$$\text{Var}[\mu]:=\int_{\mathbb{R}}(x-m[\mu])^2\mu(dx),~~m[\mu]:=\int_{\mathbb{R}}x\mu(dx),$$
and the associated mean field SDE is given by
\begin{equation}\label{in06}
dX_t=-\nabla \Phi(X_t)\text{Var}[\mathscr{L}_{X_t}]dt+\sqrt{2}dW_t.
\end{equation}
 For simplicity, here we only consider the one-dimensional case (i.e.~$X^{N,i}\in\mathbb{R},1\leq i\leq N$).  However, it is important to note that the results presented in this paper are also applicable to the general dimension $d$ (in this case, the empirical variance is replaced by the  empirical covariance).

Moreover, the interacting particle systems with polynomial kernels also have wide applications in the field of machine learning. For instance, the following type of interacting particle system is derived by the \emph{Stein Variational Gradient Descent} (SVGD)  in \cite{LW16} (see also \cite[Lemma 45]{DNS23}),
  \begin{equation}\label{mean03}
  dX^{N,i}_t=-\frac{1}{N}\sum_{j=1}^N\nabla_y\kappa(X^{N,i}_t,X^{N,j}_t)dt-\frac{1}{N}\sum_{j=1}^N\kappa(X^{N,i}_t,X^{N,j}_t)\nabla \Phi(X^{N,j}_t)dt,
  \end{equation}
where the functions $\kappa,\Phi$ are typically chosen as
  \begin{equation}\label{mean01}
\kappa(x,y)=x\cdot y~\text{and}~\Phi(y)=\frac{1}{2}y^2.
  \end{equation}
 In this case, the corresponding mean field limit of system \eref{mean03} is given by the following mean field equation
\begin{equation}\label{mean003b}
dX_t=-X_tdt-\int_{\mathbb{R}}\big(\kappa(X_t,y)\nabla \Phi(y)\big)\mathscr{L}_{X_t}(dy)dt.
\end{equation}

It is evident that  the interacting kernels in systems \eref{in05} and \eref{mean03} are of polynomial growth and do not satisfy globally Lipschitz or monotone conditions (cf.~\cite{WFY}). This distinguishes them from the assumptions typically imposed in existing literature, such as the works \cite{BCC11,E1,GHM,WFY}. It worth noting that, as stated in \cite[Remark 46]{DNS23}, the main results presented in \cite{DNS23} cannot be directly applied to the polynomial kernel (\ref{mean01}), since it does not satisfy the integrally strictly positive-definite assumption proposed in \cite{DNS23}.
With these applications in mind, in the current study, we aim to establish a general setting for handling a class of particle systems with polynomial kernels.
Then the mean field limit of \eref{in05} and \eref{mean03} as well as the well-posedness of \eref{in06} and \eref{mean003b} can be directly derived using our general theory.

By the way, we also mention that the mean field limit for certain singular interacting kernels has achieved significant progress recently.
  Several notable works have contributed to this area, including references such as \cite{ESY10,HRZ,JW18,WZF}. In particular, Jabin and Wang \cite{JW18} have made important contributions by  studying the mean field limit for a class of interacting particle systems with $W^{-1,\infty} $ kernels, and they obtain the quantitative convergence rates in terms of the relative entropy for all finite marginal laws of the particles. For more recent significant progress on the mean field limit of interacting particle systems with singular interacting kernels, one could refer to the works by Serfaty \cite{Se} and Wang et al. \cite{WZZ23}.
  We also refer the interested readers to the surveys \cite{CD22a,CD22b,G1,S1} and the monographs \cite{CD1,S2} on this topic.

\subsubsection{Infinite-dimensional case}
In recent years, there has been a growing interest in the study of weakly interacting stochastic partial differential equations (SPDEs) and their limiting mean field counterparts, due to the potential applications in several fields such as neurophysiology and quantum
field theory. As far as we know, the initial research on this topic was done by Chiang et al.~\cite{CKS} and  Kallianpur and Xiong \cite{KX}, in which they considered the mean field limit for the following interacting system
$$dX^{N,i}_t=\Bigg\{\Delta X^{N,i}_t+\frac{1}{N}\sum_{j=1}^Nb_t(X^{N,i}_t,X^{N,j}_t)\Bigg\}dt+dW^i_t,$$
taking values in the dual of a countably Hilbertian nuclear space and with globally Lipschitz type conditions,
 which was later extended in \cite{BKK} to more general non-nuclear spaces, here and in the following, $\Delta$ denotes the classical Laplace operator.
The results in \cite{CKS,KX} have applications in describing random strings and the fluctuation of voltage potentials of interacting spatially extended neurons. They are particularly relevant when modeling large numbers of neurons in close proximity to each other, providing a more realistic representation of neural systems.
 E and Shen \cite{ES1} considered a dynamical model of polymer systems, which is described by a system of $N$-coupled SPDEs, assuming suitable
assumptions on the initial and boundary conditions as well as regularity of the coefficients. They proved the mean field limit for the $N$-coupled system and the well-posedness of the mean field limit equation. Criens \cite{C1} generalized the results of \cite{BKK,CKS,ES1} and established the mean field limit for the  weakly interacting semi-linear SPDEs whose coefficients satisfy linear growth and Lipschitz type conditions. Ren et al.~\cite{RTW}  derived the well-posedness for a class of path-distribution dependent  stochastic transport type equations, including stochastic Camassa-Holm equation with distribution dependent coefficients.  Bailleul and Moench \cite{BM} investigated the mean field limit for some systems of interacting fields, whose evolution is given by singular SPDEs of mean field type, under Lipschitz conditions based on paracontrolled calculations. In \cite{CHK}, Crisan, Holm and Korn investigated the stochastic  idealized climate models.  This climate model couples a stochastic PDE for the atmospheric circulation to a deterministic PDE for the circulation of the ocean, which is  derived by the {\it{Stochastic Advection by Lie Transport}} (SALT) method, and the abstract form of such model is given by the following mean field SPDE
 \begin{equation}\label{I007}
dX_t+\big(B(X_t,X_t)-\mathcal{R}(X_t)+\kappa\mathbb{E}X_t\big)dt=\nu\Delta X_tdt+\sigma(X_t)dW_t,
\end{equation}
where $B$ is the usual bilinear transport operator, $\mathcal{R}$ is a linear operator, $\kappa\in\mathbb{R}$, and $W_t$ is a cylindrical Wiener process.

In view of the aforementioned works, a natural question is whether the mean field limit holds for weakly interacting SPDEs with nonlinear kernels.
More recently, Shen et al.~\cite{SSZZ,SZZ22} have made significant progress in studying the following interacting stochastic Allen-Cahn equations
\begin{equation}\label{mean004}
dX_t^{N,i}=\Bigg\{\Delta X_t^{N,i}+X_t^{N,i}-\frac{1}{N}\sum_{j=1}^{N}(X_t^{N,j})^2X_t^{N,i}\Bigg\}dt+dW_t^i,
\end{equation}
where
$\{W_t^i\}_{1\leq i\leq N}$ are independent  cylindrical Wiener processes.
This model arises as the stochastic quantization of the $N$-component generalization of the $\Phi^4_d$ model in the quantum
field theory. They specifically focused on the large $N$ limit problem of the system \eref{mean004}, where such  limit can be characterize by the following mean field stochastic Allen-Cahn equations
$$dX_t=\Big\{\Delta X_t+X_t-(\mathbb{E}X_t^2)X_t\Big\}dt+dW_t.$$
In \cite{GRZ}, Gabriel et al. presented a novel and interesting result that demonstrates mean field stochastic Allen-Cahn equation can be derived as a scaling limit of the classical Allen-Cahn equation with a rough random initial datum, which also provides an independent interest supporting the significance of studying mean field SPDEs. The reader can refer to e.g. \cite{MZ03} and the references therein for the applications of large $N$ limit in quantum field theory.



\subsection{Main results}\label{secmr}
In this paper, our aim is to develop a general result to investigate the mean field limit (or the large $N$ limit problem) for weakly interacting SPDEs with nonlinear interactions, which also covers the finite-dimensional models  \eref{in05} and \eref{mean03} as special cases.
More precisely, we consider the following general weakly interacting SPDEs
 \begin{equation}\label{eqi01}
dX^{N,i}_t=\mathcal{A}(t,X^{N,i}_t,\mathcal{S}^N_t)dt+\mathcal{B}(t,X^{N,i}_t,\mathcal{S}^N_t)dW_t^i,~X^{N,i}_0=\xi^i,
\end{equation}
where $\{W_t^i\}_{1\leq i\leq N}$ are independent cylindrical Wiener processes defined on a complete filtered probability space $\left(\Omega,\mathscr{F},\{\mathscr{F}_t\}_{t\in[0,T]},\mathbb{P}\right)$, the initial random variables $\{\xi^i\}_{1\leq i\leq N}$ are exchangeable. To extend the applicability of our study to various important interacting systems, we will investigate this problem within the generalized variational framework.

The classical variational framework in the theory of SPDEs, under the following Gelfand triple
\begin{equation}\label{gel1}
\mathbb{V}\subset \mathbb{H} \subset \mathbb{V}^*,
\end{equation}
was initially established by Pardoux, Krylov and Rozovskii, see \cite{Par75,KR,LR1,RS}. In their works, they employed the well-known monotonicity trick to verify the well-posedness of
solutions to SPDEs that satisfy the classical monotonicity and coercivity assumptions. In the past several years, the variational framework has been substantially extended in \cite{AV24,HHL,LR2,LR13,LR1,RSZ1,RZZ15} to more general circumstances that satisfy the local or fully local monotonicity conditions.

Let us  briefly collect our main results established in this work.
%
%
First, under general growth and coercivity conditions (see $\mathbf{(A_0)}$-$\mathbf{(A_4)}$ below), we provide
the existence of martingale/weak solutions to  the following mean field SPDE
\begin{equation}\label{eqSPDE01}
dX_t=\mathcal{A}(t,X_t,\mathscr{L}_{X_t})dt+\mathcal{B}(t,X_t,\mathscr{L}_{X_t})dW_t,~X_0\sim\mu_0.
\end{equation}
Then under the following local monotonicity condition
 \begin{eqnarray*}
&&2_{{\mathbb{V}}^*}\langle \mathcal{A}(t,u,\mu)-\mathcal{A}(t,v,\mu),u-v\rangle_{\mathbb{V}}+\|\mathcal{B}(t,u,\mu)-\mathcal{B}(t,v,\mu)\|_{L_2(U,{\mathbb{H}})}^2
\nonumber\\
&&\leq
(C+\rho(u,\mu)+\eta(v,\mu))\|u-v\|_{\mathbb{H}}^2,
\end{eqnarray*}
 we show the existence of (probabilistically) strong solutions to (\ref{eqSPDE01}), by proving the uniqueness of solutions to the corresponding decoupled equation with frozen distributions. However, under this condition, it is not sufficient to obtain the uniqueness of either strong or weak solutions for the mean field equation (\ref{eqSPDE01}), which differs from the classical theory of SPDEs. In fact, the strong and weak uniqueness of solutions to \eref{eqSPDE01} may not hold  as the drift $\mathcal{A}(t,u,\mu)$ is locally monotone with respect to both the spatial variable $u$ and the measure $\mu$ like the following form
 \begin{eqnarray*}
&&2_{{\mathbb{V}}^*}\langle \mathcal{A}(t,u,\mu)-\mathcal{A}(t,v,\nu),u-v\rangle_{\mathbb{V}}+\|\mathcal{B}(t,u,\mu)-\mathcal{B}(t,v,\nu)\|_{L_2(U,{\mathbb{H}})}^2
\nonumber\\
&&\leq
(C+\rho(u,\mu)+\eta(v,\nu))(\|u-v\|_{\mathbb{H}}^2+\mathcal{W}_{2,{\mathbb{H}}}(\mu,\nu)^2).
\end{eqnarray*}
This  highlights the intrinsic difference between mean field  SPDEs and classical  (distribution-independent) SPDEs.  One can  refer to Remark  \ref{re1} below for some typical counterexamples at this point.

Moreover, due to the inherent interplay  between
 the high nonlinearity of operators and the non-local effect of  coefficients that depend on the measure, comparing to the classical theory of locally monotone operators (cf.~\cite{LR2,LR13,LR1,RSZ1}), the well-posedness of solutions in the mean field case is a more challenging problem.
To illustrate this, for example, let us consider  the stopped process $Z_t:= X_{t\wedge\tau}$ of $X_t$ in (\ref{eqSPDE01}), where $\tau$ is some stopping time. In this case, $Z_t$ satisfies
$$Z_t=Z_0+\int_0^{t\wedge\tau}\mathcal{A}(s,Z_s,\mathscr{L}_{X_s})ds
+\int_0^{t\wedge\tau}\mathcal{B}(s,Z_s,\mathscr{L}_{X_s})dW_s,$$
where it is crucial to mention that $X_t$ is still involved in the equation, rather than only $Z_t$. As a result, the classical localization argument cannot be directly applied to deal with the well-posedness of mean field SPDEs \eref{eqSPDE01}.

 To deal with this problem, it is essential to balance the influence between solutions and laws. In this work, we propose two  distinct local monotonicity conditions (i.e.~$(\mathbf{A'_5})$ and $(\mathbf{A''_5})$ below) and establish the uniqueness of strong solutions for mean field SPDE  \eref{eqSPDE01}. Shortly speaking,  $(\mathbf{A'_5})$ is a general local monotonicity condition of $\mathcal{A}(t,u,\mu)$  with respect to the measure $\mu$, but requires $\mathcal{A}(t,u,\mu)$  to be monotone with respect to $u$. On the other hand, $(\mathbf{A''_5})$ is a more general local monotonicity condition of $\mathcal{A}(t,u,\mu)$ with respect to $u$, but imposes restrictions on the measure $\mu$ based on a ``local" Wasserstein distance. More specifically, for any $\mu,\nu\in\mathscr{P}(\mathbb{S})$, $\mathbb{S}:=C([0,T];\mathbb{H})\cap L^{\alpha}([0,T];\mathbb{V})$, we recall the  local $L^2$-Wasserstein distance, which is a modified form of \cite{RTW}, as follows
\begin{equation*}
\mathcal{W}_{2,R}(\mu,\nu):=\inf_{\pi\in\mathscr{C}(\mu,\nu)}\Bigg(\int\|\xi^{\tau_R^\xi\wedge\tau_R^{\eta}}-\eta^{\tau_R^\xi\wedge\tau_R^{\eta}}\|_{T,\mathbb{H}}^2\pi(d\xi,d\eta)\Bigg)^{\frac{1}{2}},
\end{equation*}
where the time  $\tau_R^\xi$ is defined by
$$\tau_R^\xi:=\inf\Big\{t\geq 0:\|\xi_t\|_{\mathbb{H}}+\int_0^t\|\xi_s\|_{\mathbb{V}}^{\alpha}ds\geq R\Big\}\wedge T.$$
Different from the classical $L^2$-Wasserstein distance, the local one means that the distance between two measure flows is determined by the laws of local paths.

Based on the well-posedness of mean field SPDE  \eref{eqSPDE01},  we combine the Galerkin method with the joint tightness argument to study the mean field limit of the infinite-dimensional interacting  system (\ref{eqi01}). More precisely,
we consider the  Galerkin approximation  system of (\ref{eqi01}), i.e.,
$$dX^{(n),N,i}_t=\pi_n\mathcal{A}(t,X^{(n),N,i}_t,\mathcal{S}^{(n),N}_t)dt+\pi_n\mathcal{B}(t,X^{(n),N,i}_t,\mathcal{S}^{(n),N}_t)dW_t^{(n),i},$$
and prove that the sequence
\begin{equation}\label{ti1}
\{X^{(n),N,1}\}_{n,N\in\mathbb{N}}~\text{is jointly tight in}~C([0,T];\mathbb{H})\cap L^{\alpha}([0,T];\mathbb{V})
\end{equation}
  with respect to both the dimension $n$ and the number of particles $N$.
    Next, we first take the limit as $n$ tends to infinity, and show the Galerkin interacting  system $\{X^{(n),N,i}\}_{1\leq i\leq N}$ converges to the original particle system $\{X^{N,i}\}_{1\leq i\leq N}$.  In this case, we can see the original particle system $\{X^{N,i}\}_{1\leq i\leq N}$ remains tight in $C([0,T];\mathbb{H})\cap L^{\alpha}([0,T];\mathbb{V})$ based on the result (\ref{ti1}).  Then by letting $N$ tends to infinity, we can establish the convergence of the empirical laws sequence $\{\mathcal{S}^{N}\}_{N\in\mathbb{N}}$
to the law of solutions of mean field SPDE  \eref{eqSPDE01} in the following Wasserstein space
 \begin{equation*}\label{wass1}
 \mathscr{P}_{2}(C([0,T];\mathbb{H}))\cap \mathscr{P}_{\alpha}(L^{\alpha}([0,T];\mathbb{V})).
\end{equation*}
To the best of our knowledge, this is the first time to employ
the Galerkin method and utilize the joint tightness argument to prove the mean field limit of $N$-stochastic PDEs.

Now, let us compare the main results of this paper with some related works in the literature.
Note that even in finite-dimensional settings, our results still have independent interest and cannot be encompassed by existing works.  In \cite{BCC11}, Bolley et al. significantly extend McKean's classical result to non-globally Lipschitz bounded settings.
However, this comes at the price of some additional assumptions on the boundedness of exponential moments of solutions and on the linear growth of interacting kernels (cf. (1.7) in \cite{BCC11} or \cite[Theorem 3.2]{CD22a}). Recently, Erny \cite{E1} further established the well-posedness and mean field limit for mean field SDEs with jump noise and locally Lipschitz coefficients. Similarly, in \cite{E1} some analogous assumptions are also required, such as the linear growth condition on the drift coefficients and the boundedness on the diffusion coefficients. Moreover, the initial value needs to be exponentially integrable.

In the present work, we do not need to assume such exponential moment control and the linear growth assumption. Another feature of our theory is that we also weaken the locally Lipschitz condition of the coefficients by introducing two distinct local monotonicity conditions.
Based on this more general setting, our results have a broader range of applications and can be used to handle more interesting examples, including but not limited to the models (\ref{in05})-(\ref{mean03}) in the fields of data science and machine learning (see Section \ref{sec5.1} for details). However, the aforementioned models (\ref{in05})-(\ref{mean03}) can not be treated by \cite{BCC11} and \cite{E1}.

On the other hand, in contrast to the finite-dimensional settings, dealing with the mean field limit for infinite dimensional interacting systems with nonlinear kernels is a more challenging problem. Existing results mainly considered the interacting SPDEs with the linear growth and Lipschitz type conditions (cf. \cite{BKK,CKS,C1,ES1}). In this work, we refine the famous variational framework introduced by Pardoux, Krylov and Rozovskii (cf. \cite{Par75,KR,LR1,RS}) to deal with the mean field SPDEs and infinite-dimensional particle systems with coefficients of polynomial growth and with nonlinear kernels, which expands the scope of applicability for the analysis in the infinite dimensional interacting systems and the associated mean field SPDEs.

\vspace{2mm}

\vspace{1mm}
The rest of manuscript is organized as follows. In Section \ref{sec4}, we introduce the detailed framework and state the main results about the existence and uniqueness of martingale/strong solutions in Theorems \ref{th1}-\ref{th2}, respectively. The main results about the mean field limit for weakly interacting SPDEs  are provided in Theorem \ref{th4}. In Sections \ref{sec5.1} and \ref{example}, we present many concrete examples to illustrate the wide applicability of
our main results.
In Section \ref{well-posed}, we prove the existence and uniqueness of solutions to mean field SPDE \eref{eqSPDE01}. Section \ref{sec meam} is devoted to proving the main results on the mean field limit.

\vspace{1mm}
Throughout this paper $C_p$ denotes a positive constant which may changes from line to line,
where the subscript $p$ is used to emphasize that the constant depends on certain parameter.

\section{Preliminaries and main results}\label{sec4}

 \subsection{Definitions and notations}
In this subsection, we intend to introduce some functional spaces and necessary notations that are frequently used in this paper.

Let $(U,\langle\cdot,\cdot\rangle_U)$ and $({\mathbb{H}}, \langle\cdot,\cdot\rangle_{\mathbb{H}}) $ be  separable Hilbert spaces, and ${\mathbb{H}}^*$ be the dual space of ${\mathbb{H}}$. Let $({\mathbb{X}},\|\cdot\|_{\mathbb{X}})$, $({\mathbb{V}},\langle\cdot,\cdot\rangle_{\mathbb{V}})$ denote a reflexive Banach space and a separable Hilbert space, respectively, such that the embeddings
$$\mathbb{X}\subset \mathbb{V}\subset \mathbb{H}$$
are compact and dense. Identifying $\mathbb{H}$ with its dual space in view of the Riesz isomorphism, we obtain a  Gelfand triple
$$ {\mathbb{V}}\subset {\mathbb{H}}(\simeq {\mathbb{H}}^*)\subset {\mathbb{V}}^*.$$
The dualization between ${\mathbb{V}}$ and ${\mathbb{V}}^*$ is denoted by $_{{\mathbb{V}}^*}\langle\cdot,\cdot\rangle_{\mathbb{V}}$. It is straightforward that
$$_{{\mathbb{V}}^*}\langle\cdot,\cdot\rangle_{\mathbb{V}}|_{{\mathbb{H}}\times {\mathbb{V}}}=\langle\cdot,\cdot\rangle_{\mathbb{H}}.$$
Let $L_2(U,\mathbb{H}),L_2(U,\mathbb{V})$ be  the spaces of all Hilbert-Schmidt operators from $U$ to $\mathbb{H},\mathbb{V}$, respectively.
Let $\mathbb{C}_T:=C([0,T];\mathbb{H})$ be the space of all continuous functions from $[0,T]$ to $\mathbb{H}$, which is a Banach space equipped with  the uniform norm given by
$$\|u\|_{T,\mathbb{H}}:=\sup_{t\in[0,T]}\|u_t\|_{\mathbb{H}},~u\in \mathbb{C}_T.$$

For any separable Banach space $({\mathbb{B}}, \|\cdot\|_{\mathbb{B}}) $,
let $\mathscr{P}({\mathbb{B}})$ be the space of all probability measures on $\mathbb{B}$ equipped with the weak topology,  i.e., the topology of convergence in law. Furthermore, for any $p>0$, we denote
$$\mathscr{P}_p({\mathbb{B}}):=\Big\{\mu\in\mathscr{P}({\mathbb{B}}):\mu(\|\cdot\|_{{\mathbb{B}}}^p):=\int\|\xi\|_{\mathbb{B}}^p\mu(d\xi)<\infty\Big\}.$$
Then $\mathscr{P}_p({\mathbb{B}})$ is a Polish space under the following $L^p$-Wasserstein distance
$$\mathcal{W}_{p,{\mathbb{B}}}(\mu,\nu):=\inf_{\pi\in\mathscr{C}(\mu,\nu)}\Bigg(\int\|\xi-\eta\|_{\mathbb{B}}^p\pi(d\xi,d\eta)\Bigg)^{\frac{1}{p\vee1}},~\mu,\nu\in\mathscr{P}_p({\mathbb{B}}),$$
where $\mathscr{C}(\mu,\nu)$ stands for the set of all couplings for  $\mu$ and $\nu$,  i.e., the set of all Borel probability measures $\pi$ on $\mathbb{B}\times \mathbb{B}$ such that
$\pi(\cdot\times \mathbb{B})=\mu(\cdot)~\text{and}~\pi(\mathbb{B}\times \cdot)=\nu(\cdot).$

\subsection{Existence of martingale and strong solutions}\label{secin}
For the measurable maps
$$
\mathcal{A}:[0,T]\times {\mathbb{V}}\times\mathscr{P}(\mathbb{V})\rightarrow {\mathbb{V}}^*,~~\mathcal{B}:[0,T]\times {\mathbb{V}}\times\mathscr{P}(\mathbb{V})\rightarrow L_2(U,\mathbb{H}),
$$
we are interested in the following general mean field SPDE,
\begin{equation}\label{eqSPDE}
dX_t=\mathcal{A}(t,X_t,\mathscr{L}_{X_t})dt+\mathcal{B}(t,X_t,\mathscr{L}_{X_t})dW_t,
\end{equation}
where $\{W_t\}_{t\in [0,T]}$ is an $U$-valued cylindrical Wiener process defined on a complete filtered probability space $\left(\Omega,\mathscr{F},\{\mathscr{F}_t\}_{t\in [0,T]},\mathbb{P}\right)$.

Denote $\Omega:=\mathbb{C}_T$. We will use $w$ to denote
a path in $\Omega$, and $\pi_t(w):=w_t$ to denote the coordinate process.
  Define the $\sigma$-algebra
$$\mathscr{F}_t:=\sigma\big\{\pi_s:~s\leq t\big\}.$$

We  recall the definition of the martingale problem associated with mean field SPDE (\ref{eqSPDE}).

\begin{definition} \label{de3}$($Martingale solution$)$ Let $\mu_0\in\mathscr{P}(\mathbb{H})$. A probability measure $\mathbf{P}\in\mathscr{P}(\Omega)$ is  called a martingale solution of mean field SPDE (\ref{eqSPDE}) with initial law $\mu_0$ if

$(M1)$~~$\mathbf{P}\circ \pi_0^{-1}=\mu_0$ and
$$\mathbf{P}\Bigg(w\in \Omega:\int_0^T\|\mathcal{A}(s,w_s,\mu_s)\|_{\mathbb{V}^*}ds+\int_0^T\|\mathcal{B}(s,w_s,\mu_s)\|_{L_2(U;{\mathbb{H}})}^2ds<\infty\Bigg)=1,$$
where $\mu_s:=\mathbf{P}\circ \pi_s^{-1}$.

\vspace{1mm}
$(M2)$~~for each $l\in\mathbb{V}$, the process
$$\mathcal{M}_l(t,w,\mu):={}_{\mathbb{V}^*}\langle w_t,l\rangle_{\mathbb{V}}-{}_{\mathbb{V}^*}\langle w_0,l\rangle_{\mathbb{V}}-\int_0^t{}_{\mathbb{V}^*}\langle \mathcal{A}(s,w_s,\mu_s),l\rangle_{\mathbb{V}}ds,~t\in[0,T],$$
is a continuous square integrable $(\mathscr{F}_t)$-martingale under $\mathbf{P}$, 
whose quadratic variation process is given by
$$\langle \mathcal{M}_l\rangle(t,w,\mu)=\int_0^t\|\mathcal{B}(s,w_s,\mu_s)^*l\|_U^2ds,~t\in[0,T].$$
We shall use $\mathscr{M}_{\mu_0}^{\mathcal{A},\mathcal{B}}$ to denote the set
of all martingale solutions of mean field SPDE (\ref{eqSPDE}) associated with $\mathcal{A},\mathcal{B}$ and initial law $\mu_0$.
\end{definition}

In this part, we suppose that there are some constants $\alpha>1$, $\gamma\geq \alpha$, $\lambda\geq1$,  $\beta\geq 0$ and $C,\delta_1,\delta_2>0$ such that the following conditions hold for all $t\in[0,T]$.

\begin{enumerate}
\item [$(\mathbf{A_0})$] There exists an orthogonal set $\{e_1,e_2,\ldots\}\subset \mathbb{X}$ in $\mathbb{V}$ such that it constitute an orthonormal
basis of $\mathbb{H}$.

\vspace{1mm}
\item [$(\mathbf{A_1})$]
 $($Demicontinuity$)$ For any $v\in {\mathbb{V}}$, the maps
\begin{eqnarray*}
&&{\mathbb{V}}\times\mathbb{M}_1\ni(u,\mu)\mapsto~_{{\mathbb{V}}^*}\langle \mathcal{A}(t,u,\mu),v\rangle_{\mathbb{V}},
\nonumber \\
&&{\mathbb{V}}\times\mathbb{M}_1\ni(u,\mu)\mapsto\| \mathcal{B}(t,u,\mu)^*v\|_U
\end{eqnarray*}
are continuous, where $\mathbb{M}_1:=\mathscr{P}_2(\mathbb{H})\cap \mathscr{P}_{\alpha}(\mathbb{V}).$

\vspace{1mm}
\item [$(\mathbf{A_2})$]
 $($Coercivity$)$ For any $u\in \mathbb{V}$, $\mu\in\mathbb{M}_1$,
\begin{eqnarray*}
2_{\mathbb{V}^*}\langle \mathcal{A}(t,u,\mu),u\rangle_{\mathbb{V}}+\|\mathcal{B}(t,u,\mu)\|_{L_2(U,\mathbb{H})}^2+\delta_1\|u\|_{\mathbb{V}}^\alpha\leq C(1+\|u\|_{\mathbb{H}}^2+\mu(\|\cdot\|_{\mathbb{H}}^2)).
\end{eqnarray*}

\vspace{1mm}
\item [$(\mathbf{A_3})$]
 $($Growth$)$ For any $u\in \mathbb{V}$, $\mu\in\mathbb{M}_2:=\mathscr{P}_{\beta}(\mathbb{H})\cap \mathscr{P}_{\alpha}(\mathbb{V})$,
\begin{equation}\label{conA3}
\|\mathcal{A}(t,u,\mu)\|_{{\mathbb{V}}^*}^{\frac{\alpha}{\alpha-1}}\leq C(1+\|u\|_{\mathbb{V}}^{\alpha}+\mu(\|\cdot\|_{\mathbb{V}}^{\alpha} ))(1+\|u\|_{{\mathbb{H}}}^{\beta}+\mu(\|\cdot\|_{\mathbb{H}}^{\beta})),
\end{equation}
and for any $u\in \mathbb{V}$, $\mu\in\mathbb{M}_1$,
\begin{equation}\label{conb}
\|\mathcal{B}(t,u,\mu)\|_{L_2({U},{\mathbb{H}})}^2\leq C(1+\|u\|_{\mathbb{H}}^2+\mu(\|\cdot\|_{\mathbb{H}}^2)).
\end{equation}

\vspace{1mm}
\item [$(\mathbf{A_4})$]
Let $\mathbb{H}_n:=\text{span}\{e_1,\ldots,e_n\}$. For any $n\in\mathbb{N}$, the operators $\mathcal{A},\mathcal{B}$
\begin{eqnarray*}
&&[0,T]\times\mathbb{H}_n\times\mathscr{P}(\mathbb{H}_n)\ni (t,u,\mu)\mapsto \mathcal{A}(t,u,\mu)\in \mathbb{V},
\nonumber \\
&&
[0,T]\times\mathbb{H}_n\times\mathscr{P}(\mathbb{H}_n)\ni (t,u,\mu)\mapsto \mathcal{B}(t,u,\mu)\in L_2({U},{\mathbb{V}})
\end{eqnarray*}
 satisfy that there exists $C>0$ (independent of $n$) such that for any $u\in\mathbb{H}_n,\mu\in\mathscr{P}_{(\alpha\vee \lambda)}(\mathbb{H}_n)$,
\begin{eqnarray*}
&&2\langle \mathcal{A}(t,u,\mu),u\rangle_{\mathbb{V}}+\delta_2\|u\|_{\mathbb{X}}^{\gamma}
\nonumber\\
&&\leq C\|u\|_{\mathbb{V}}^{\alpha+2}\|u\|_{\mathbb{H}}^{\lambda} + C(1+\|u\|_{\mathbb{V}}^{\alpha}+\mu(\|\cdot\|_{\mathbb{V}}^{\alpha}))(1+\|u\|_{\mathbb{V}}^{2}+\|u\|_{\mathbb{H}}^{\lambda}+\mu(\|\cdot\|_{\mathbb{H}}^{\lambda}))
\end{eqnarray*}
and
\begin{eqnarray*}
\|\mathcal{B}(t,u,\mu)\|_{L_2({U},{\mathbb{V}})}^2
\leq C(1+\|u\|_{\mathbb{V}}^{2}+\mu(\|\cdot\|_{\mathbb{V}}^{\alpha}))(1+\|u\|_{\mathbb{H}}^{\lambda}+\mu(\|\cdot\|_{\mathbb{H}}^{\lambda})).
\end{eqnarray*}
\end{enumerate}
\begin{remark}
We would like to comment on  the non-homogeneous character of assumptions $(\mathbf{A_2})$-$(\mathbf{A_4})$.

Considering the case of finite-dimensional stochastic differential equations (in this case, both $\|\cdot\|_{\mathbb{H}}$ and $\|\cdot\|_{\mathbb{V}}$ are equal to the Euclidean norm $|\cdot|$), since the constant $\alpha>1$, if $\mathcal{A}$ is linear in $u$ (e.g.~$\mathcal{A}(t,u,\mu)=C_0u$) then we can take $ \alpha=2$ in assumptions $(\mathbf{A_2})$-$(\mathbf{A_4})$, e.g. in  assumption $(\mathbf{A_2})$
$$2\langle C_0 u,u\rangle+\delta_1|u|^{2}=(2C_0+\delta_1)|u|^2\leq C(1+|u|^2).$$

However, in the case of infinite dimensions, we consider the embedding structure $\mathbb{V}\subset \mathbb{H}\subset \mathbb{V}^*$ compactly and densely. In this case, we can consider  $\mathcal{A}(u)=\Delta u+C_0u$ rather than only  $\mathcal{A}(u)=C_0u$. More precisely,  we take $\mathbb{V}=W_0^{1,2}$, $\mathbb{H}=L^2$, $\alpha=2$, $\delta_1=2$ in  assumption  $(\mathbf{A_2})$ to get
$$2{}_{\mathbb{V}^*}\langle \Delta u+C_0 u,u\rangle_{\mathbb{V}}+\delta_1\|u\|_{\mathbb{V}}^{2}=-2\|u\|_{\mathbb{V}}^{2}+2C_0\|u\|_{\mathbb{H}}^{2}+2\|u\|_{\mathbb{V}}^{2}\leq C(1+\|u\|_{\mathbb{H}}^{2}).$$
\end{remark}

We state the first main result concerning the  existence of martingale solutions to  mean field SPDE (\ref{eqSPDE}).
\begin{theorem}\label{th1}
Suppose that $\mathbf{(A_0)}$-$\mathbf{(A_4)}$ hold.
For any initial data $\xi\sim\mu_0\in\mathscr{P}_2(\mathbb{V})\cap \mathscr{P}_p(\mathbb{H})$ with $p\in[\vartheta,\infty)\cap (2,\infty)$, where
$\vartheta:= \max\{\beta+2, \lambda\}$,
then (\ref{eqSPDE}) has a martingale solution $\mathbf{P}$ in the sense of Definition \ref{de3}. Moreover, we have
\begin{equation}\label{esq370}
\mathbb{E}^{\mathbf{P}}\Big[\sup_{t\in[0,T]}\|w_t\|_{\mathbb{H}}^p\Big]+\mathbb{E}^{\mathbf{P}}\int_0^T(1+\|w_t\|_{\mathbb{H}}^{p-2})\|w_t\|_{\mathbb{V}}^{\alpha}dt<\infty.
\end{equation}
\end{theorem}

%

\begin{remark}
We mention that  the current setting, particularly the condition $\mathbf{(A_4)}$,  is also of independent interest in the investigations of classical SPDEs. Recently,  the authors in \cite{RSZ1} employed the technique of  pseudo-monotone operators to prove the existence of weak solutions to SPDE (\ref{eqSPDE}) without distribution dependence, where the pseudo-monotonicity of drift was required. In the present work, instead of the pseudo-monotonicity, we  utilize the condition $\mathbf{(A_4)}$  and take a different approach  to establish the existence of weak solutions.
\end{remark}

Next, we want to show the existence of strong solutions to  mean field SPDE (\ref{eqSPDE}). To this end, let us review the following definitions.

\begin{definition}\label{dew} $($Weak solution$)$ A pair $(X,W)$ is called a (probabilistically) weak solution to mean field SPDE (\ref{eqSPDE}), if there exists a stochastic basis $(\Omega,\mathscr{F},\{\mathscr{F}_t\}_{t\in[0,T]},\mathbb{P})$ such that $X$ is an $\{\mathscr{F}_t\}$-adapted process and  $W$ is an $U$-valued cylindrical Wiener process on $(\Omega,\mathscr{F},\{\mathscr{F}_t\}_{t\in[0,T]},\mathbb{P})$ and the following holds:

\vspace{2mm}
(i) $X\in \mathbb{C}_T\cap L^\alpha([0,T];\mathbb{V})$, $\mathbb{P}$-a.s.;

\vspace{2mm}
(ii) $\int_0^T\|\mathcal{A}(s,X_s,\mathscr{L}_{X_s})\|_{\mathbb{V}^*}ds+\int_0^T\|\mathcal{B}(s,X_s,\mathscr{L}_{X_s})\|_{L_2(U;\mathbb{H})}^2ds<\infty$, $\mathbb{P}$-a.s.;

\vspace{2mm}
(iii)
$X_t=X_0+\int_0^t \mathcal{A}(s,X_s,\mathscr{L}_{X_s})ds+\int_0^t \mathcal{B}(s,X_s,\mathscr{L}_{X_s})dW_s,t\in[0,T],\mathbb{P}\text{-a.s.}$
holds in ${\mathbb{V}}^*$.
\end{definition}

\begin{definition}\label{de1}
(Strong solution) We say that there exists a (probabilistically) strong solution to (\ref{eqSPDE}) if for every probability space $(\Omega,\mathscr{F},\{\mathscr{F}_t\}_{t\in[0,T]},\mathbb{P})$ with an $U$-valued cylindrical Wiener process $W$, there exists
an $\{\mathscr{F}_t\}$-adapted process $X$ such that properties (i)-(iii) in Definition \ref{dew} hold.
\end{definition}

We assume the following type of local monotonicity condition to guarantee the existence of strong solutions to  mean field SPDE (\ref{eqSPDE}):

\vspace{1mm}
\begin{enumerate}
\item [$(\mathbf{A_5})$]\label{H2}
  For any $u,v\in {\mathbb{V}}$, $\mu\in\mathbb{M}_2$,
\begin{eqnarray*}
&&2_{{\mathbb{V}}^*}\langle \mathcal{A}(t,u,\mu)-\mathcal{A}(t,v,\mu),u-v\rangle_{\mathbb{V}}+\|\mathcal{B}(t,u,\mu)-\mathcal{B}(t,v,\mu)\|_{L_2(U,{\mathbb{H}})}^2
\nonumber\\
&&\leq
(C+\rho(u,\mu)+\eta(v,\mu))\|u-v\|_{\mathbb{H}}^2,
\end{eqnarray*}
where $\rho,\eta:{\mathbb{V}}\times \mathbb{M}_2\to [0,\infty)$ are  measurable functions satisfying
\begin{equation}\label{esq22}
\rho(u,\mu)+\eta(u,\mu)\leq C(1+\|u\|_{\mathbb{V}}^{\alpha}+\mu(\|\cdot\|_{\mathbb{V}}^{\alpha}))(1+\|u\|_{\mathbb{H}}^{\beta}+\mu(\|\cdot\|_{\mathbb{H}}^{\beta})).
\end{equation}
\end{enumerate}

\vspace{1mm}
Now, we state the second main result concerning the  existence of strong solutions to  mean field SPDE (\ref{eqSPDE}).
\begin{theorem}\label{th3}
Suppose that $\mathbf{(A_0)}$-$\mathbf{(A_5)}$ hold.
For any initial data $\xi\in L^p(\Omega,\mathscr{F}_0,\mathbb{P};{\mathbb{H}})\cap L^2(\Omega,\mathscr{F}_0,\mathbb{P};{\mathbb{V}})$ with $p\in[\vartheta,\infty)\cap (2,\infty)$,
then (\ref{eqSPDE}) has a strong solution $X$ in the sense of Definition \ref{de1}. Moreover, we have the estimates
\begin{equation}\label{es45}
\mathbb{E}\Big[\sup_{t\in[0,T]}\|X_t\|_{\mathbb{H}}^p\Big]+\mathbb{E}\int_0^T(1+\|X_t\|_{\mathbb{H}}^{p-2})\|X_t\|_{\mathbb{V}}^{\alpha}dt<\infty.
\end{equation}
\end{theorem}

\subsection{Uniqueness of  solutions}
In this subsection, our objective is to demonstrate the  uniqueness  of both strong and weak solutions to mean field SPDE (\ref{eqSPDE}) under different  assumptions. To this end, we first introduce several new definitions and notations.

Define
\begin{equation}\label{es73}
 \mathbb{S}:=\mathbb{C}_T\cap L^{\alpha}([0,T];\mathbb{V}),
\end{equation}
where $\alpha$ is the same as in $(\mathbf{A_2})$, which is a Polish space under the metric
$$d_T(u,v):=d_T^1(u,v)+d_T^2(u,v):=\sup_{t\in[0,T]}\|u_t-v_t\|_{\mathbb{H}}+\Big(\int_0^T\|u_t-v_t\|_{\mathbb{V}}^{\alpha} dt\Big)^{\frac{1}{\alpha}},~u,v\in\mathbb{S}.$$
For any $t\in[0,T],~R>0,~\xi\in \mathbb{S}$, we define $\xi^t:[0,T]\to\mathbb{H}$ by
$$\xi^t_s:=\xi_{t\wedge s},~s\in[0,T],$$
and map $\pi^t(\xi):=\xi^t$. Then the marginal distribution before time $t$ of a probability measure $\mu\in\mathscr{P}(\mathbb{S})$ is denoted by
$$\mu^t:=\mu\circ (\pi^t)^{-1}.$$
Define a time
\begin{equation*}
\tau_R^\xi:=\inf\Big\{t\geq 0:\|\xi_t\|_{\mathbb{H}}+\int_0^t\|\xi_s\|_{\mathbb{V}}^{\alpha}ds\geq R\Big\}\wedge T.
\end{equation*}
Then for any $\mu,\nu\in\mathscr{P}(\mathbb{S})$, we can define the following ``local'' $L^2$-Wasserstein distance
\begin{equation*}
\mathcal{W}_{2,R}(\mu,\nu):=\inf_{\pi\in\mathscr{C}(\mu,\nu)}\Bigg(\int\|\xi^{\tau_R^\xi\wedge\tau_R^{\eta}}-\eta^{\tau_R^\xi\wedge\tau_R^{\eta}}\|_{T,\mathbb{H}}^2\pi(d\xi,d\eta)\Bigg)^{\frac{1}{2}}.
\end{equation*}

In what follows, we propose  two types of local monotonicity conditions,  which are essential for ensuring the uniqueness of solutions to (\ref{eqSPDE}).

\vspace{2mm}
\begin{enumerate}
\item [$(\mathbf{A'_5})$] $($Local Monotonicity I$)$
 There exists $C>0$ such that for any $t\in[0,T]$, $u,v\in {\mathbb{V}}$, $\mu,\nu\in\mathbb{M}_2$,
\begin{eqnarray*}
&&2_{{\mathbb{V}}^*}\langle \mathcal{A}(t,u,\mu)-\mathcal{A}(t,v,\nu),u-v\rangle_{\mathbb{V}}+\|\mathcal{B}(t,u,\mu)-\mathcal{B}(t,v,\nu)\|_{L_2(U,{\mathbb{H}})}^2
\nonumber\\
&&\leq
(C+\rho(0,\mu)+\eta(0,\nu))\|u-v\|_{\mathbb{H}}^2
+(C+\rho(u,\mu)+\eta(v,\nu))\mathcal{W}_{2,\mathbb{H}}(\mu,\nu)^2,
\end{eqnarray*}
where $\rho,\eta$ are the same as in  $(\mathbf{A_5})$.

\vspace{3mm}
\item [$(\mathbf{A''_5})$] $($Local Monotonicity II$)$
 There exists $R_0,C>0$ such that for any $R>R_0$,  $u,v\in \mathbb{V}$, $\mu,\nu\in\mathscr{P}_{\beta}(\mathbb{C}_T)\cap \mathscr{P}_{\alpha}(L^{\alpha}([0,T];\mathbb{V}))$ and $t\in[0,T]$,
\begin{eqnarray*}
&&_{{\mathbb{V}}^*}\langle \mathcal{A}(t,u,\mu_t)-\mathcal{A}(t,v,\nu_t),u-v\rangle_{\mathbb{V}}
\nonumber\\
&&\leq
(C+\rho(u,\mu_t)+\eta(v,\nu_t))(\|u-v\|_{\mathbb{H}}^2+\mathcal{W}_{2,R}(\mu^t,\nu^t)^2)
\end{eqnarray*}
and
\begin{eqnarray*}
&&\|\mathcal{B}(t,u,\mu_t)-\mathcal{B}(t,v,\nu_t)\|_{L_2(U,{\mathbb{H}})}^2
\nonumber\\
&&\leq
(C+\rho(u,\mu_t)+\eta(v,\nu_t))(\|u-v\|_{\mathbb{H}}^2+\mathcal{W}_{2,R}(\mu^t,\nu^t)^2),
\end{eqnarray*}
where $\rho,\eta$ are the same as in  $(\mathbf{A_5})$.

\end{enumerate}

\begin{definition}
We say that the pathwise uniqueness holds if whenever the stochastic basis $(\Omega,\mathscr{F},\{\mathscr{F}_t\}_{t\in[0,T]},\mathbb{P})$ and the $U$-valued cylindrical Wiener process $W$ are fixed, two solutions $X$, $\tilde{X}$ in the sense of Definition \ref{de1} such that $X_0=\tilde{X}_0$ $\mathbb{P}$-a.s., then $\mathbb{P}$-a.s.
$$X_t=\tilde{X}_t,~t\in[0,T].$$
\end{definition}

We present the pathwise uniqueness result to  mean field SPDE (\ref{eqSPDE}).
\begin{theorem}\label{th2}
Suppose that one of conditions $(\mathbf{A'_5})$-$(\mathbf{A''_5})$ holds.
For any initial data $\xi\in L^p(\Omega,\mathscr{F}_0,\mathbb{P};{\mathbb{H}})\cap L^2(\Omega,\mathscr{F}_0,\mathbb{P};{\mathbb{V}})$ with $p\in[\vartheta,\infty)\cap (2,\infty)$,
then (\ref{eqSPDE}) admits pathwise uniqueness provided satisfying (\ref{es45}).
\end{theorem}

\begin{remark}\label{re1}
(i) It is worth pointing out that instead of $(\mathbf{A'_5})$, a more natural condition for the (pathwise) uniqueness of solutions to (\ref{eqSPDE}) should be
\begin{eqnarray}\label{unip}
&&2_{{\mathbb{V}}^*}\langle \mathcal{A}(t,u,\mu)-\mathcal{A}(t,v,\nu),u-v\rangle_{\mathbb{V}}+\|\mathcal{B}(t,u,\mu)-\mathcal{B}(t,v,\nu)\|_{L_2(U,{\mathbb{H}})}^2
\nonumber\\
&&\leq
(C+\rho(u,\mu)+\eta(v,\nu))(\|u-v\|_{\mathbb{H}}^2+\mathcal{W}_{2,{\mathbb{H}}}(\mu,\nu)^2).
\end{eqnarray}
However, it is important to note that there exist some counterexamples (cf.~\cite{SCHEUTZOW}) in the case of mean field SDE, e.g.,
\begin{equation}\label{eqcoun1}
\frac{dX_t}{dt}=\mathcal{A}(X_t)+\mathbb{E}X_t,~X_0=\xi,
\end{equation}
and
\begin{equation}\label{eqcoun2}
dX_t=\mathbb{E}\mathcal{A}(X_t)dt+dW_t,~X_0=\xi.
\end{equation}
These counterexamples demonstrate  that if the drift $\mathcal{A}(\cdot)$ is merely locally Lipschitz continuous, the pathwise or distribution uniqueness of solutions to (\ref{eqcoun1}) or (\ref{eqcoun2}) does not hold in general.

Hence, it is both reasonable and necessary to assume a stronger structure, such as the condition $(\mathbf{A'_5})$, on the coefficients in order to obtain the uniqueness of solutions.

\vspace{2mm}
(ii) While condition (\ref{unip}) may not be sufficient to ensure the uniqueness of solutions to (\ref{eqSPDE}), it is possible to construct an  alternative metric between laws, such as a ``local'' Wasserstein distance. By using this modified metric  and introducing a similar condition, i.e. $(\mathbf{A''_5})$, we can also obtain the uniqueness of solutions.
\end{remark}

\begin{remark}
In this remark, we present some examples to compare the assumptions $(\mathbf{A_5})$, $(\mathbf{A_5'})$, $(\mathbf{A_5''})$, and $(\mathbf{A_5^*})$ below (i.e.~(\ref{unip})).

\textbf{Example 1.}  Consider the following mean field equation
\begin{equation}\label{eqa1}
\frac{dX_t}{dt}=b(X_t)+\mathbb{E}X_t,~X_0=\xi,
\end{equation}
where the mapping $b:\mathbb{R}\to\mathbb{R}$ satisfies that there exists a constant $\kappa>0$ such that for any $u,v\in\mathbb{R}$,
\begin{equation}\label{local1}
|b(u)-b(v)|\leq C(1+|u|^{\kappa}+|v|^{\kappa})|u-v|^2.
\end{equation}

It is clear that  $(\mathbf{A_5^*})$ holds for Eq.~(\ref{eqa1}) but not $(\mathbf{A_5'})$.
 In fact, Eq.~(\ref{eqa1}) is a counterexample that the uniqueness of solutions does not hold as stated in Remark \ref{re1}.

 Moreover, we can see that Eq.~(\ref{eqa1}) satisfies $(\mathbf{A_5})$ but not $(\mathbf{A_5''})$.

\vspace{2mm}
\textbf{Example 2.} For the example such that $(\mathbf{A_5})$ holds but not $(\mathbf{A_5^*})$, we set
a function $b:\mathbb{R}\to \mathbb{R}$ satisfying the global monotonicity condition, i.e., there exists a constant $C>0$ such that for any $u,v\in\mathbb{R}$,
\begin{equation}\label{esa4}
\langle b(u)-b(v),u-v\rangle \leq C|u-v|^2.
\end{equation}

Then we consider  the following mean field equation
\begin{equation}\label{eqa2}
dX_t=\tilde{b}(X_t,\mathscr{L}_{X_t})dt+dW_t,~X_0=\xi,
\end{equation}
where the measure dependent function $\tilde{b}:\mathbb{R}\times\mathscr{P}_2(\mathbb{R})\to \mathbb{R}$ is given by
\begin{equation}\label{esa5}
\tilde{b}(u,\mu):=\int_{\mathbb{R}}b(u-y)\mu(dy).
\end{equation}
By (\ref{esa4}), we can see that (\ref{esa5}) satisfies $(\mathbf{A_5})$, namely,
 \begin{eqnarray*}
&&\!\!\!\!\!\!\!\!\langle\tilde{b}(u,\mu)-\tilde{b}(v,\mu),u-v\rangle
\nonumber \\
=&&\!\!\!\!\!\!\!\!\int_{\mathbb{R}}\langle b(u-y)-b(v-y),u-y-(v-y)\rangle\mu(dy)
\nonumber \\
\leq&&\!\!\!\!\!\!\!\!C\int_{\mathbb{R}}|u-v|^2\mu(dy)=C|u-v|^2.
\end{eqnarray*}
However, (\ref{esa5}) does not satisfy $(\mathbf{A_5^*})$.

\vspace{2mm}
Below we also present an example to illustrate the applicability of  $(\mathbf{A_5''})$, which can be viewed as a modification of the counterexample (\ref{eqcoun1}).

\vspace{2mm}
\textbf{Example 3.}  Fix $T_0\in(0,T)$. Consider the following mean field equation
\begin{equation*}
\frac{dX_t}{dt}=b(X_t)+\mathbb{E}\big[\mathbf{1}_{[0,T_0]}(t)X_t\big],~X_0=\xi,
\end{equation*}
where the mapping $b:\mathbb{R}\to\mathbb{R}$ satisfies (\ref{local1}) and is of linear growth. Let us denote
$$\mathcal{A}(t,\mu):=\int_{\mathbb{R}}\mathbf{1}_{[0,T_0]}(t)x\mu(dx).$$
We shall show that $\mathcal{A}(t,\mu)$ satisfies the assumption $(\mathbf{A_5''})$. First, let $\xi,\eta\in C([0,T];\mathbb{R})~\mathbb{P}\text{-a.s.}$. In view of the definition of $\tau_R^\xi,\tau_R^{\eta}$, it is possible to find an $R_0>0$ such that for any $R>R_0$, we have
\begin{equation}\label{stopping1}
\tau_R^{\xi,\eta}:=\tau_R^\xi\wedge\tau_R^{\eta} >T_0~~~~~\mathbb{P}\text{-a.s.}.
\end{equation}
Then we can get
 \begin{eqnarray}\label{stopping2}
\!\!\!\!\!\!\!\!&&\mathbf{1}_{[0,T_0]}(t)\xi_t-\mathbf{1}_{[0,T_0]}(t)\eta_t
\nonumber \\
\!\!\!\!\!\!\!\!&&
=\left\{
 \begin{aligned}
     &\mathbf{1}_{[0,T_0]}(t)(\xi_{t\wedge\tau_R^{\xi,\eta}}-\eta_{t\wedge\tau_R^{\xi,\eta}}),~t\in[0,\tau_R^{\xi,\eta}], \\
     &\mathbf{1}_{[0,T_0]}(t)(\xi_{t}-\eta_{t}),~t>\tau_R^{\xi,\eta}>T_0, \\
     &\mathbf{1}_{[0,T_0]}(t)(\xi_{t}-\eta_{t}),~t>\tau_R^{\xi,\eta},\tau_R^{\xi,\eta}\leq T_0, \\
  \end{aligned}
\right.
\nonumber \\
\!\!\!\!\!\!\!\!&&
\leq\left\{
 \begin{aligned}
     &\sup_{s\in[0,t\wedge\tau_R^{\xi,\eta}]}|\xi_{s}-\eta_{s}|,~t\in[0,\tau_R^{\xi,\eta}], \\
     &\mathbf{1}_{[0,T_0]}(t)\mathbf{1}_{(T_0,T]}(t)(\xi_{t}-\eta_{t})=0,~t>\tau_R^{\xi,\eta}>T_0, \\
     &\mathbf{1}_{[0,T_0]}(t)(\xi_{t}-\eta_{t}),~t>\tau_R^{\xi,\eta},\tau_R^{\xi,\eta}\leq T_0. \\
  \end{aligned}
\right.
\end{eqnarray}
In light of (\ref{stopping1}), taking expectation on both sides of (\ref{stopping2}), we derive
$$\big|\mathbb{E}\big[\mathbf{1}_{[0,T_0]}(t)\xi_t\big]-\mathbb{E}\big[\mathbf{1}_{[0,T_0]}(t)\eta_t\big]\big|\leq \mathbb{E}\Big[\sup_{s\in[0,t\wedge\tau_R^{\xi,\eta}]}|\xi_{s}-\eta_{s}|\Big]=\mathbb{E}\Big[\sup_{s\in[0,t]}\big|\xi_{s}^{\tau_R^{\xi,\eta}}-\eta_{s}^{\tau_R^{\xi,\eta}}\big|\Big],$$
which implies that $\mathcal{A}(t,\mu)$ satisfies the assumption $(\mathbf{A_5''})$.
\end{remark}

As a consequence of Theorems \ref{th3}-\ref{th2}, we eventually derive the well-posedness of mean field SPDE (\ref{eqSPDE}).
\begin{theorem}\label{th8}
Suppose that $\mathbf{(A_0)}$-$\mathbf{(A_4)}$ and one of  $(\mathbf{A'_5})$-$(\mathbf{A''_5})$ hold. For any initial data $\xi\in L^p(\Omega,\mathscr{F}_0,\mathbb{P};{\mathbb{H}})\cap L^2(\Omega,\mathscr{F}_0,\mathbb{P};{\mathbb{V}})$ with $p\in[\vartheta,\infty)\cap (2,\infty)$, then (\ref{eqSPDE}) has a unique strong and weak solution such that (\ref{es45}) holds.
\end{theorem}

\begin{proof}
Combining Theorems \ref{th3}-\ref{th2} and the modified Yamada-Watanabe theorem (see Lemma \ref{lem112}), we conclude that this theorem holds.
\end{proof}


\begin{remark}
In the existing works for the finite-dimensional mean field SDEs (e.g.~\cite{BCC11,E1,GHM}), the authors utilized the Picard iteration method to construct an approximating sequence of
mean field SDEs and obtain the corresponding  limit as a solution of  mean field SDEs, which essentially needs the  exponential moment control in their proof of the existence of solutions. Similarly, in the proof of the uniqueness of solutions, the  exponential moment control is also required under their local Lipschitz assumptions.

Differ from previous works, we first employ the martingale approach and the stopping time technique to prove the existence of martingale solutions to mean field SPDEs, which avoids the use of exponential moment control. On the other hand, we present two distinct local monotonicity conditions (i.e.~$(\mathbf{A'_5})$ and $(\mathbf{A''_5})$) to obtain the uniqueness, which  also  avoid the requirement of the exponential moment control.
Shortly speaking,  $(\mathbf{A'_5})$ is a general local monotonicity condition  with respect to the measure $\mu$ but requires to be monotone with respect to $u$, whereas $(\mathbf{A''_5})$ is a more general local monotonicity condition  with respect to $u$ but imposes restrictions on the measure $\mu$ based on a ``local" Wasserstein distance.
\end{remark}

\subsection{Mean field limit}
 In this subsection, we are interested in the theory of the mean field limit for a class of weakly interacting SPDEs. These SPDEs govern the dynamic of an $N$-particle/field system,  where the interactions among the particles/fields are captured by the dependence of the coefficients in SPDEs on the empirical laws of the particles/fields.

\vspace{1mm}
Specifically, we consider an $N$-interacting system $(X^{N,1},\ldots,X^{N,N})$ given by
 \begin{equation}\label{eqi}
dX^{N,i}_t=\mathcal{A}(t,X^{N,i}_t,\mathcal{S}^N_t)dt+\mathcal{B}(t,X^{N,i}_t,\mathcal{S}^N_t)dW_t^i,~X^{N,i}_0=\xi^i,
\end{equation}
where  $W_t^1,\ldots,W_t^N$ are independent $U$-valued cylindrical Wiener processes defined on a complete filtered probability space $\left(\Omega,\mathscr{F},\{\mathscr{F}_t\}_{t\in[0,T]},\mathbb{P}\right)$, and
$$\mathcal{S}^N_t:=\frac{1}{N}\sum_{j=1}^N\delta_{X^{N,j}_t}$$
is the empirical law of $(X^{N,1},\ldots,X^{N,N})$.

\vspace{1mm}
In order to ensure the uniqueness of solutions to $N$-interacting system (\ref{eqi}), it is sufficient to use the condition (\ref{unip}), i.e.,

\vspace{1mm}
\begin{enumerate}
\item [$(\mathbf{A_5^*})$]
 For any $u,v\in {\mathbb{V}}$, $\mu,\nu\in\mathbb{M}_2$,
\begin{eqnarray*}
&&2_{{\mathbb{V}}^*}\langle \mathcal{A}(t,u,\mu)-\mathcal{A}(t,v,\nu),u-v\rangle_{\mathbb{V}}+\|\mathcal{B}(t,u,\mu)-\mathcal{B}(t,v,\nu)\|_{L_2(U,{\mathbb{H}})}^2
\nonumber\\
&&\leq
(C+\rho(u,\mu)+\eta(v,\nu))(\|u-v\|_{\mathbb{H}}^2+\mathcal{W}_{2,\mathbb{H}}(\mu,\nu)^2).
\end{eqnarray*}
\end{enumerate}

\begin{remark}
We observe that $(\mathbf{A_5^*})$ is a stronger condition compared to $(\mathbf{A_5})$. Additionally, it is worth mentioning that $(\mathbf{A_5'})$ is stronger than $(\mathbf{A_5^*})$, while $(\mathbf{A_5''})$ are not comparable to $(\mathbf{A_5^*})$.

\end{remark}

The existence and uniqueness of (probabilistically) strong solutions to the $N$-interacting system (\ref{eqi}) are given in the following.
\begin{proposition}\label{th5}
Suppose that $\mathbf{(A_0)}$-$\mathbf{(A_4)}$ and $\mathbf{(A_5^*)}$  hold and the embeddings $\mathbb{X}\subset \mathbb{V}\subset \mathbb{H}$ are compact.
For any $N\in\mathbb{N}$ and initial random variables $\xi^i\in L^p(\Omega,\mathscr{F}_0,\mathbb{P};{\mathbb{H}})\cap L^2(\Omega,\mathscr{F}_0,\mathbb{P};{\mathbb{V}})$, with $p\in[\vartheta,\infty)\cap (2,\infty)$, $i=1,\ldots,N$,
there exists a unique strong solution $X^{N}:=(X^{N,1},\ldots,X^{N,N})$ to (\ref{eqi}).
\end{proposition}
\begin{proof}
Similar to the finite-dimensional case (cf.~\cite{Lacker}), for the existence of martingale solutions to the coupled system (\ref{eqi}) it suffices to check that the coefficients of  (\ref{eqi}) satisfy  $\mathbf{(A_0)}$-$\mathbf{(A_4)}$ on $N$-product spaces. Moreover, the assumption $\mathbf{(A_5^*)}$ is sufficient to guarantee the pathwise uniqueness of solutions to  (\ref{eqi}), we omit the detailed proof since it is quite standard. Finally, by the classical Yamada-Watanabe theorem in infinite dimensions \cite{RSZ}, we deduce that the coupled system (\ref{eqi}) admits a unique strong solution $X^{N}$.
\end{proof}

For investigating the mean field limit for the weakly interacting system (\ref{eqi}), in this part, we fix an initial random vector $X_0^N:=(\xi^1,\ldots,\xi^N)$ which satisfies the following assumptions:

\begin{enumerate}

\item [$(\mathbf{A_6})$]\label{H6}
 $($Initial values$)$ For any $N\in\mathbb{N}$,
 $$\text{the law of}~X_0^N~\text{is symmetric on}~{\mathbb{V}}\times\cdots\times {\mathbb{V}} ,$$
 and
 $$\mathcal{S}^N_0:=\frac{1}{N}\sum_{i=1}^N\delta_{X_0^{N,i}}\to \mu_0~\text{in probability},$$
where  $\mu_0$ is the initial law of solutions to (\ref{eqSPDE}).

\vspace{2mm}
Moreover, for any $p\in[\vartheta,\infty)\cap (2,\infty)$,
there exists a constant $C_p>0$ such that
\begin{equation}\label{c3}
\sup_{N\in\mathbb{N}}\mathbb{E}\|X_0^{N,1}\|_{{\mathbb{H}}}^p\leq C_p.
\end{equation}
There exists a constant $C>0$ such that
\begin{equation}\label{c4}
\sup_{N\in\mathbb{N}}\mathbb{E}\|X_0^{N,1}\|_{{\mathbb{V}}}^2\leq C.
\end{equation}
\end{enumerate}

\begin{remark}\label{remark1}
Let us mention that due to the symmetry in the law of the initial random vector $X_0^N$ (for any fixed $N\geq 2$), if   solutions of the interacting system (\ref{eqi}) are unique in the sense of  laws, then it follows that the law of $X^N:=(X^{N,1},\ldots,X^{N,N})$ is also symmetric.
\end{remark}

Now, we present the main result of this subsection, which establishes the mean field limit for the weakly interacting SPDE system (\ref{eqi}).
\begin{theorem}\label{th4}
Suppose that $\mathbf{(A_1)}$-$\mathbf{(A_4)}$, $\mathbf{(A_5^*)}$ and $\mathbf{(A_6)}$  hold.
The empirical law  $\{\mathcal{S}^{N},N\in\mathbb{N}\}$ is  tight in $\mathscr{P}_{2}(\mathbb{C}_T)\cap \mathscr{P}_{\alpha}(L^{\alpha}([0,T];\mathbb{V}))$, and  any accumulation point $\mathcal{S}$ is a solution  of martingale problem associated with mean field SPDE (\ref{eqSPDE}) with initial law $\mu_0\in\mathscr{P}_p(\mathbb{H})\cap \mathscr{P}_2(\mathbb{V})$.

\vspace{1mm}
Furthermore, if one of the conditions $(\mathbf{A'_5})$-$(\mathbf{A''_5})$ hold, then  there exists a unique martingale solution $\Gamma$ to the mean field SPDE (\ref{eqSPDE}) such that
\begin{equation}\label{es111}
\lim_{N\to\infty}\mathbb{E}\Big\{\mathcal{W}_{2,\mathbb{C}_T}(\mathcal{S}^{N},\Gamma)^2+\mathcal{W}_{\alpha,L^{\alpha}([0,T];\mathbb{V})}(\mathcal{S}^{N},\Gamma)^{\alpha}\Big\}=0.
\end{equation}
\end{theorem}

\begin{remark}
In contrast to the related works \cite{BCC11} and \cite{E1} in the finite-dimensional case, where the  mean field limit for interacting particle systems with locally Lipschitz coefficients but with exponential moment controls and satisfying the linear growth condition is established. In our work,  we do not need to impose the exponential moment condition and allow the interacting kernels to be of polynomial growth. This broader setting enables to apply the results to a wide range of interesting models in fields of the data science and the machine learning, which has its own interest.

\end{remark}

\begin{remark} \label{rem2.6}
The existing works \cite{BKK,CKS,C1,ES1} for the infinite-dimensional interacting systems  are needed to satisfy the linear growth and Lipschitz type conditions. Different from them, we allow for drift coefficients satisfying more general local monotonicity conditions within the generalized variational framework by leveraging the Galerkin method and the joint tightness argument, which extend the applicability to the highly nonlinear interacting systems, such as the stochastic climate models, the interacting Allen-Cahn equation, and the interacting Burgers type equations.

To  our knowledge, this is the first time to employ
the Galerkin method and utilize the joint tightness argument to prove the mean field limit of $N$-stochastic PDEs, which has its own interest.
\end{remark}
\section{Applications in finite dimensions}\label{sec5.1}
In this section, we apply our abstract setting   to the particle systems derived in the fields of data science and machine learning. These results seem also new compared to the existing works (cf.~\cite{CD22a,DNS23,E1,GHM,LW16}  and the references within). Here, we denote by $|\cdot|$, $\langle\cdot,\cdot\rangle$ the Euclidean norm and inner product, respectively.

\subsection{Particle system in  data sciences}
For $\mu\in\mathscr{P}_2(\mathbb{R})$, we first denote the variance and  mean functions as follows
$$\text{Var}[\mu]:=\int_{\mathbb{R}}(x-m[\mu])^2\mu(dx),~~m[\mu]:=\int_{\mathbb{R}}x\mu(dx).$$

The following interacting particle system (called the ensemble in this context) is derived by  the {\it{ Metropolis
 Adjusted Langevin Algorithm}} (cf.~Section 5.3.1 in \cite{CD22a}):
\begin{equation}\label{var1}
dX_t^{N,i}=-\nabla\Phi(X_t^{N,i})\text{Var}[\mathcal{S}^N_t]dt
+\sigma(X_t^{N,i},\mathcal{S}^N_t)dW_t^i,~~X^{N,i}_0=\xi_i,
\end{equation}
where $\mathcal{S}^N:=\frac{1}{N}\sum_{j=1}^N\delta_{X^{N,j}}$, the function $\Phi$ satisfies certain assumptions that are given later, $ \text{Var}[\mathcal{S}^N_t]$ depends non-linearly on all ensemble members, which is
 the empirical variance between samplers, i.e.,
 $$\text{Var}[\mathcal{S}^N_t]=\frac{1}{N}\sum_{j=1}^{N}\big(X_t^{N,j}-\bar{X}_t^N\big)^{2}.                 $$
Here, $\bar{X}^N$ denotes the sample mean, i.e.,
$\bar{X}^N:=\frac{1}{N}\sum_{i=1}^{N} X_t^{N,i}$, $\sigma:\mathbb{R}\times\mathscr{P}_2(\RR)\to \mathbb{R}$ satisfies the Lipschitz continuity condition.

\begin{remark}
For simplicity,  we only consider the case of dimension $d=1$ here. Nevertheless, we  remark that the main results (cf.~Theorem \ref{th9} below) are also applicable to the general dimension $d$, and, in this case, the empirical variance is replaced by the  empirical covariance.
\end{remark}
\begin{remark}The system \eref{var1} is related to the Ensemble Kalman Sampler, which is an algorithm introduced by Garbuno-Inigo et al.~in \cite{GHLS20} to find approximately independent and identically distributed samples from a target distribution.  The continuous version of the algorithm is a set of coupled SDEs, which form the interacting particle system \eref{var1}. The particle $X^{N,i}_t$ in \eref{var1} approximates the target distribution in the sense that the scaling limit of \eref{var1}, as $N\rightarrow\infty$, is given by a nonlinear Fokker-Planck equation.
  \end{remark}
In the following, we present the conditions on $\sigma$ and  $\Phi$.
\begin{enumerate}

\item [$({\mathbf{A_\sigma}})$]\label{sig}
$\sigma:\mathbb{R}\times\mathscr{P}_2(\RR)\to \mathbb{R}$ satisfies the Lipschitz continuity condition, i.e., there exists a constant $C>0$ such that
$$|\sigma(u,\mu)-\sigma(v,\nu)|\leq C\big(|u-v|+\mathcal{W}_{2,\mathbb{R}}(\mu,\nu)\big).$$
\end{enumerate}
\begin{enumerate}
\item [$(\mathbf{A}^1_{\Phi})$]  The mapping $u\mapsto\nabla\Phi(u)$ is continuous.

\vspace{1mm}
\item [$(\mathbf{A}^2_{\Phi})$]  For any $u\in\mathbb{R}$
$$u\nabla\Phi(u)\geq 0.$$

\vspace{1mm}
\item [$(\mathbf{A}^3_{\Phi})$]  There exist constants $C>0$ and $n\geq 1$ such that for any $u\in\mathbb{R}$
$$|\nabla\Phi(u)|\leq C(1+|u|^n).$$

\vspace{1mm}
\item [$(\mathbf{A}^4_{\Phi})$]  There exists constant $C>0$ such that for any $u,v\in\mathbb{R}$
$$(\nabla\Phi(u)-\nabla\Phi(v))(u-v)\geq -C|u-v|^2.$$

\end{enumerate}

\begin{remark}
 A typical example of $\Phi$ is
$$\Phi(u):=\frac{1}{2n}u^{2n},~n\geq 1.$$
In the existing works (cf.~e.g.~\cite{CD22a,DL21,GHLS20}), the function $\Phi$ is generally chosen as $\Phi(u):=\frac{1}{2}u^{2}$ (i.e.~$\nabla\Phi(u)$ is linear), then the particle system (\ref{var1}) is called the {\it{linear forward model}}.  As mentioned in Page 143 in \cite{CD22a}, since the variance function is a quadratic quantity in (\ref{var1}), the Lipschitz assumptions of McKean's theorem do not hold. Moreover, it is a challenging problem to deal with the case that $\nabla\Phi$ is a polynomial  function.

However, basing on our general result, we can obtain  the  mean field limit of the particle system (\ref{var1}) and the well-posedness of the associated mean field SDE (see \eref{eqas} blow) with polynomial  potential $\nabla\Phi(u)$ directly, which should be of independent interest.
\end{remark}

Let $N\to \infty$, the formal mean field limit of particle system (\ref{var1}) will be given by the following mean field SDE
\begin{equation}\label{eqas}
dX_t=-\nabla\Phi(X_t)\text{Var}[\mathscr{L}_{X_t}]dt+\sigma(X_t,\mathscr{L}_{X_t})dW_t,~~X_0\sim\mu_0.
\end{equation}

\vspace{2mm}
In this section, we suppose the following initial values conditions.
\begin{enumerate}

\item [$({\mathbf{C}})$]\label{C5}
 For any $N\in\mathbb{N}$, the law of $X_0^N$ is symmetric on $\RR\times\cdots\times {\RR}$,
 and
 $$\mathcal{S}^N_0:=\frac{1}{N}\sum_{i=1}^N\delta_{X_0^{N,i}}\to \mu_0~~\text{in}~~\mathscr{P}_2(\RR)~\text{in probability},$$
where $\mu_0$ is the initial law of solution  to (\ref{eqas}).

Moreover, for any $p\in[\vartheta,\infty)\cap (2,\infty)$,
there exists a constant $C_p>0$ such that
\begin{equation}\label{c33}
\sup_{N\in\mathbb{N}}\mathbb{E}|X_0^{N,1}|^p\leq C_p.
\end{equation}

\end{enumerate}

\begin{theorem}\label{th9}
Suppose that $(\mathbf{C})$ holds with $\vartheta:=\max\{4n,8\}$  and the assumptions $({\mathbf{A_\sigma}})$, $(\mathbf{A}^1_{\Phi})$-$(\mathbf{A}^4_{\Phi})$ hold.  Let the initial law $\mu_0\in\mathscr{P}_p(\mathbb{R})$, where $p\in(2,\infty)$. Then (\ref{eqas}) has a unique strong and weak (martingale) solution.

Furthermore, let $\Gamma$ be the unique  martingale solution to (\ref{eqas}) with initial law $\mu_0$, then
\begin{equation*}
\lim_{N\to\infty}\mathbb{E}\Big[\mathcal{W}_{2,\mathbb{C}_T}(\mathcal{S}^N,\Gamma)^2\Big]=0.
\end{equation*}

\end{theorem}

\begin{proof}
We only need to show
$$\mathcal{A}(u,\mu):=-\nabla\Phi(u)\text{Var}[\mu]$$
satisfies the conditions $\mathbf{(A_1)}$-$\mathbf{(A_4)}$ and $\mathbf{(A_5')}$.
First, it is clear that the continuity condition $\mathbf{(A_1)}$ holds. Moreover, by $(\mathbf{A}^2_{\Phi})$, $(\mathbf{A}^3_{\Phi})$ and the definition of $ \text{Var}[\cdot]$  we have
\begin{eqnarray*}
\langle \mathcal{A}(u,\mu),u\rangle
=-u\nabla\Phi(u)\text{Var}[\mu]
\leq0,
\end{eqnarray*}
and
$$|\mathcal{A}(u,\mu)|\leq C(1+|u|^{n})\mu(|\cdot|^{2})\leq C(1+|u|^{2n}+(\mu(|\cdot|^{2}))^2),$$
thus,  $\mathbf{(A_2)}$, $\mathbf{(A_4)}$ and $\mathbf{(A_3)}$ hold with $\alpha=\lambda=\gamma=2$ and $\beta=\max\{4n-2,6\}$.

Next, we show that the condition $\mathbf{(A_5')}$ is satisfied. For any coupling $\pi\in\mathscr{C}(\mu,\nu)$,
\begin{eqnarray*}
&&\langle \mathcal{A}(u,\mu)-\mathcal{A}(v,\nu),u-v\rangle
\nonumber \\
=&&\!\!\!\!\!\!\!\!-\langle(\nabla\Phi(u)-\nabla\Phi(v))\text{Var}[\mu],u-v\rangle-
\langle\nabla\Phi(u)(\text{Var}[\mu]-\text{Var}[\nu]),u-v\rangle
\nonumber \\
\leq
&&\!\!\!\!\!\!\!\!C|\nabla\Phi(u)|\cdot|u-v|\Bigg(\int_{\mathbb{R}\times \mathbb{R}}(x-m[\mu])^2-(y-m[\mu])^2\pi(dx,dy)\Bigg)
\nonumber \\
\leq
&&\!\!\!\!\!\!\!\!C(1+|u|^{n})|u-v|
\Bigg(\int_{\mathbb{R}\times \mathbb{R}}(|x|+|y|+|m[\mu]|)|x-y|\pi(dx,dy)\Bigg)
\nonumber \\
\leq&&\!\!\!\!\!\!\!\!C(1+|u|^{n})((\mu(|\cdot|^2))^{\frac{1}{2}}+(\nu(|\cdot|^2))^{\frac{1}{2}})|u-v| \Bigg(\int_{\mathbb{R}\times \mathbb{R}}|x-y|^2\pi(dx,dy)\Bigg)^{\frac{1}{2}}
\nonumber \\
\leq&&\!\!\!\!\!\!\!\!C(1+|u|^{2n}) \Bigg(\int_{\mathbb{R}\times \mathbb{R}}|x-y|^2\pi(dx,dy)\Bigg)+C(\mu(|\cdot|^2)+\nu(|\cdot|^2))|u-v|^2.\label{ese10}
\end{eqnarray*}
 Taking infimum for couplings $\pi\in\mathscr{C}(\mu,\nu)$ on both sides of above inequality, we obtain
\begin{eqnarray*}
&&\langle \mathcal{A}(u,\mu)-\mathcal{A}(v,\nu),u-v\rangle
\nonumber \\
\leq&&\!\!\!\!\!\!\!\!C(\mu(|\cdot|^{2})+\nu(|\cdot|^{2}))|u-v|^2
+ C(1+|u|^{2n})
\mathcal{W}_{2,\mathbb{R}}(\mu,\nu)^2.
\end{eqnarray*}
Thus, $\mathbf{(A_5')}$ holds. We complete the proof.
\end{proof}

\subsection{Particle system in machine learning}
Recently, there has been interest in  {\it{particle optimisation technique}} combined with {\it{(stochastic) Stein variational gradient
 descent}} methodology (cf.~e.g.~\cite{LW16,DNS23}), which relies on iterated steepest descent steps.
This construction leads to the following interacting particle system
\begin{equation}\label{mean3}
dX^{N,i}_t=-\frac{1}{N}\sum_{j=1}^N\nabla_y\kappa(X^{N,i}_t,X^{N,j}_t)dt-\frac{1}{N}\sum_{j=1}^N\kappa(X^{N,i}_t,X^{N,j}_t)\nabla \Phi(X^{N,j}_t)dt
+\sigma(X_t^{N,i},\mathcal{S}^N_t)dW_t^i,
\end{equation}
with $X^{N,i}_0=\xi_i$, $W_t^1,\ldots,W_t^N$ are independent $\RR$-valued standard Wiener processes defined on a complete filtered probability space $\left(\Omega,\mathscr{F},\{\mathscr{F}_t\}_{t\in[0,T]},\mathbb{P}\right)$.

One interest of considering the particle system \eref{mean3} is mainly motivated
by the recent works in machine learning (cf.~\cite{DNS23,LW16}), where a time-discretized form of \eref{mean3}
was introduced as an algorithm called the (stochastic) SVGD. The
idea of the algorithm is to transport a set of $N$ particles so that their empirical measure $\mathcal{S}^N$ approximates the target probability measure.

In \cite[Lemma 45]{DNS23}, the authors mentioned the following polynomial kernels
\begin{equation}\label{kers1}
\kappa(x,y):=xy,~\Phi(y):=\frac{1}{2}y^2.
\end{equation}
However, as stated in \cite[Remark 46]{DNS23}, the polynomial kernels \eref{kers1} do not satisfy the integrally strictly positive-definite assumption proposed in \cite{DNS23}, thus their main results cannot be applied directly.

In the present work,  we consider the following more general form
\begin{equation}\label{ml1}
\kappa(x,y):=x^{2k-1}y^{2m-1},~\Phi(y):=\frac{1}{2n}y^{2n},
\end{equation}
where $k,m,n\geq 1$, and $\sigma:\mathbb{R}\times\mathscr{P}_2(\RR)\to \mathbb{R}$ satisfies the Lipschitz continuity condition. In particular, by taking $k=m=n=1$ and $\sigma=0$, we can cover the model discussed in the Section 7 of \cite{DNS23}.

By applying our main results, we will show that the mean field limit of (\ref{mean3}) is given by the following mean field equation
\begin{equation}\label{eqex1}
dX_t=-\int_{\mathbb{R}}\big(\nabla_y\kappa(X_t,y)+\kappa(X_t,y)\nabla\Phi(y)\big)\mu_t(dy)dt+\sigma(X_t,\mu_t)dW_t,~~X_0\sim\mu_0,
\end{equation}
where $\mu_t:=\mathscr{L}_{X_t}$, the potentials $\kappa$ and $\Phi$ are given by (\ref{ml1}).

\begin{theorem}\label{th7}
Suppose that the assumptions $({\mathbf{A_\sigma}})$ and $(\mathbf{C})$ hold with $\vartheta:=\max\{8k-4,8(m+n)-8\}$ and the initial law $\mu_0\in\mathscr{P}_p(\mathbb{R})$, where $p\in[\vartheta,\infty)\cap (2,\infty)$. Then (\ref{eqex1}) has a unique strong and weak (martingale) solution.

Furthermore, let $\Gamma$ be the unique  martingale solution to (\ref{eqex1}) with initial law $\mu_0$, then
\begin{equation*}
\lim_{N\to\infty}\mathbb{E}\Big[\mathcal{W}_{2,\mathbb{C}_T}(\mathcal{S}^N,\Gamma)^2\Big]=0.
\end{equation*}

\end{theorem}

\begin{proof}
Let
$$\mathcal{A}_1(u,\mu):=-\int_{\mathbb{R}}\nabla_y\kappa(u,y) \mu(dy)=-(2m-1)u^{2k-1}\int_{\mathbb{R}} y^{2m-2}\mu(dy)$$
and
$$\mathcal{A}_2(u,\mu):=-\int_{\mathbb{R}}\kappa(u,y)\nabla\Phi(y) \mu(dy)=-u^{2k-1}\int_{\mathbb{R}} y^{2(m+n)-2}\mu(dy).$$
We only need to show that $\mathcal{A}_2(u,\mu)$ satisfies the conditions $\mathbf{(A_1)}$-$\mathbf{(A_4)}$ and $\mathbf{(A_5')}$, since $\mathcal{A}_1(u,\mu)$ is exactly similar.
The continuity $\mathbf{(A_1)}$ is obvious. Moreover, we can see
\begin{eqnarray*}
\langle \mathcal{A}_2(u,\mu),u\rangle
=-u^{2k}\mu(|\cdot|^{2(m+n)-2})
\leq0,
\end{eqnarray*}
and
$$|\mathcal{A}_2(u,\mu)|\leq|u|^{2k-1}\mu(|\cdot|^{2(m+n)-2})\leq C(|u|^{4k-2}+(\mu(|\cdot|^{2(m+n)-2}))^2),$$
thus,  $\mathbf{(A_2)}$, $\mathbf{(A_4)}$ and $\mathbf{(A_3)}$ hold with $\alpha=\lambda=\gamma=2$ and $\beta=\max\{8k-6,8(m+n)-10\}$.  Next we show that the condition $\mathbf{(A_5')}$ is satisfied. Specifically, for any coupling $\pi\in\mathscr{C}(\mu,\nu)$,
\begin{eqnarray*}
&&\!\!\!\!\!\!\!\!\langle \mathcal{A}_2(u,\mu)-\mathcal{A}_2(v,\nu),u-v\rangle
\nonumber \\
\leq&&\!\!\!\!\!\!\!\!-\langle\int_{\mathbb{R}}
(u^{2k-1}-v^{2k-1})y_1^{2(m+n)-2}\mu(dy_1),u-v\rangle
\nonumber \\
&&\!\!\!\!\!\!\!\!-
\langle\int_{\mathbb{R}\times \mathbb{R}}v^{2k-1}(y_1^{2(m+n)-2}-y_2^{2(m+n)-2})\pi(dy_1,dy_2),u-v\rangle
\nonumber \\
\leq
&&\!\!\!\!\!\!\!\!C|v|^{2k-1}|u-v|\Bigg(\int_{\mathbb{R}\times \mathbb{R}}(|y_1|^{2(m+n)-3}+|y_2|^{2(m+n)-3})|y_1-y_2|\pi(dy_1,dy_2)\Bigg)
\nonumber \\
\leq
&&\!\!\!\!\!\!\!\!C|v|^{2k-1}|u-v| \Big((\mu(|\cdot|^{4(m+n)-6}))^{\frac{1}{2}}+(\nu(|\cdot|^{4(m+n)-6}))^{\frac{1}{2}}\Big)
\Bigg(\int_{\mathbb{R}\times \mathbb{R}}|y_1-y_2|^2\pi(dy_1,dy_2)\Bigg)^{\frac{1}{2}}
\nonumber \\
\leq&&\!\!\!\!\!\!\!\!C\Big(\mu(|\cdot|^{4(m+n)-6})+\nu(|\cdot|^{4(m+n)-6})\Big)|u-v|^2
\nonumber \\
&&\!\!\!\!\!\!\!\!+ C|v|^{4k-2}
\Bigg(\int_{\mathbb{R}\times \mathbb{R}}|y_1-y_2|^2\pi(dy_1,dy_2)\Bigg).\label{ese10}
\end{eqnarray*}
 Taking infimum for couplings $\pi\in\mathscr{C}(\mu,\nu)$ on both sides of above inequality, we obtain
\begin{eqnarray*}
&&\!\!\!\!\!\!\!\!\langle \mathcal{A}_2(u,\mu)-\mathcal{A}_2(v,\nu),u-v\rangle
\nonumber \\
\leq&&\!\!\!\!\!\!\!\!C\Big(\mu(|\cdot|^{4(m+n)-6})+\nu(|\cdot|^{4(m+n)-6})\Big)|u-v|^2
+ C|v|^{4k-2}
\mathcal{W}_{2,\mathbb{R}}(\mu,\nu)^2.
\end{eqnarray*}
Thus, $\mathbf{(A_5')}$ holds. We complete the proof.
\end{proof}

\section{Applications in infinite dimensions}\label{example}

The main results of the paper can be applied to deal with the well-posedness and the mean field limit for a large class of weakly interacting SPDEs and the associated mean field SPDEs. Here, we illustrate our main results by the stochastic climate models, the weakly interacting Allen-Cahn equations and the weakly interacting Burgers type equations.

Note that in the infinite dimensional settings, existing framework mainly considered the interacting SPDEs under the linear growth and Lipschitz type conditions, which cannot be applied to the general nonlinear interaction models  (see e.g. \cite{BKK,CKS,C1}). By developing the variational framework, we can deal with the mean field SPDEs and infinite dimensional interacting systems with coefficients of polynomial growth and with nonlinear kernels,  including but not limited to the models presented in this section, which are much more general than previous works.

As preparations, we first recall some definitions and notations which are commonly used in the analysis of infinite dimensions.

Let $\mathcal{O}\subset\mathbb{R}^d$, $d\geq 1$, be a bounded domain with smooth boundary.
Let $(L^p(\mathcal{O},\mathbb{R}^d),\|\cdot\|_{L^p})$ be the space of all $L^p$-integrable functions on $\mathcal{O}$. For any $m\geq0$, we denote by $(W^{m,p}_0(\mathcal{O},\mathbb{R}^d),\|\cdot\|_{m,p})$ the Sobolev space with Dirichlet boundary, whose differentials belong to $L^p(\mathcal{O},\mathbb{R}^d)$ up to the order $m$. We use the following norm in $W^{m,p}_0(\mathcal{O},\mathbb{R}^d)$, i.e.,
$$ \|f\|_{m,p} := \left( \sum_{0\le |\alpha|\le m} \int_{\mathcal{O}} |D^\alpha f|^pd x \right)^\frac{1}{p}.$$
For convenience,  we use $\|\cdot\|_{m}$ to denote the norm on $W^{m,2}_0(\mathcal{O},\mathbb{R}^d)$. As usual, we also use $\mathcal{C}_{c}^{\infty}(\mathcal{O}, \mathbb{R}^{d})$ to denote the space of all infinitely differentiable $d$-dimensional vector fields with compact support on domain $\mathcal{O}$. Let $\mathcal{C}_b^1(\mathbb{R}^d;\mathbb{R}^d)$ be the space of bounded functions on $\mathbb{R}^d$ with bounded derivatives of order one.

\vspace{3mm}
We recall the standard Sobolev inequality.
\begin{lemma}\label{lem77}
Let $q>1$, $p\in[q, \infty)$ and
$$\frac{1}{q}+\frac{m}{d}=\frac{1}{p}.$$
Suppose that $f\in W_0^{m,p}(\mathcal{O},\mathbb{R}^d)$, then $f\in L^q(\mathcal{O},\mathbb{R}^d)$ and there exists $C>0$ such that the following inequality holds
$$\|f\|_{L^q}\leq C\|f\|_{m,p}.$$
\end{lemma}

The following is the Gagliardo-Nirenberg interpolation inequality (cf. \cite[Theorem 2.1.5]{Taira}).
\begin{lemma}\label{lem88}
If for any $m,n\in\mathbb{N}$ and $q\in[1,\infty]$ satisfying
$$\frac{1}{q}=\frac{1}{2}+\frac{n}{d}-\frac{m \tilde{\theta}}{d},\ \frac{n}{m}\le\tilde{\theta}\le1,$$
then there is a constant $C>0$ such that
\begin{equation}
\|f\|_{n,q}\le C\|f\|_{m}^{\tilde{\theta}}\|f\|_{L^2}^{1-\tilde{\theta}},\ \ f\in W_0^{m,2}(\mathcal{O}, \mathbb{R}^d).\label{GN_inequality}
\end{equation}
\end{lemma}

\subsection{Stochastic climate models}\label{SecNS}
Recently, Crisan, Holm and Korn \cite{CHK} developed the {\it{Stochastic Advection by Lie Transport}} method, which constructed the idealised climate model couples a stochastic PDE for the atmospheric circulation to a deterministic PDE for the circulation of the ocean.
The abstract form of such model is given by
 \begin{eqnarray}\label{eqcm}
\left\{
 \begin{aligned}
&dX_t+(B(X_t,X_t)-\mathcal{R}(X_t)+\kappa\mathbb{E}X_t)dt=\nu\Delta X_tdt+\sigma(X_t)dW_t,~X_0\sim\mu_0,\\
    &X_t=0, ~~\text{on}~\partial\mathcal{O},
  \end{aligned}
\right.
\end{eqnarray}
where $X:=(X^a,X^o)$ is a state vector with the atmospheric component $X^a$ and the ocean component $X^o$,
 $B$ is the usual bilinear transport operator, $\mathcal{R}$ is a linear operator, $\nu\Delta X:=(\nu_1\Delta^a X^a,\nu_2\Delta^o X^o)$ denotes the dissipation/diffusion operator for velocity and temperature, $\nu_1, \nu_2>0$ denote the dissipation constant, $\kappa\in\mathbb{R}$. Here, $\sigma$ satisfies some appropriate assumptions and $W_t$ is a cylindrical Wiener process defined on $\left(\Omega,\mathscr{F},\{\mathscr{F}_t\}_{t\in[0,T]},\mathbb{P}\right)$ taking values in a separable Hilbert space $U$.

With this application in mind, in the current work, we consider the more general case of (\ref{eqcm}). To this end, we first collect some necessary notations and definitions.

\vspace{1mm}
Let $\mathbb{H}$ be a separable Hilbert space equipped with norm $\|\cdot\|_{\mathbb{H}}$.  Let $L$ be an (unbounded) positive linear self-adjoint operator on $\mathbb{H}$. There exists an orthonormal basis $\{e_k\}_{k\geq 1}$ on $\mathbb{H}$ of eigenfunctions of $L$ and the increasing eigenvalue sequence
$$0<\lambda_1\leq\lambda_2\leq...\leq\lambda_n\leq...\uparrow\infty.$$

Define
$\mathbb{V}=\mathcal{D}(L^{1/2})$ and $\mathbb{X}=\mathcal{D}(L)$, the associated norms are given by
$$\|v\|_{\mathbb{V}}=\|L^{1/2}v\|_{\mathbb{H}},~\|v\|_{\mathbb{X}}=\|Lv\|_{\mathbb{H}}.$$
 Let $\mathbb{V}^*$ be the dual space of $\mathbb{V}$ with respect to the scalar product
$(\cdot,\cdot)$ on $\mathbb{H}$.
Then we can consider a Gelfand triple
$$\mathbb{V} \subset \mathbb{H} \subset \mathbb{V}^*.$$
Let us denote by $\langle{u},v\rangle$ the dualization between $u\in \mathbb{V}$ and $v\in \mathbb{V}^*$.

Let $B:\mathbb{V}\times \mathbb{V} \to \mathbb{V}^*$ be a continuous map satisfying the following conditions
\begin{enumerate}
\item [$(\mathbf{A}^1_{B})$]  $B: \mathbb{V} \times \mathbb{V} \to \mathbb{V}^*$ is a continuous bilinear map.

\item [$(\mathbf{A}^2_{B})$] For all $u_i \in \mathbb{V}, i=1,2,3$,
\begin{equation*}
\langle B(u_1,u_2),u_3 \rangle = -\langle B(u_1,u_3),u_2\rangle,~~\langle B(u_1,u_2),u_2 \rangle=0.
\end{equation*}

\item[$(\mathbf{A}^3_{B})$] There exists a Banach space $\mathcal{H}$ such that
\\
(i) $\mathbb{V} \subset \mathcal{H} \subset \mathbb{H};$
\\
(ii) there exists a constant $a_0>0$ such that
\begin{equation*}
\|u\|_{\mathcal{H}}^2 \le a_0 \|u\|_{\mathbb{H}}\|u\|_{\mathbb{V}}~~\text{for all} ~ u\in \mathbb{V};
\end{equation*}
(iii) for every $\eta>0$ there exists a constant $C_\eta>0$ such that
\begin{equation*}
|\langle B(u_1,u_2),u_3 \rangle| \leq \eta\|u_3\|_{\mathbb{V}}^2+C_\eta\|u_1\|_{\mathcal{H}}^2\|u_2\|_{\mathcal{H}}^2
~~\text{for all}~ u_i\in \mathbb{V},  i=1,2,3,
\end{equation*}
and
\begin{equation*}
|(B(u,u),Lu)| \leq \eta\|u\|_{\mathbb{X}}^2+C_\eta\|u\|_{\mathbb{V}}^4\|u\|_{\mathbb{H}}^2
~~\text{for all}~ u\in \mathbb{X}.
\end{equation*}
\end{enumerate}
For simplicity, we denote $B(u)=B(u,u)$.

Let $\mathcal{R}:\mathbb{V}\times \mathbb{V} \to \mathbb{V}^*$ be a continuous map satisfying the following conditions
\begin{enumerate}
\item[$(\mathbf{A}_{\mathcal{R}})$] There exist constants $\epsilon<1$ and $C>0$ such that
\\
(i)  for any $u\in \mathbb{V}$
\begin{eqnarray*}
2_{\mathbb{V}^*}\langle \mathcal{R}(u),u\rangle_{\mathbb{V}}\leq \epsilon\|u\|_{\mathbb{V}}^2+C(1+\|u\|_{\mathbb{H}}^2)
\end{eqnarray*}
and
\begin{equation*}
\|\mathcal{R}(u)\|_{{\mathbb{V}}^*}^{2}\leq C(1+\|u\|_{\mathbb{V}}^{2})(1+\|u\|_{{\mathbb{H}}}^{2});
\end{equation*}
(ii) for any $u,v\in {\mathbb{V}}$,
\begin{eqnarray*}
&&2_{{\mathbb{V}}^*}\langle \mathcal{R}(u)-\mathcal{R}(v),u-v\rangle_{\mathbb{V}}
\nonumber\\
&&\leq\epsilon\|u-v\|_{\mathbb{V}}^2+
C(1+\|u\|_{\mathbb{V}}^{2})(1+\|u\|_{{\mathbb{H}}}^{2})\|u-v\|_{\mathbb{H}}^2;
\end{eqnarray*}
(iii) $\mathcal R:\mathbb{H}_n\rightarrow\mathbb{V}$  and there exists $C>0$ (independent of $n$) such that
$$\langle \mathcal{R}(u),u\rangle_{\mathbb{V}}\leq \epsilon\|u\|_{\mathbb{X}}^2+C(1+\|u\|_{\mathbb{V}}^2),~u\in\mathbb{H}_n, $$
where $\mathbb{H}_n:=\text{span}\{e_1,\ldots,e_n\}$.
\end{enumerate}

Assume that $\sigma:\mathbb{V}\to L_2(U,\mathbb{H})$ satisfies the following conditions
\begin{enumerate}
\item[$(\tilde{\mathbf{A}}_{\mathcal{\sigma}})$]

(i) there exists a constant $C>0$ such that
$$\|\sigma(u)-\sigma(v)\|_{L_2(U,\mathbb{H})}\leq C\|u-v\|_{\mathbb{H}},~u,v\in\mathbb{V};$$

    (ii) $\sigma:\mathbb{H}_n\rightarrow L_2({U},{\mathbb{V}})$ and there exists $C>0$ (independent of $n$) such that
$$\|\sigma(u)\|_{L_2(U,\mathbb{V})}^2\leq C(1+\|u\|_{\mathbb{V}}^2),~u\in\mathbb{H}_n. $$
\end{enumerate}

Now, we consider
the following mean field  systems
\begin{equation}\label{hys}
dX_t+\big(L X_t+B(X_t)-\mathcal{R}(X_t)+\kappa\mathbb{E}X_t\big)dt=\sigma(X_t)dW_t.
\end{equation}

\begin{theorem}\label{thn0}
Suppose that $(\mathbf{A}^1_{B})$-$(\mathbf{A}^3_{B})$, $(\mathbf{A}_{\mathcal{R}})$ and $(\tilde{\mathbf{A}}_{\mathcal{\sigma}})$ hold. Let the initial value $X_0\in L^p(\Omega;\mathbb{H})\cap L^2(\Omega;\mathbb{V})$, $p\geq 4$. Then (\ref{hys}) has a strong and weak (martingale) solution.
\end{theorem}
\begin{proof}
It suffices to check that $\mathbf{(A_1)}$-$\mathbf{(A_5)}$ hold for
$$\mathcal{A}(t,u,\mu):=-\big(L u+B(u)-\mathcal{R}(u)+\kappa\int y\mu(dy)\big).$$
Note that similar with \cite{LR1} we can see that $\tilde{\mathcal{A}}(u):=-L u-B(u)+\mathcal{R}(u)$ satisfy $\mathbf{(A_1)}$-$\mathbf{(A_3)}$ and $\mathbf{(A_5)}$ with $\alpha=\beta=2$. By $(\mathbf{A}^3_{B})$ (iii) and $(\mathbf{A}_{\mathcal{R}})$ (iii), for $\eta<1-\epsilon$ we obtain
 \begin{eqnarray*}
\langle \tilde{\mathcal{A}}(u),u\rangle_{\mathbb{V}}\leq &&-(Lu,Lu)+|(B(u,u),Lu)|+\langle \mathcal{R}(u),u\rangle_{\mathbb{V}}
 \nonumber\\
\leq&&-(1-\eta-\epsilon)\|u\|_{\mathbb{X}}^2+C(1+\|u\|_{\mathbb{V}}^4\|u\|_{\mathbb{H}}^2+\|u\|_{\mathbb{V}}^2), ~\text {for}~u\in\mathbb{H}_n,
 \end{eqnarray*}
which implies that $\mathbf{(A_4)}$ holds with $\lambda=2$.
\end{proof}

In addition to the existence of weak and strong solutions to (\ref{hys}), we also investigate the corresponding weakly interacting systems
 \begin{eqnarray}\label{eqiNS}
~~~~\left\{
 \begin{aligned}
     &dX^{N,i}_t+\Big(L X^{N,i}_t+B(X^{N,i}_t)-\mathcal{R}(X^{N,i}_t)+\kappa\frac{1}{N}\sum_{j=1}^NX^{N,j}_t\Big)dt=\sigma(X^{N,i}_t)dW_t^i, \\
     &X^{N,i}_0=\xi_i,
  \end{aligned}
\right.
\end{eqnarray}
where $i=1,\ldots,N$, $W_t^1,\ldots,W_t^N$ are  independent  cylindrical Wiener processes defined on a complete filtered probability space $\left(\Omega,\mathscr{F},\{\mathscr{F}_t\}_{t\in[0,T]},\mathbb{P}\right)$. Here, the system \eref{eqiNS} is defined by a family of interacting stochastic climate models that characterize the dynamics of atmosphere-ocean interactions.

Let $N\to\infty$, we demonstrate that the large $N$ limit of (\ref{eqiNS}) is governed by (\ref{hys}).
\begin{theorem}\label{thn1} Suppose that $(\mathbf{A}^1_{B})$-$(\mathbf{A}^3_{B})$, $(\mathbf{A}_{\mathcal{R}})$, $(\tilde{\mathbf{A}}_{\mathcal{\sigma}})$ and the initial value condition $(\mathbf{A_6})$ hold with $\vartheta=4$, then the empirical law  $\{\mathcal{S}^{N},N\in\mathbb{N}\}$ is  tight in $\mathscr{P}_{2}(\mathbb{C}_T)\cap \mathscr{P}_{2}(L^{2}([0,T];\mathbb{V}))$, and any accumulation point $\mathcal{S}$ is a solution  of martingale problem associated with (\ref{hys}) with initial law $\mu_0\in\mathscr{P}_p(\mathbb{H})\cap \mathscr{P}_2(\mathbb{V})$, where $p\in[4,\infty)$.
\end{theorem}
\begin{proof}
It suffices to prove that $\mathbf{(A_5^*)}$ holds for $$\mathcal{A}(t,u,\mu):=-\big(L u+B(u)-\mathcal{R}(u)+\kappa\int y\mu(dy)\big).$$
By $(\mathbf{A}^2_{B})$, $(\mathbf{A}^3_{B})$ and $(\mathbf{A}_{\mathcal{R}})$, it is easy to see that $\mathbf{(A_5^*)}$ holds, whose proof is similar to \cite{LR1}.
\end{proof}

\subsection{Weakly interacting Allen-Cahn equations}\label{ex01}
The Allen-Cahn equation was introduced in \cite{AC79} as a model describing the evolution of antiphase boundaries, the reader can refer to \cite{HLR23,WH10}
and references therein for the study of the  stochastic Allen-Cahn equation.

 Let $\mathcal{O} := [0,1]$. In this part, we consider the following weakly interacting Allen-Cahn equations,
 \begin{eqnarray}\label{Secex1}
\left\{
 \begin{aligned}
     &dX^{N,i}_t=\Big\{\Delta X^{N,i}_t+X_t^{N,i}-\frac{1}{N}\sum_{j=1}^{N}(X_t^{N,j})^2X_t^{N,i}\Big\}dt+dW_t^i,~~X^{N,i}_0=\xi_i, \\
    & X^{N,i}_t(0) =X^{N,i}_t(1)=0,
  \end{aligned}
\right.
\end{eqnarray}
where $i=1,\ldots,N$, $W_t^1,\ldots,W_t^N$ are independent $Q$-Wiener processes (with $Q$ being  nonnegative, symmetric and with finite trace throughout this section) defined on a complete filtered probability space $\left(\Omega,\mathscr{F},\{\mathscr{F}_t\}_{t\in[0,T]},\mathbb{P}\right)$.

By applying our abstract result, we will show that the large $N$ limit of  (\ref{Secex1}) is
 the following mean field type Allen-Cahn equations,
  \begin{eqnarray}\label{eqSem}
\left\{
 \begin{aligned}
     &dX_t=\big(\Delta X_t+X_t-(\mathbb{E}X_t^2)X_t\big)dt+dW_t,~~X_0\sim\mu_0, \\
    & X_t(0) =X_t(1)=0,
  \end{aligned}
\right.
\end{eqnarray}
where $W_t$ is also an $Q$-Wiener process defined on $\left(\Omega,\mathscr{F},\{\mathscr{F}_t\}_{t\in[0,T]},\mathbb{P}\right)$.

\begin{remark}
One of motivations for considering the stochastic interacting Allen-Cahn equations with the quadratic interacting kernel is related to the stochastic quantization theory. For instance, Shen et al.~\cite{SSZZ,SZZ22} utilized the dynamic (\ref{Secex1}) with more singular noise to investigate the large $N$ limit of the linear sigma model, where the term $\mathbb{E}[X_t^2]X_t$ appears in the mean field limit equation (cf.~\cite[(1.5)]{SSZZ}). Another interesting observation is that the mean field Allen-Cahn equation with such type of interacting kernel is related to the scaling limit of the classical Allen-Cahn equation with a rough random initial datum, see e.g.~\cite[Proposition 1.2]{GRZ}.
\end{remark}

Set
\begin{equation*}
\mathbb{H}:=L^2(\mathcal{O},\mathbb{R}),~~\mathbb{V}:=W_0^{1,2}(\mathcal{O}, \mathbb{R}),~~\mathbb{X}:=W_0^{2,2}(\mathcal{O}, \mathbb{R}).
\end{equation*}
In particular, we use the following (equivalent) Sobolev norm on $\mathbb{V}$ and $\mathbb{X}$, respectively,
\begin{eqnarray*}
&&\|u\|_{1}:=\|\nabla u\|_{L^2}=\left(\int_\mathcal{O}|\nabla u|^2 dx\right)^\frac{1}{2},\nonumber\\
&&
\|u\|_{2}:=\|\Delta u\|_{L^2}=\left(\int_\mathcal{O}|\Delta u|^2 dx\right)^\frac{1}{2}.
\end{eqnarray*}
Identifying $\mathbb H$ with its dual space by the Riesz isomorphism, then we have
\begin{equation*}\label{g1Sem}
\mathbb V\subset \mathbb H(\simeq \mathbb H^*)\subset \mathbb V^*.
\end{equation*}

\begin{theorem}\label{thsem2}
Suppose that the initial value condition $(\mathbf{A_6})$ holds with $\vartheta=8$ and $\mu_0\in\mathscr{P}_p(\mathbb{H})\cap \mathscr{P}_2(\mathbb{V})$ with $p\in[8,\infty)$, then (\ref{eqSem}) has a unique strong and weak (martingale) solution such that
\begin{equation}\label{essem2}
\lim_{N\to\infty}\mathbb{E}\Big\{\mathcal{W}_{2,\mathbb{C}_T}(\mathcal{S}^{N},\Gamma)^2+\mathcal{W}_{2,L^{2}([0,T];\mathbb{V})}(\mathcal{S}^{N},\Gamma)^{2}\Big\}=0,
\end{equation}
where $\Gamma$ is the unique martingale solution to (\ref{eqSem}) with initial law $\mu_0$.

\end{theorem}
\begin{proof}
Note that  all eigenvectors $\{e_1,e_2,\ldots\}\subset W_0^{2,2}(\mathcal{O}, \mathbb{R})$ of $\Delta$ constitute an orthonormal basis of $L^2(\mathcal{O},\mathbb{R})$ and an orthogonal set in $W_0^{1,2}(\mathcal{O}, \mathbb{R})$, i.e.~the condition $\mathbf{(A_0)}$ holds. We intend to show
$$\mathcal{A}(u,\mu):=\Delta u+u-u\int y^2\mu(dy),$$
satisfies the conditions $\mathbf{(A_1)}$-$\mathbf{(A_4)}$ and $\mathbf{(A_5')}$.

It is easy to see
\begin{eqnarray}\label{exSem02}
_{\mathbb{V}^*}\langle \mathcal{A}(u,\mu),u\rangle_{\mathbb{V}}
=&&\!\!\!\!\!\!\!\!-\|u\|_{1}^{2}+\|u\|_{L^2}^{2}-\int_{\mathcal{O}}u(x)^2\Big(\int y(x)^2\mu(dy)\Big)dx
\nonumber \\
\leq&&\!\!\!\!\!\!\!\!-\|u\|_{1}^{2}+\|u\|_{L^2}^{2},
\end{eqnarray}
and
\begin{eqnarray}\label{prSem02}
|_{\mathbb{V}^*}\langle \mathcal{A}(u,\mu),v\rangle_{\mathbb{V}}|\leq&&\!\!\!\!\!\!\!\!|\int_{\mathcal{O}}\Delta u(x)v(x)dx|+|\int_{\mathcal{O}} u(x)v(x)dx|
\nonumber \\
&&\!\!\!\!\!\!\!\!
+\int_{\mathcal{O}}\Big(|u(x)v(x)|\int y(x)^2\mu(dy)\Big)dx
\nonumber \\
\leq&&\!\!\!\!\!\!\!\!\|u\|_1\|v\|_1+\|u\|_{L^2}\|v\|_{L^2}+\|v\|_{L^\infty} \|u\|_{L^2}\int\|y\|_{L^4}^2\mu(dy) \nonumber \\
\leq&&\!\!\!\!\!\!\!\!C\|v\|_1\Big(\|u\|_1+\|u\|_{L^2}\int(\|y\|_{1}+\|y\|_{L^2}^3)\mu(dy)\Big),
\end{eqnarray}
where we have used Young's inequality and interpolation inequality \eref{GN_inequality} with $n=0, q=4, m=1,
\tilde{\theta}=\frac{1}{4}$ in the last step. Then
$$\| \mathcal{A}(u,\mu)\|_{\mathbb{V}^*}^2\leq C\big(\|u\|_{1}^2+\|u\|_{L^2}^2\big(
\mu(\|\cdot\|_{1}^2)+\mu(\|\cdot\|_{L^2}^6)\big)\big),$$
thus  $\mathbf{(A_2)}$ and $\mathbf{(A_3)}$ hold for $\alpha=2,\beta=6$.

For any $u_1,u_2,v\in {\mathbb{V}}$, $\mu,\nu\in\mathbb{M}_1$ and coupling $\pi\in\mathscr{C}(\mu,\nu)$, we have
\begin{eqnarray}\label{demi01}
&&_{\mathbb{V}^*}\langle \mathcal{A}(u_1,\mu)-\mathcal{A}(u_2,\nu),v\rangle_{\mathbb{V}}
\nonumber \\ \leq&&\!\!\!\!\!\!\!\!\|u_1-u_2\|_{1}\|v\|_1+\|u_1-u_2\|_{L^2}\|v\|_{L^2}+\int_{\mathcal{O}}|(u_1-u_2)(x)v(x)|\Big(\int y(x)^2\mu(dy)\Big)dx
\nonumber \\
&&\!\!\!\!\!\!\!\!+\int_{\mathcal{O}}|u_2(x)v(x)|\Big(\int |y_1^2-y_2^2|\pi(dy_1,dy_2)\Big)dx
\nonumber \\ \leq&&C\|u_1-u_2\|_{1}\|v\|_1+\|u_1-u_2\|_{L^\infty}\|v\|_{L^\infty}\int_{\mathcal{O}}\Big(\int y(x)^2\mu(dy)\Big)dx
\nonumber \\
&&\!\!\!\!\!\!\!\!+C\|u_2\|_{L^\infty}\|v\|_{L^\infty}\big(\mu(\|\cdot\|_{L^2}^2)^{1/2}
+\nu(\|\cdot\|_{L^2}^2)^{1/2}\big)
\Big(\int \|y_1-y_2\|_{L^2}^2\pi(dy_1,dy_2)\Big)^{1/2}
\nonumber \\
\leq&&\!\!\!\!\!\!\!\!C\|u_1-u_2\|_{1}\|v\|_1+C\|u_1-u_2\|_{1}\|v\|_1\mu(\|\cdot\|_{L^2}^2)^2
\nonumber \\
&&\!\!\!\!\!\!\!\!+C\|u_2\|_{1}\|v\|_1\big(\mu(\|\cdot\|_{L^2}^2)^{1/2}
+\nu(\|\cdot\|_{L^2}^2)^{1/2}\big)\Big(\int \|y_1-y_2\|_{L^2}^2\pi(dy_1,dy_2)\Big)^{1/2}.\nonumber
\end{eqnarray}
Taking infimum for couplings $\pi\in\mathscr{C}(\mu,\nu)$, we obtain
\begin{eqnarray*}
&&\!\!\!\!\!\!\!\!_{\mathbb{V}^*}\langle \mathcal{A}(u_1,\mu)-\mathcal{A}(u_2,\nu),v\rangle_{\mathbb{V}}
\nonumber \\
\leq&&\!\!\!\!\!\!\!\!C\|u_1-u_2\|_{1}\|v\|_1+C\|u_1-u_2\|_{1}\|v\|_1\mu(\|\cdot\|_{L^2}^2)^2
\nonumber \\
&&+C\|u_2\|_{1}\|v\|_1\big(\mu(\|\cdot\|_{L^2}^2)^{1/2}
+\nu(\|\cdot\|_{L^2}^2)^{1/2}\big)\mathcal{W}_{2,L^2}(\mu,\nu),\nonumber
\end{eqnarray*}
which tends to $0$ if $(u_1,\mu)$ tends to $(u_2,\nu)$ in ${\mathbb{V}}\times\mathbb{M}_1$. Thus, the demicontinuity $\mathbf{(A_1)}$ holds.

Moreover,
\begin{eqnarray}\label{prSem04}
\langle \mathcal{A}(u,\mu),u\rangle_{\mathbb{V}}=&&\!\!\!\!\!\!\!\!-\|u\|_2^2+\|u\|_1^2-\langle u\int y^2\mu(dy),u\rangle_{\mathbb{V}}
\nonumber \\
\leq&&\!\!\!\!\!\!\!\!-\|u\|_2^2+\|u\|_1^2+\frac{1}{2}\|u\|_2^2+C\|u\|_{1}^2\int\|y^2\|_{L^2}^2\mu(dy)\nonumber
\nonumber \\
\leq&&\!\!\!\!\!\!\!\!-\frac{1}{2}\|u\|_2^2+\|u\|_1^2+C\|u\|_{1}^2\int\|y\|_{L^4}^4\mu(dy)\nonumber
\\
\leq&&\!\!\!\!\!\!\!\!-\frac{1}{2}\|u\|_2^2+\|u\|_1^2+C\|u\|_{1}^2\int(\|y\|_{1}^2+\|y\|_{L^2}^6)\mu(dy)
,
\end{eqnarray}
which shows that the condition $(\mathbf{A_4})$ holds with $\gamma=2$ and $\lambda=6$.

Now, we show the condition $\mathbf{(A_5')}$ is satisfied. For any  $u,v\in {\mathbb{V}}$, $\mu,\nu\in\mathbb{M}_2$ and coupling $\pi\in\mathscr{C}(\mu,\nu)$,
\begin{eqnarray*}\label{ese1}
&&\!\!\!\!\!\!\!\!_{\mathbb{V}^*}\langle \mathcal{A}(u,\mu)-\mathcal{A}(v,\nu),u-v\rangle_{\mathbb{V}}
\nonumber \\
=&&\!\!\!\!\!\!\!\!-\|u-v\|_{1}^2+\|u-v\|_{L^2}^2-\int_{\mathcal{O}}(u-v)^2\Big(\int y_1^2\mu(dy_1)\Big)dx
\nonumber \\
&&\!\!\!\!\!\!\!\!
-\int_{\mathcal{O}}v(u-v)\Big(\int (y_1^2-y_2^2)\pi(dy_1,dy_2)\Big)dx
\nonumber \\
\leq
&&\!\!\!\!\!\!\!\!\|u-v\|_{L^2}^2+C\int_{\mathcal{O}}\Bigg(|v|(u-v)\Big(\int (y_1-y_2)^2\pi(dy_1,dy_2)\Big)^{1/2}
\nonumber \\
&&\!\!\!\!\!\!\!\!\cdot
\Big(\int y_1^2\mu(dy_1)+\int y_2^2\nu(dy_2)\Big)^{1/2}\Bigg)dx
\nonumber \\\leq
&&\!\!\!\!\!\!\!\!\|u-v\|_{L^2}^2+C\int_{\mathcal{O}}(u-v)^2\Big(\int y_1^2\mu(dy_1)+\int y_2^2\nu(dy_2)\Big)dx
\nonumber \\
&&\!\!\!\!\!\!\!\!
+C\int_{\mathcal{O}}v^2\Big(\int (y_1-y_2)^2\pi(dy_1,dy_2)\Big)dx
\nonumber \\ \leq
&&\!\!\!\!\!\!\!\!C\|u-v\|_{L^2}^2\Big(1+\int \|y_1\|_{L^\infty}^2\mu(dy_1)+\int \|y_2\|_{L^\infty}^2\nu(dy_2)\Big)
\nonumber \\
&&\!\!\!\!\!\!\!\!+C\|v\|_{L^\infty}^2\int \|y_1-y_2\|_{L^2}^2\pi(dy_1,dy_2)
\nonumber \\ \leq
&&\!\!\!\!\!\!\!\!C\big(1+\mu(\|\cdot\|_{1}^2)+\nu(\|\cdot\|_{1}^2)\big)\|u-v\|_{L^2}^2+
C\|v\|_{1}^2\int \|y_1-y_2\|_{L^2}^2\pi(dy_1,dy_2).
\end{eqnarray*}
Taking infimum for couplings $\pi\in\mathscr{C}(\mu,\nu)$, we obtain
\begin{eqnarray*}
&&_{\mathbb{V}^*}\langle \mathcal{A}(u,\mu)-\mathcal{A}(v,\nu),u-v\rangle_{\mathbb{V}}
\nonumber \\
\leq&& C\big(1+\mu(\|\cdot\|_{1}^2)+\nu(\|\cdot\|_{1}^2)\big)\|u-v\|_{L^2}^2+
C\|v\|_{1}^2\mathcal{W}_{2,L^2}(\mu,\nu)^2.
\end{eqnarray*}
Therefore, the results hold by applying Theorems \ref{th8}-\ref{th4}.
\end{proof}

\begin{remark}
 In this subsection, we restrict our analysis to the case of additive noise for the sake of simplicity. Note that all the theorems and results presented in this section hold for general multiplicative noise as well.

\end{remark}

\subsection{Weakly interacting Burgers type equations}\label{SecBer}
In this part, we are interested in the weakly interacting stochastic Burgers type equations. The Burgers equation is used to describe both nonlinear propagation effects and diffusive effects, occurring in various areas of applied mathematics, such as gas dynamics, fluid mechanics, nonlinear acoustics, and (more recently) traffic flow, which was studied mathematically by Burgers (cf. \cite{Bur39}) in the 1940s. The Burgers equations perturbed by random noise serve as a suitable model for a wide range of problems in statistical physics, see e.g. \cite{BCJ94, EKMS20} and references therein for the study of stochastic Burgers equations.

Let $\mathcal{O} := [0,1]$. In this part, we consider the following system
 \begin{eqnarray}\label{eqiBer}
\left\{
 \begin{aligned}
     &dX^{N,i}_t=\Big\{\Delta X^{N,i}_t+\frac{1}{N}\sum_{j=1}^N\partial_x \big(X^{N,i}_t\cdot\varphi(X^{N,j}_t)\big)\Big\}dt+dW^i_t,~~X^{N,i}_0=\xi_i, \\
    & X^{N,i}_t(0) =X^{N,i}_t(1)=0,
  \end{aligned}
\right.
\end{eqnarray}
where $i=1,\ldots,N$, $W_t^1,\ldots,W_t^N$ are  independent $U$-valued $Q$-Wiener processes defined on a complete filtered probability space $\left(\Omega,\mathscr{F},\{\mathscr{F}_t\}_{t\in[0,T]},\mathbb{P}\right)$. The system \eqref{eqiBer} investigated in this context comprises a family of interacting stochastic Burgers-type equations, which model the interactive dynamics of fluid systems.

In this work, we can apply our general results  to study the large $N$ limit of the weakly interacting Burgers type equations (\ref{eqiBer}), which is governed by the following mean field Burgers type equations,
\begin{eqnarray}\label{eqBer}
\left\{
 \begin{aligned}
&dX_t=\big[\Delta X_t+\big(\mathbb{E}\varphi(X_t)\big)\partial_xX_t
+X_t\big(\partial_x\mathbb{E}\varphi(X_t)\big)\big]dt+dW_t,~X_0\sim\mu_0,\\
    &X_t(0) =X_t(1)=0,
  \end{aligned}
\right.
\end{eqnarray}
where $W_t$ is an $Q$-Wiener process defined on $\left(\Omega,\mathscr{F},\{\mathscr{F}_t\}_{t\in[0,T]},\mathbb{P}\right)$.

\vspace{1mm}
Set
\begin{equation*}
\mathbb{H}:=L^2(\mathcal{O},\mathbb{R}),~~\mathbb{V}:=W_0^{1,2}(\mathcal{O}, \mathbb{R}),~~\mathbb{X}:=W_0^{2,2}(\mathcal{O}, \mathbb{R}).
\end{equation*}
Identifying $\mathbb H$ with its dual space by the Riesz isomorphism, then we have
\begin{equation}\label{g1Ber}
\mathbb V\subset \mathbb H(\simeq \mathbb H^*)\subset \mathbb V^*.
\end{equation}

The large $N$ limit result to (\ref{eqiBer}) is given by the following theorem.
\begin{theorem}\label{thber1}
Suppose that  $\varphi\in \mathcal{C}_b^1(\mathbb{R};\mathbb{R})$ and  the initial value condition $(\mathbf{A_6})$ holds with $\vartheta=4$. There exists a unique martingale solution $\Gamma$  to  (\ref{eqBer}) with initial law $\mu_0\in\mathscr{P}_p(\mathbb{H})\cap \mathscr{P}_2(\mathbb{V})$, where $p\in[4,\infty)$,  such that
\begin{equation}\label{esb1}
\lim_{N\to\infty}\mathbb{E}\Big\{\mathcal{W}_{2,\mathbb{C}_T}(\mathcal{S}^{N},\Gamma)^2+\mathcal{W}_{2,L^{2}([0,T];\mathbb{V})}(\mathcal{S}^{N},\Gamma)^{2}\Big\}=0.
\end{equation}

\end{theorem}
\begin{proof}
We denote
$$\mathcal{A}(t,u,\mu):=\Delta u+\int\partial_x(u\cdot\varphi(y))\mu(dy).$$
Then it suffices to check that $\mathbf{(A_1)}$-$\mathbf{(A_4)}$ and $\mathbf{(A_5')}$ hold for $\mathcal{A}$.

By the integration of parts formula, we have
\begin{eqnarray*}
_{\mathbb{V}^*}\langle \mathcal{A}(u,\mu),u\rangle_{\mathbb{V}}\leq
&&\!\!\!\!\!\!\!\!-\|u\|_{1}^{2}-\int\int_{\mathcal{O}}\varphi(y)u\cdot\partial_xudx\mu(dy)
\nonumber \\
\leq&&\!\!\!\!\!\!\!\!-\frac{1}{2}\|u\|_{1}^{2}+C\|u\|_{L^2}^2,
\end{eqnarray*}
and
\begin{eqnarray*}
|_{\mathbb{V}^*}\langle \mathcal{A}(u,\mu),v\rangle_{\mathbb{V}}|\leq&&\!\!\!\!\!\!\!\!\|u\|_1\|v\|_1
+C(\|u\|_{1}+\|u\|_{L^2})\|v\|_{L^2}.
\end{eqnarray*}
Then
$$\|\mathcal{A}(u,\mu)\|_{\mathbb{V}^*}^2\leq C\(\|u\|_{1}^2+\|u\|_{L^2}^2),$$
thus  $\mathbf{(A_2)}$ and $\mathbf{(A_3)}$ hold for $\alpha=2,\beta=2$.

For any $u_1,u_2,v\in {\mathbb{V}}$, $\mu,\nu\in\mathbb{M}_1$ and coupling $\pi\in\mathscr{C}(\mu,\nu)$, we have
\begin{eqnarray}\label{demi02}
&&\!\!\!\!\!\!\!\!_{\mathbb{V}^*}\langle \mathcal{A}(u_1,\mu)-\mathcal{A}(u_2,\nu),v\rangle_{\mathbb{V}}
\nonumber \\ \leq&&\!\!\!\!\!\!\!\!\|u_1-u_2\|_{1}\|v\|_1+{_{\mathbb{V}^*}}\langle \Big(\int\varphi(y_1)\mu(dy_1)-\int\varphi(y_2)\nu(dy_2)\Big)\partial_xu_2,v\rangle_{\mathbb{V}}
\nonumber \\
&&\!\!\!\!\!\!\!\!+{_{\mathbb{V}^*}}\langle \int\varphi(y_1)\mu(dy_1)(\partial_xu_1-\partial_xu_2),v\rangle_{\mathbb{V}}
+{_{\mathbb{V}^*}}\langle \int\partial_x\varphi(y_1)\mu(dy_1)(u_1-u_2),v\rangle_{\mathbb{V}}
\nonumber \\
&&\!\!\!\!\!\!\!\!
+{_{\mathbb{V}^*}}\langle \Big(\int\partial_x\varphi(y_1)\mu(dy_1)-\int\partial_x\varphi(y_2)\nu(dy_2)\Big)u_2,v\rangle_{\mathbb{V}}
\nonumber \\ \leq&&\!\!\!\!\!\!\!\!\|u_1-u_2\|_{1}\|v\|_1+C\|v\|_{L^\infty}\|u_2\|_{1}
\int\|y_1-y_2\|_{L^2}\pi(dy_1,dy_2)dx+C\|u_1-u_2\|_{1}\|v\|_{L^2}
\nonumber \\
&&\!\!\!\!\!\!\!\!+C\|u_1-u_2\|_{L^2}\|v\|_{1}
+C(\|u_2\|_{L^\infty}\|v\|_{1}+\|u_2\|_{1}\|v\|_{L^\infty})
\int\|y_1-y_2\|_{L^2}\pi(dy_1,dy_2)
\nonumber \\
\leq&&\!\!\!\!\!\!\!\!C\|u_1-u_2\|_{1}\|v\|_1+C\|u_2\|_{1}\|v\|_1\Big(\int \|y_1-y_2\|_{L^2}^2\pi(dy_1,dy_2)\Big)^{1/2}.\nonumber
\end{eqnarray}
Taking infimum for couplings $\pi\in\mathscr{C}(\mu,\nu)$, we obtain
\begin{eqnarray*}
_{\mathbb{V}^*}\langle \mathcal{A}(u_1,\mu)-\mathcal{A}(u_2,\nu),v\rangle_{\mathbb{V}}
\leq&&C\|u_1-u_2\|_{1}\|v\|_1+C\|u_2\|_{1}\|v\|_1\mathcal{W}_{2,L^2}(\mu,\nu),\nonumber
\end{eqnarray*}
which tends to $0$ if $(u_1,\mu)$ tends to $(u_2,\nu)$ in ${\mathbb{V}}\times\mathbb{M}_1$. Thus, the demicontinuity $\mathbf{(A_1)}$ holds.

On the other hand,
\begin{equation}\label{CHL01a}
\langle \Delta u,u\rangle_{\mathbb{V}}=-\|u\|_2^2
\end{equation}
and
\begin{eqnarray}\label{CHL02b}
&&\!\!\!\!\!\!\!\!\langle\int\varphi(y)\mu(dy)\partial_xu+\int\partial_x\varphi(y)\mu(dy)u,u\rangle_{\mathbb{V}}
\nonumber\\
=&&\!\!\!\!\!\!\!\!\int\int_{\mathcal{O}}\varphi(y)\partial_xu\cdot\partial^2_xudx \mu(dy) +\int\int_{\mathcal{O}}\partial_x\varphi(y)u\cdot\partial^2_xudx\mu(dy)
\nonumber\\
=&&\!\!\!\!\!\!\!\!\int\int_{\mathcal{O}}\varphi(y)\partial_xu\cdot\partial^2_xudx \mu(dy) +\int\int_{\mathcal{O}}\partial_y\varphi(y)\partial_xy\cdot u\partial^2_xudx\mu(dy)
\nonumber \\
\leq&&\!\!\!\!\!\!\!\!C\|u\|_{2}(\|u\|_{1}+\|u\|_{1}\int\|y\|_1\mu(dy))\nonumber \\
\leq&&\!\!\!\!\!\!\!\!\frac{1}{2}\|u\|_2^2+C\big(\|u\|_{1}^2+\|u\|_{1}^2\mu(\|\cdot\|_1^2)\big),
\end{eqnarray}
which  implies the the condition $\mathbf{(A_4)}$ hold for $\gamma=\lambda=2$.

Next we show the condition $\mathbf{(A_5')}$ is satisfied. Note that
\begin{eqnarray}\label{ese1a1}
&&\!\!\!\!\!\!\!\!_{\mathbb{V}^*}\langle \mathcal{A}(u,\mu)-\mathcal{A}(v,\nu),u-v\rangle_{\mathbb{V}}
\nonumber \\
=&&\!\!\!\!\!\!\!\!-\|u-v\|_{1}^2+{_{\mathbb{V}^*}}\langle \big(\int\varphi(y_1)\mu(dy_1)-\int\varphi(y_2)\nu(dy_2)\big)\partial_xv,u-v\rangle_{\mathbb{V}}
\nonumber \\
&&\!\!\!\!\!\!\!\!+{_{\mathbb{V}^*}}\langle \int\varphi(y_1)\mu(dy_1)(\partial_xu-\partial_xv),u-v\rangle_{\mathbb{V}}
\nonumber \\
&&\!\!\!\!\!\!\!\!+{_{\mathbb{V}^*}}\langle \int\partial_x\varphi(y_1)\mu(dy_1)(u-v),u-v\rangle_{\mathbb{V}}
\nonumber \\
&&\!\!\!\!\!\!\!\!
+{_{\mathbb{V}^*}}\langle \big(\int\partial_x\varphi(y_1)\mu(dy_1)-\int\partial_x\varphi(y_2)\nu(dy_2)\big)v,u-v\rangle_{\mathbb{V}}
\nonumber \\
=:
&&\!\!\!\!\!\!\!\!-\|u-v\|_{1}^2+I+II+III+IV.
\end{eqnarray}
For $I$, there exists a constant $C>0$ such that
\begin{eqnarray}\label{ese1a02}
I\leq
&&\!\!\!\!\!\!\!\!\|u-v\|_{L^\infty}\int_{\mathcal{O}}\int|y_1-y_2|\pi(dy_1,dy_2)\partial_xvdx
\nonumber \\\leq
&&\!\!\!\!\!\!\!\!\frac{1}{4}\|u-v\|_{1}^2+
C\|v\|_{1}^2\int \|y_1-y_2\|_{L^2}^2\pi(dy_1,dy_2),
\end{eqnarray}
with any coupling $\pi\in\mathscr{C}(\mu,\nu)$.

Using the boundedness of $\varphi$, it can be inferred that
\begin{eqnarray}\label{ese1a02}
II\leq
&&\!\!\!\!\!\!\!\!C\|u-v\|_{1}\|u-v\|_{L^2}
\nonumber \\\leq
&&\!\!\!\!\!\!\!\!\frac{1}{4}\|u-v\|_{1}^2+C\|u-v\|_{L^2}^2.
\end{eqnarray}
Similarly, we obtain
\begin{eqnarray}\label{ese1a03}
III=
&&\!\!\!\!\!\!\!\!-\int\int_{\mathcal{O}}\varphi(y_1)\partial_x(u-v)^2dx\mu(dy_1)
\nonumber \\\leq
&&\!\!\!\!\!\!\!\!C\int_{\mathcal{O}}|\partial_x(u-v)|\cdot|u-v|dx
\nonumber \\\leq
&&\!\!\!\!\!\!\!\!\frac{1}{4}\|u-v\|_{1}^2+C\|u-v\|_{L^2}^2
\end{eqnarray}
and
\begin{eqnarray}\label{ese1a04}
IV=
&&\!\!\!\!\!\!\!\!-\int\int_{\mathcal{O}}(\varphi(y_1)-\varphi(y_2))\partial_x((v(u-v))dx\pi(dy_1,dy_2)
\nonumber \\\leq
&&\!\!\!\!\!\!\!\!C\int\int_{\mathcal{O}}|y_1-y_2|\big((u-v)\partial_xv+v(\partial_xu-\partial_xv)\big)dx\pi(dy_1,dy_2)
\nonumber \\\leq
&&\!\!\!\!\!\!\!\!\frac{1}{4}\|u-v\|_{1}^2+
C\|v\|_{1}^2\int \|y_1-y_2\|_{L^2}^2\pi(dy_1,dy_2).
\end{eqnarray}
Combining \eref{ese1a1}-\eref{ese1a04} and taking infimum for couplings $\pi\in\mathscr{C}(\mu,\nu)$, we obtain
\begin{eqnarray*}
&&\!\!\!\!\!\!\!\!_{\mathbb{V}^*}\langle \mathcal{A}(u,\mu)-\mathcal{A}(v,\nu),u-v\rangle_{\mathbb{V}}
\leq C\|u-v\|_{L^2}^2+
C\|v\|_{1}^2\mathcal{W}_{2,L^2}(\mu,\nu).
\end{eqnarray*}
Therefore, the results hold by applying Theorems \ref{th8}-\ref{th4}.
\end{proof}

\begin{remark}
 To the best of our knowledge, this is the first result showing the mean field limit for the weakly interacting SPDEs, especially for stochastic Burgers type equations, with high-order nonlinear interactions.
\end{remark}

\section{Proof of well-posedness}\label{well-posed}

In this part, we shall prove the well-posedness for the mean field SPDE (\ref{eqSPDE}). More precisely, in Subsection \ref{sub5.1}, we first consider the Galerkin approximating sequences of Eq.~(\ref{eqSPDE}) and establish some a priori estimates. In Subsection \ref{sub5.2}, we  establish the tightness of the approximation sequence in $\mathbb{C}_T(\mathbb{H})$ and $L^\alpha([0,T];\mathbb{V})$. Leveraging the martingale approach, we give the proof of the existence of martingale  solutions to mean field SPDE (\ref{eqSPDE}) and the existence of strong solutions by means of the modified Yamada-Watanabe theorem in Subsections \ref{sec2.4}. Finally, we complete the proof of Theorem \ref{th2} in Subsection \ref{proof2} by showing the pathwise uniqueness of solutions to Eq.~(\ref{eqSPDE}) under two distinct types of local monotonicity conditions.

\subsection{Galerkin scheme}\label{sub5.1}
Choose $\{e_1,e_2,\cdots\}\subset {\mathbb{X}}$ defined in the condition $(\mathbf{A_0})$ as an orthonormal basis (ONB) on ${\mathbb{H}}$. Consider the maps
$$\pi_n:{\mathbb{V}}^{*}\rightarrow {\mathbb{H}}_n,~n\in\mathbb{N},$$
by
$$\pi_n x:=\sum\limits_{i=1}^{n}{}_{{\mathbb{V}}^*}\langle x,e_i\rangle_{{\mathbb{V}}
}e_i,~x\in \mathbb{V}
^*.$$
It is straightforward that if we restrict $\pi_n$ to ${\mathbb{H}}$, denoted by $\pi_n|_{{\mathbb{H}}}$, then it is an orthogonal projection onto ${\mathbb{H}}_n$ on ${\mathbb{H}}$, and we know
$$_{\mathbb{V}^*}\langle \pi_n\mathcal{A}(t,u,\mu),v\rangle_{\mathbb{V}}=\langle \pi_n\mathcal{A}(t,u,\mu),v\rangle_{\mathbb{H}}={_{\mathbb{V}^*}}\langle \mathcal{A}(t,u,\mu),v\rangle_{\mathbb{V}},~u\in\mathbb{V},v\in \mathbb{H}_n.$$
Due to $\mathbf{(A_0)}$,  for each $v\in \mathbb{H}_n$ we can get that
$$\langle e_i,v\rangle_{\mathbb{V}}=0,~i\geq n+1.$$
In light of $\mathbf{(A_4)}$, for each $t\in[0,T],u\in \mathbb{H}_n,\mu\in\mathscr{P}(\mathbb{H}_n)$ we know that $\mathcal{A}(t,u,\mu)\in\mathbb{V}$ and thus
$$\sum_{i=1}^{\infty}{}_{\mathbb{V}^*}\langle \mathcal{A}(t,u,\mu),e_i\rangle_{\mathbb{V}}e_i=\sum_{i=1}^{\infty}\langle \mathcal{A}(t,u,\mu),e_i\rangle_{\mathbb{H}}e_i=\mathcal{A}(t,u,\mu).$$
Then for each $u,v\in \mathbb{H}_n,\mu\in\mathscr{P}(\mathbb{H}_n)$, it follows that
\begin{eqnarray}\label{es7}
\langle \pi_n\mathcal{A}(t,u,\mu),v\rangle_{\mathbb{V}}=&&\!\!\!\!\!\!\!\!\Big\langle \sum_{i=1}^{n}{}_{\mathbb{V}^*}\langle \mathcal{A}(t,u,\mu),e_i\rangle_{\mathbb{V}}e_i,v\Big\rangle_{\mathbb{V}}
=\Big\langle \sum_{i=1}^{\infty}{}_{\mathbb{V}^*}\langle \mathcal{A}(t,u,\mu),e_i\rangle_{\mathbb{V}}e_i,v\Big\rangle_{\mathbb{V}}
\nonumber \\
=&&\!\!\!\!\!\!\!\!\langle \mathcal{A}(t,u,\mu),v\rangle_{\mathbb{V}}.
\end{eqnarray}

 Denote by $\{g_1,g_2,\cdots\}$ the ONB of $U$. Let
\begin{equation*}
W^{(n)}_t:=\tilde{\pi}_nW_t=\sum\limits_{i=1}^{n}\langle W_t,g_i\rangle_{U}g_i,~n\in\mathbb{N},
\end{equation*}
where $\tilde{\pi}_n$ is an orthonormal projection onto $U^n:=\text{span}\{g_1,g_2,\cdots,g_n\}$ on $U$.

For any $n\in\mathbb{N}$, we consider the following stochastic equation on ${\mathbb{H}}_n$
\begin{eqnarray}\label{eqf}
dX^{(n)}_t=\pi_n\mathcal{A}(t,X^{(n)}_t,\mathscr{L}_{X^{(n)}_t})dt
+\pi_n\mathcal{B}(t,X^{(n)}_t,\mathscr{L}_{X^{(n)}_t})dW^{(n)}_t,
\end{eqnarray}
with initial value $X^{(n)}_0=\pi_n\xi$.  Setting
$$\Omega^{(n)}:=C([0,T];\mathbb{H}_n)$$
and
$$\mathscr{F}_t^{(n)}:=\mathscr{B}(C([0,t];\mathbb{H}_n)),~\mathscr{F}^{(n)}:=\vee_{t\in[0,T]}\mathscr{B}(C([0,t];\mathbb{H}_n)).$$
Then by \cite[Theorem 3.1]{HLL23}, under $\mathbf{(A_1)}$-$\mathbf{(A_3)}$, we know that (\ref{eqf}) admits a martingale solution, i.e., there exists a probability measure $\mathbf{P}_n\in\mathscr{P}(\Omega^{(n)})$ such that $(M1)$ and $(M2)$ in Definition \ref{de3} hold. Throughout this section, the generic point in $\Omega^{(n)}$ is still denoted by  $X^{n}$.

\vspace{2mm}
We have the following a priori estimates.
\begin{lemma}\label{lem3.0}
Suppose that the assumptions in Theorem \ref{th1} hold. For any $T>0,q\geq 2$ there exists a constant $C_{q,T}>0$ such that for any $n\in\mathbb{N}$,
\begin{eqnarray*}
\mathbb{E}^{\mathbf{P}_n}\Big[\sup_{t\in[0,T]}\|X^{(n)}_t\|_{\mathbb{H}}^q\Big]+\mathbb{E}^{\mathbf{P}_n}\int_0^T\|X^{(n)}_t\|_{\mathbb{V}}^{\alpha}\|X^{(n)}_t\|_{\mathbb{H}}^{q-2}dt
\leq C_{q,T}\big(1+\mathbb{E}\|\xi\|_{\mathbb{H}}^q\big).
\end{eqnarray*}
\end{lemma}

\begin{proof}
Applying It\^{o}'s formula, for any $t\in[0,T]$ we have
\begin{eqnarray}\label{es241}
&&\!\!\!\!\!\!\!\!\|X^{(n)}_t\|_{\mathbb{H}}^q
\nonumber \\
=&&\!\!\!\!\!\!\!\!\|X^{(n)}_0\|_{\mathbb{H}}^q+q\int_0^t\|X^{(n)}_s\|_{\mathbb{H}}^{q-2}\langle X^{(n)}_s,\pi_n\mathcal{B}(s,X^{(n)}_s,\mathscr{L}_{X^{(n)}_s})dW^{(n)}_s\rangle_{\mathbb{H}}
\nonumber \\
&&\!\!\!\!\!\!\!\!
\frac{q(q-2)}{2}  \int_0^t\|X^{(n)}_s\|_{\mathbb{H}}^{q-4}
\cdot\|\big(\pi_n \mathcal{B}(s,X^{(n)}_s,\mathscr{L}_{X^{(n)}_s})\tilde{\pi}_n\big)^*X^{(n)}_s\|_{U}^2ds
\nonumber \\
&&\!\!\!\!\!\!\!\!+\frac{q}{2}\int_0^t\|X^{(n)}_s\|_{\mathbb{H}}^{q-2}\big(2{}_{\mathbb{V}^*}\langle \pi_n\mathcal{A}(s,X^{(n)}_s,\mathscr{L}_{X^{(n)}_s}),X^{(n)}_s\rangle_{\mathbb{V}}
\nonumber \\
&&\!\!\!\!\!\!\!\!
+\|\pi_n\mathcal{B}(s,X^{(n)}_s,\mathscr{L}_{X^{(n)}_s})\tilde{\pi}_n\|_{L_2(U,\mathbb{H})}^2\big)ds
\nonumber \\
=:&&\!\!\!\!\!\!\!\!\|X^{(n)}_0\|_{\mathbb{H}}^q+q\int_0^t\|X^{(n)}_s\|_{\mathbb{H}}^{q-2}\langle X^{(n)}_s,\pi_n\mathcal{B}(s,X^{(n)}_s,\mathscr{L}_{X^{(n)}_s})dW^{(n)}_s\rangle_{\mathbb{H}}+I+II.
\end{eqnarray}
For the term $I$, by assumption $\mathbf{(A_3)}$ it follows that
 \begin{eqnarray}\label{es36}
I\leq&&\!\!\!\!\!\!\!\!C_q\int_0^t\|X^{(n)}_s\|_{\mathbb{H}}^{q-2}\big(1+\|X^{(n)}_s\|_{\mathbb{H}}^2+\mathbb{E}^{\mathbf{P}_n}\|X^{(n)}_s\|_{\mathbb{H}}^2\big)ds
\nonumber \\
\leq&&\!\!\!\!\!\!\!\!C_q\int_0^t\Big(1+\|X^{(n)}_s\|_{\mathbb{H}}^q+(\mathbb{E}^{\mathbf{P}_n}\|X^{(n)}_s\|_{\mathbb{H}}^2)^{\frac{q}{2}}\Big)ds
\nonumber \\
\leq&&\!\!\!\!\!\!\!\!C_q\int_0^t\Big(1+\|X^{(n)}_s\|_{\mathbb{H}}^q+\mathbb{E}^{\mathbf{P}_n}\|X^{(n)}_s\|_{\mathbb{H}}^q\Big)ds.
\end{eqnarray}
For the term $II$, by assumption $\mathbf{(A_3)}$  and using similar arguments as in (\ref{es36}), we can get
 \begin{eqnarray}\label{es37}
II\leq&&\!\!\!\!\!\!\!\!-\frac{q\delta_1}{2}\int_0^t\|X^{(n)}_s\|_{\mathbb{H}}^{q-2}\|X^{(n)}_s\|_{\mathbb{V}}^{\alpha}ds
\nonumber \\
&&\!\!\!\!\!\!\!\!
+C_q\int_0^t\|X^{(n)}_s\|_{\mathbb{H}}^{q-2}\big(1+\|X^{(n)}_s\|_{\mathbb{H}}^2+\mathbb{E}^{\mathbf{P}_n}\|X^{(n)}_s\|_{\mathbb{H}}^2\big)ds
\nonumber \\
\leq&&\!\!\!\!\!\!\!\!-\frac{q\delta_1}{2}\int_0^t\|X^{(n)}_s\|_{\mathbb{H}}^{q-2}\|X^{(n)}_s\|_{\mathbb{V}}^{\alpha}ds+C_q\int_0^t\Big(1+\|X^{(n)}_s\|_{\mathbb{H}}^q+\mathbb{E}^{\mathbf{P}_n}\|X^{(n)}_s\|_{\mathbb{H}}^q\Big)ds.
\end{eqnarray}

For any $n\in\mathbb{N}$, we set a stopping time
$$\tau_R^{(n)}:=\inf\Big\{t\in[0,T]:\|X^{(n)}_t\|_{\mathbb{H}}> R\Big\}\wedge T,~R>0,$$
where we take $\inf{\emptyset}=\infty$.  It's easy to see that
$$\lim_{R\to\infty}\tau_R^{(n)}=T,~\mathbf{P}_n\text{-a.s.},~n\in\mathbb{N}.$$
Using  Burkholder-Davis-Gundy's inequality, we obtain
\begin{eqnarray}\label{es251}
&&\!\!\!\!\!\!\!\!q\mathbb{E}^{\mathbf{P}_n}\Bigg\{\sup_{t\in[0,T\wedge \tau_{R}^{(n)}]}\Big|\int_0^t\|X^{(n)}_s\|_{\mathbb{H}}^{q-2}\langle X^{(n)}_s,\pi_n\mathcal{B}(s,X^{(n)}_s,\mathscr{L}_{X^{(n)}_s})dW^{(n)}_s\rangle_{\mathbb{H}}\Big|\Bigg\}
\nonumber \\
&&\!\!\!\!\!\!\!\!\leq C_q\mathbb{E}^{\mathbf{P}_n}\Bigg\{\int_0^{T\wedge\tau_{R}^{(n)}}\|X^{(n)}_s\|_{\mathbb{H}}^{2q-2}\|\mathcal{B}(s,X^{(n)}_s,\mathscr{L}_{X^{(n)}_s})\|_{L_2(U,\mathbb{H})}^2ds\Bigg\}^{\frac{1}{2}}
\nonumber \\
&&\!\!\!\!\!\!\!\!\leq C_q\mathbb{E}^{\mathbf{P}_n}\Bigg\{\int_0^{T\wedge\tau_{R}^{(n)}}\big(\|X^{(n)}_t\|_{\mathbb{H}}^q+\mathbb{E}^{\mathbf{P}_n}\|X^{(n)}_t\|_{\mathbb{H}}^q+1\big)dt\Bigg\}
\nonumber \\
&&\!\!\!\!\!\!\!\!
+\frac{1}{2}\mathbb{E}^{\mathbf{P}_n}\Big[\sup_{t\in[0,T\wedge\tau_{R}^{(n)}]} \|X^{(n)}_t\|_{\mathbb{H}}^q\Big].
\end{eqnarray}
Combining (\ref{es241})-(\ref{es251})  we have
\begin{eqnarray}\label{es261}
&&\!\!\!\!\!\!\!\!\mathbb{E}^{\mathbf{P}_n}\Big[\sup_{t\in[0,T\wedge\tau_{R}^{(n)}]}\|X^{(n)}_t\|_{\mathbb{H}}^q\Big]+q\delta_1\mathbb{E}^{\mathbf{P}_n}\int_0^{T\wedge\tau_{R}^{(n)}}\|X^{(n)}_t\|_{\mathbb{H}}^{q-2}\|X^{(n)}_t\|_{\mathbb{V}}^{\alpha}dt
\nonumber \\
&&\!\!\!\!\!\!\!\!\leq C_{q,T}(1+\mathbb{E}^{\mathbf{P}_n}\|X^{(n)}_0\|_{\mathbb{H}}^q)+C_q\mathbb{E}^{\mathbf{P}_n}\int_0^{T}\|X^{(n)}_t\|_{\mathbb{H}}^qdt.
\end{eqnarray}
Taking $R\to \infty$ on both sides of (\ref{es261}) and then applying the monotone convergence theorem and Gronwall's lemma, it leads to
$$\mathbb{E}^{\mathbf{P}_n}\Big[\sup_{t\in[0,T]}\|X^{(n)}_t\|_{\mathbb{H}}^q\Big]+q\delta_1\mathbb{E}^{\mathbf{P}_n}\int_0^{T}\|X^{(n)}_t\|_{\mathbb{H}}^{q-2}\|X^{(n)}_t\|_{\mathbb{V}}^{\alpha}dt\leq C_{q,T}\big(1+\mathbb{E}\|\xi\|_{\mathbb{H}}^q\big).$$
We complete the proof.
\end{proof}

\begin{lemma}\label{lem5}
Suppose that the assumptions in Theorem \ref{th1} hold. For each $r>1$ and  $n\in\mathbb{N}$,
\begin{eqnarray}\label{es83}
\mathbb{E}^{\mathbf{P}_n}\Bigg\{\int_0^T\|X^{(n)}_t\|_{\mathbb{V}}^{\alpha}dt\Bigg\}^r
\leq C_{r,T}\big(1+\mathbb{E}\|\xi\|_{\mathbb{H}}^{2r}\big).
\end{eqnarray}
\end{lemma}

\begin{proof}
Applying It\^{o}'s formula we have
\begin{eqnarray*}
&&\!\!\!\!\!\!\!\!\|X^{(n)}_t\|_{\mathbb{H}}^2
\nonumber \\
=&&\!\!\!\!\!\!\!\!\|X^{(n)}_0\|_{\mathbb{H}}^2
+\int_0^t\big(2{}_{\mathbb{V}^*}\langle \pi_n\mathcal{A}(s,X^{(n)}_s,\mathscr{L}_{X^{(n)}_s}),X^{(n)}_s\rangle_{\mathbb{V}}
\nonumber \\
&&\!\!\!\!\!\!\!\!
+\|\pi_n\mathcal{B}(s,X^{(n)}_s,\mathscr{L}_{X^{(n)}_s})\tilde{\pi}_n\|_{L_2(U,\mathbb{H})}^2\big)ds
\nonumber \\
&&\!\!\!\!\!\!\!\!+2\int_0^t\langle X^{(n)}_s,\pi_n\mathcal{B}(s,X^{(n)}_s,\mathscr{L}_{X^{(n)}_s})dW^{(n)}_s\rangle_{\mathbb{H}}
\nonumber \\
\leq&&\!\!\!\!\!\!\!\!\|X^{(n)}_0\|_{\mathbb{H}}^2-\delta_1\int_0^t\|X^{(n)}_s\|_{\mathbb{V}}^{\alpha}ds
\nonumber \\
&&\!\!\!\!\!\!\!\!
+C\int_0^t\big(\|X^{(n)}_s\|_{\mathbb{H}}^2+\mathbb{E}^{\mathbf{P}_n}\|X^{(n)}_s\|_{\mathbb{H}}^2+1\big)ds
\nonumber \\
&&\!\!\!\!\!\!\!\!+2\int_0^t\langle X^{(n)}_s,\pi_n\mathcal{B}(s,X^{(n)}_s,\mathscr{L}_{X^{(n)}_s})dW^{(n)}_s\rangle_{\mathbb{H}}.
\end{eqnarray*}
Then we obtain
\begin{eqnarray*}
&&\!\!\!\!\!\!\!\!\|X^{(n)}_t\|_{\mathbb{H}}^2+\delta_1\int_0^t\|X^{(n)}_s\|_{\mathbb{V}}^{\alpha}ds
\nonumber \\
\leq&&\!\!\!\!\!\!\!\!\|X^{(n)}_0\|_{\mathbb{H}}^2
+C\int_0^t\big(\|X^{(n)}_s\|_{\mathbb{H}}^2+\mathbb{E}^{\mathbf{P}_n}\|X^{(n)}_s\|_{\mathbb{H}}^2+1\big)ds
\nonumber \\
&&\!\!\!\!\!\!\!\!+2\int_0^t\langle X^{(n)}_s,\pi_n\mathcal{B}(s,X^{(n)}_s,\mathscr{L}_{X^{(n)}_s})dW^{(n)}_s\rangle_{\mathbb{H}}.
\end{eqnarray*}
Thus for each $r>1$,
\begin{eqnarray}\label{es81}
&&\!\!\!\!\!\!\!\!\mathbb{E}^{\mathbf{P}_n}\Bigg\{\int_0^T\|X^{(n)}_s\|_{\mathbb{V}}^{\alpha}ds\Bigg\}^r
\nonumber \\
\leq&&\!\!\!\!\!\!\!\!C_{r}\mathbb{E}^{\mathbf{P}_n}\|X^{(n)}_0\|_{\mathbb{H}}^{2r}
+C_r\mathbb{E}^{\mathbf{P}_n}\Bigg\{\int_0^T\big(\|X^{(n)}_s\|_{\mathbb{H}}^2+\mathbb{E}^{\mathbf{P}_n}\|X^{(n)}_s\|_{\mathbb{H}}^2+1\big)ds\Bigg\}^r
+C_r\text{(I)}
\nonumber \\
\leq&&\!\!\!\!\!\!\!\!C_{r}\mathbb{E}\|\xi\|_{\mathbb{H}}^{2r}+ C_{r,T}\int_0^T\mathbb{E}^{\mathbf{P}_n}\|X^{(n)}_s\|_{\mathbb{H}}^{2r}ds+C_r\text{(I)},
\end{eqnarray}
where
$$\text{(I)}:=\mathbb{E}^{\mathbf{P}_n}\Bigg\{\sup_{t\in[0,T]}\Big|\int_0^t\langle X^{(n)}_s,\pi_n\mathcal{B}(s,X^{(n)}_s,\mathscr{L}_{X^{(n)}_s})dW^{(n)}_s\rangle_{\mathbb{H}}\Big|\Bigg\}^r.  $$
Using Burkholder-Davis-Gundy's inequality, it follows that
\begin{eqnarray}\label{es82}
\text{(I)}\leq&&\!\!\!\!\!\!\!\!\mathbb{E}^{\mathbf{P}_n}\Bigg\{\int_0^T\|X^{(n)}_s\|_{\mathbb{H}}^2\|\pi_n\mathcal{B}(s,X^{(n)}_s,\mathscr{L}_{X^{(n)}_s})\tilde{\pi}_n\|_{L_2(U,\mathbb{H})}^2ds\Bigg\}^{\frac{r}{2}}
\nonumber \\
\leq&&\!\!\!\!\!\!\!\!\mathbb{E}^{\mathbf{P}_n}\Big[\sup_{s\in[0,T]}\|X^{(n)}_s\|_{\mathbb{H}}^{2r}\Big]+C_{r,T}\int_0^T\mathbb{E}^{\mathbf{P}_n}\|X^{(n)}_s\|_{\mathbb{H}}^{2r}ds.
\end{eqnarray}
Combining (\ref{es81}), (\ref{es82}) and Lemma \ref{lem3.0} implies that (\ref{es83}) holds.
\end{proof}

For any~$M>0$, set a stopping time $\tau_M^{(n)}$ given by
\begin{equation}\label{destop}
\tau_M^{(n)}:=\inf\Big\{t\geq0:\|X^{(n)}_t\|_{\mathbb{H}}>M\Big\}\wedge\inf\Big\{t\geq0:\int_0^t \|X^{(n)}_s\|_{\mathbb{V}}^{\alpha}ds>M\Big\}\wedge T.
\end{equation}

\vspace{1mm}
We proceed to derive the following improved bounds before stopping time, which are crucial for obtaining the tightness of $X^{(n)}$ below.
\begin{lemma}\label{sec3lem4}
Suppose that the assumptions in Theorem \ref{th1} hold. For any $M>0$ there exists a constant $C_{\lambda,M,T}>0$, where $\lambda$ is defined in $(\mathbf{A_4})$, such that for any $n\in\mathbb{N}$,
\begin{eqnarray*}
\mathbb{E}^{\mathbf{P}_n}\Big[\sup_{t\in[0,T\wedge\tau_M^{(n)}]}\|X^{(n)}_t\|_{\mathbb{V}}^2\Big]+\mathbb{E}^{\mathbf{P}_n}\int_0^{T\wedge\tau_M^{(n)}}\|X^{(n)}_t\|_{\mathbb{X}}^\gamma dt
\leq C_{\lambda,M,T}\big(1+\mathbb{E}\|\xi\|_{\mathbb{V}}^2\big).
\end{eqnarray*}
\end{lemma}
\begin{proof}
By It\^{o}'s formula, (\ref{es7}) and $\mathbf{(A_4)}$, it follows that for any $t\in[0, T\wedge\tau_M^{(n)}]$,
\begin{eqnarray}\label{esq24}
&&\!\!\!\!\!\!\!\!\|X^{(n)}_t\|_{\mathbb{V}}^2
\nonumber \\
=&&\!\!\!\!\!\!\!\!\|X^{(n)}_0\|_{\mathbb{V}}^2+\int_0^t\big(2\langle \pi_n\mathcal{A}(s,X^{(n)}_s,\mathscr{L}_{X^{(n)}_s}),X^{(n)}_s\rangle_{\mathbb{V}}
\nonumber \\
&&\!\!\!\!\!\!\!\!
+\|\pi_n\mathcal{B}(s,X^{(n)}_s,\mathscr{L}_{X^{(n)}_s})\tilde{\pi}_n\|_{L_2(U,{\mathbb{V}})}^2\big)ds
\nonumber \\
&&\!\!\!\!\!\!\!\!+2\int_0^t\langle X^{(n)}_s,\pi_n\mathcal{B}(s,X^{(n)}_s,\mathscr{L}_{X^{(n)}_s})dW^{(n)}_s\rangle_{\mathbb{V}}
\nonumber \\
\leq&&\!\!\!\!\!\!\!\!\|X^{(n)}_r\|_{\mathbb{V}}^2-\delta_2\int_0^t\|X^{(n)}_s\|_{\mathbb{X}}^{\gamma}ds
\nonumber \\
&&\!\!\!\!\!\!\!\!
+2\int_0^t\langle X^{(n)}_s,\pi_n\mathcal{B}(s,X^{(n)}_s,\mathscr{L}_{X^{(n)}_s})dW^{(n)}_s\rangle_{\mathbb{V}}
\nonumber \\
&&\!\!\!\!\!\!\!\!+C\int_0^t\big\|X^{(n)}_s\|_{\mathbb{V}}^{\alpha+2}\|X^{(n)}_s\|_{\mathbb{H}}^{\lambda}ds
\nonumber \\
&&\!\!\!\!\!\!\!\!+C\int_0^t\big((1+\|X^{(n)}_s\|_{\mathbb{V}}^{\alpha}+\mathbb{E}^{\mathbf{P}_n}\|X^{(n)}_s\|_{\mathbb{V}}^{\alpha})
\nonumber \\
&&\!\!\!\!\!\!\!\!
~~~\cdot(1+\|X^{(n)}_s\|_{\mathbb{V}}^{2}
+\|X^{(n)}_s\|_{\mathbb{H}}^{\lambda}+\mathbb{E}^{\mathbf{P}_n}\|X^{(n)}_s\|_{\mathbb{H}}^{\lambda})\big)ds.
\end{eqnarray}
It follows from (\ref{esq24})  that
\begin{eqnarray*}
&&\!\!\!\!\!\!\!\! \|X^{(n)}_t\|_{\mathbb{V}}^2+\delta_2\int_0^{t}\|X^{(n)}_s\|_{\mathbb{X}}^{\gamma}ds
\nonumber \\
\leq&&\!\!\!\!\!\!\!\! C_T(1+\|X_0^{(n)}\|_{\mathbb{V}}^2)+2\int_0^t\langle X^{(n)}_s,\pi_n\mathcal{B}(s,X^{(n)}_s,\mathscr{L}_{X^{(n)}_s})dW^{(n)}_s\rangle_{\mathbb{V}}
\nonumber \\
&&\!\!\!\!\!\!\!\!+C\int_0^t\big\|X^{(n)}_s\|_{\mathbb{V}}^{\alpha+2}\|X^{(n)}_s\|_{\mathbb{H}}^{\lambda}ds
\nonumber \\
&&\!\!\!\!\!\!\!\!+C\int_0^{t}(\|X^{(n)}_s\|_{\mathbb{V}}^{2}
+\|X^{(n)}_s\|_{\mathbb{H}}^{\lambda}+\mathbb{E}^{\mathbf{P}_n}\|X^{(n)}_s\|_{\mathbb{H}}^{\lambda})ds
\nonumber \\
&&\!\!\!\!\!\!\!\!+C\int_0^{t}(\|X^{(n)}_s\|_{\mathbb{V}}^{\alpha+2}+\|X^{(n)}_s\|_{\mathbb{V}}^{\alpha}\|X^{(n)}_s\|_{\mathbb{H}}^{\lambda}+
\|X^{(n)}_s\|_{\mathbb{V}}^{\alpha} \cdot\mathbb{E}^{\mathbf{P}_n}\|X^{(n)}_s\|_{\mathbb{H}}^{\lambda}) ds
\nonumber \\
&&\!\!\!\!\!\!\!\!
+C\int_0^{t}(\mathbb{E}^{\mathbf{P}_n}\|X^{(n)}_s\|_{\mathbb{V}}^{\alpha}+\|X^{(n)}_s\|_{\mathbb{V}}^{2}\cdot\mathbb{E}^{\mathbf{P}_n}\|X^{(n)}_s\|_{\mathbb{V}}^{\alpha}
\nonumber \\
&&\!\!\!\!\!\!\!\!
+\|X^{(n)}_s\|_{\mathbb{H}}^{\lambda}\cdot\mathbb{E}^{\mathbf{P}_n}\|X^{(n)}_s\|_{\mathbb{V}}^{\alpha}+\mathbb{E}^{\mathbf{P}_n}\|X^{(n)}_s\|_{\mathbb{H}}^{\lambda}\cdot\mathbb{E}^{\mathbf{P}_n}\|X^{(n)}_s\|_{\mathbb{V}}^{\alpha})ds
\nonumber \\
\leq&&\!\!\!\!\!\!\!\! C_{\lambda,M,T}(1+\|X_0^{(n)}\|_{\mathbb{V}}^2)+2\Big|\int_0^t\langle X^{(n)}_s,\pi_n\mathcal{B}(s,X^{(n)}_s,\mathscr{L}_{X^{(n)}_s})dW^{(n)}_s\rangle_{\mathbb{V}}\Big|
\nonumber \\
&&\!\!\!\!\!\!\!\!
+C_M\int_0^{t}\|X^{(n)}_s\|_{\mathbb{V}}^{\alpha+2}ds+C_M\int_0^{t}(\mathbb{E}^{\mathbf{P}_n}\|X^{(n)}_s\|_{\mathbb{V}}^{\alpha}+\|X^{(n)}_s\|_{\mathbb{V}}^{2}\cdot\mathbb{E}^{\mathbf{P}_n}\|X^{(n)}_s\|_{\mathbb{V}}^{\alpha})ds,
\end{eqnarray*}
where we have used the definition of $\tau_{M}^{(n)}$ and  Lemma \ref{lem3.0} and Young's inequality in the last step.

Then we can get
\begin{eqnarray*}
&&\!\!\!\!\!\!\!\! \|X^{(n)}_t\|_{\mathbb{V}}^2+\delta_2\int_0^{t}\|X^{(n)}_s\|_{\mathbb{X}}^{\gamma}ds
\nonumber \\
\leq&& \!\!\!\!\!\!\!\! C_M\int_0^{t}\big(\|X^{(n)}_s\|_{\mathbb{V}}^{\alpha}+\mathbb{E}^{\mathbf{P}_n}\|X^{(n)}_s\|_{\mathbb{V}}^{\alpha}\big)\|X^{(n)}_s\|_{\mathbb{V}}^{2}ds
\nonumber \\
&&\!\!\!\!\!\!\!\!+C_{\lambda,M,T}
\Bigg\{1+\|X_0^{(n)}\|_{\mathbb{V}}^2+\Big|\int_0^t\langle X^{(n)}_s,\pi_n\mathcal{B}(s,X^{(n)}_s,\mathscr{L}_{X^{(n)}_s})dW^{(n)}_s\rangle_{\mathbb{V}}\Big|
\nonumber \\
&&\!\!\!\!\!\!\!\!
+\mathbb{E}^{\mathbf{P}_n}\int_0^{t}\|X^{(n)}_s\|_{\mathbb{V}}^{\alpha}ds\Bigg\}.
\end{eqnarray*}
Therefore, by  Gronwall's lemma we have  for any $t\in[0, T\wedge\tau_M^{(n)}]$,
\begin{eqnarray}\label{esq26}
&&\!\!\!\!\!\!\!\!\|X^{(n)}_t\|_{\mathbb{V}}^2+\delta_2\int_0^{t}\|X^{(n)}_t\|_{\mathbb{X}}^{\gamma}dt
\nonumber \\
\leq &&\!\!\!\!\!\!\!\! C_{\lambda,M,T}\Bigg\{1+\|X_0^{(n)}\|_{\mathbb{V}}^2+\Big|\int_0^t\langle X^{(n)}_s,\pi_n\mathcal{B}(s,X^{(n)}_s,\mathscr{L}_{X^{(n)}_s})dW^{(n)}_s\rangle_{\mathbb{V}}\Big|
\nonumber \\
&&\!\!\!\!\!\!\!\!+\mathbb{E}^{\mathbf{P}_n}\int_0^{T}\|X^{(n)}_s\|_{\mathbb{V}}^{\alpha}ds\Bigg\}\cdot \exp\Bigg\{C_M\int_0^{t}\|X^{(n)}_s\|_{\mathbb{V}}^{\alpha}ds+C_M\mathbb{E}^{\mathbf{P}_n}\int_0^{T}\|X^{(n)}_s\|_{\mathbb{V}}^{\alpha}ds\Bigg\}
\nonumber \\
\leq &&\!\!\!\!\!\!\!\!C_{\lambda,M,T}\Bigg\{1+\|X_0^{(n)}\|_{\mathbb{V}}^2+\Big|\int_0^t\langle X^{(n)}_s,\pi_n\mathcal{B}(s,X^{(n)}_s,\mathscr{L}_{X^{(n)}_s})dW^{(n)}_s\rangle_{\mathbb{V}}\Big|\Bigg\}.
\end{eqnarray}
Applying Burkholder-Davis-Gundy's inequality implies
\begin{eqnarray}\label{esq25}
&&\!\!\!\!\!\!\!\!\mathbb{E}^{\mathbf{P}_n}\Bigg\{\sup_{t\in[0, T\wedge\tau_M^{(n)}]}\Big|\int_{0}^t\langle X^{(n)}_s,\pi_n\mathcal{B}(s,X^{(n)}_s,\mathscr{L}_{X^{(n)}_s})dW^{(n)}_s\rangle_{\mathbb{V}}\Big|\Bigg\}
\nonumber \\
\leq &&\!\!\!\!\!\!\!\! C\mathbb{E}^{\mathbf{P}_n}\Bigg\{\int_{0}^{T\wedge\tau_M^{(n)}}\|X^{(n)}_t\|_{\mathbb{V}}^{2}\|\mathcal{B}(t,X^{(n)}_t,\mathscr{L}_{X^{(n)}_t})\|_{L_2(U,{\mathbb{V}})}^2dt\Bigg\}^{\frac{1}{2}}
\nonumber \\
\leq&&\!\!\!\!\!\!\!\! \mathbb{E}^{\mathbf{P}_n}\Bigg\{\sup_{t\in[0,T\wedge\tau_M^{(n)}]}\|X^{(n)}_t\|_{\mathbb{V}}^{2}\cdot C\int_{0}^{T\wedge\tau_M^{(n)}} \|\mathcal{B}(t,X^{(n)}_t,\mathscr{L}_{X^{(n)}_t})\|_{L_2(U,{\mathbb{V}})}^2dt\Bigg\}^{\frac{1}{2}}
\nonumber \\
\leq &&\!\!\!\!\!\!\!\! \frac{1}{2}\mathbb{E}^{\mathbf{P}_n}\Big[\sup_{t\in[0,T\wedge\tau_M^{(n)}]} \|X^{(n)}_t\|_{\mathbb{V}}^2\Big]
+C\mathbb{E}^{\mathbf{P}_n} \int_{0}^{T\wedge\tau_M^{(n)}} \|\mathcal{B}(t,X^{(n)}_t,\mathscr{L}_{X^{(n)}_t})\|_{L_2(U,{\mathbb{V}})}^2dt
\nonumber \\
\leq &&\!\!\!\!\!\!\!\! \frac{1}{2}\mathbb{E}^{\mathbf{P}_n}\Big[\sup_{t\in[0,T\wedge\tau_M^{(n)}]} \|X^{(n)}_t\|_{\mathbb{V}}^2\Big]
\nonumber \\
&&\!\!\!\!\!\!\!\!
+C\mathbb{E}^{\mathbf{P}_n}\Bigg\{\int_{0}^{T\wedge\tau_M^{(n)}}\|X^{(n)}_t\|_{\mathbb{V}}^{2}(1+\|X^{(n)}_t\|_{\mathbb{H}}^{\lambda}+\mathbb{E}^{\mathbf{P}_n}\|X^{(n)}_t\|_{\mathbb{H}}^{\lambda})
\nonumber \\
&&\!\!\!\!\!\!\!\!
+(1+\mathbb{E}^{\mathbf{P}_n}\|X^{(n)}_t\|_{\mathbb{V}}^{\alpha})(1+\|X^{(n)}_t\|_{\mathbb{H}}^{\lambda}+\mathbb{E}^{\mathbf{P}_n}\|X^{(n)}_t\|_{\mathbb{H}}^{\lambda})dt\Bigg\}
\nonumber \\
\leq &&\!\!\!\!\!\!\!\! C_{\lambda,M,T}
+\frac{1}{2}\mathbb{E}^{\mathbf{P}_n}\Big[\sup_{t\in[0,T\wedge\tau_M^{(n)}]} \|X^{(n)}_t\|_{\mathbb{V}}^2\Big]
\nonumber \\
&&\!\!\!\!\!\!\!\!+C_M\mathbb{E}^{\mathbf{P}_n}\int_{0}^{T\wedge\tau_M^{(n)}}\|X^{(n)}_t\|_{\mathbb{V}}^2dt+C_M\mathbb{E}^{\mathbf{P}_n}\int_{0}^{T\wedge\tau_M^{(n)}}\big(1+\mathbb{E}^{\mathbf{P}_n}\|X^{(n)}_t\|_{\mathbb{V}}^{\alpha}\big)dt
\nonumber \\
\leq &&\!\!\!\!\!\!\!\! C_{\lambda,M,T}
+\frac{1}{2}\mathbb{E}^{\mathbf{P}_n}\Big[\sup_{t\in[0,T\wedge\tau_M^{(n)}]} \|X^{(n)}_t\|_{\mathbb{V}}^2\Big]+C_M\mathbb{E}^{\mathbf{P}_n}\int_{0}^{T\wedge\tau_M^{(n)}}\|X^{(n)}_t\|_{\mathbb{V}}^2dt,
\end{eqnarray}
where we have used the definition of $\tau_{M}^{(n)}$ in the fifth step and  Lemma \ref{lem3.0}  in the fifth and last steps.

Combining (\ref{esq26})-(\ref{esq25}), it follows that
\begin{eqnarray*}
&&\!\!\!\!\!\!\!\!\mathbb{E}^{\mathbf{P}_n}\Big[\sup_{t\in[0,T\wedge\tau_M^{(n)}]}\|X^{(n)}_t\|_{\mathbb{V}}^2\Big]+2\delta_2\mathbb{E}^{\mathbf{P}_n}\int_{0}^{T\wedge\tau_M^{(n)}}\|X^{(n)}_t\|_{\mathbb{X}}^{\gamma}dt
\nonumber \\
\leq &&\!\!\!\!\!\!\!\! C_{\lambda,M,T}\mathbb{E}^{\mathbf{P}_n}\big(1+\|X_{0}^{(n)}\|_{\mathbb{V}}^2\big)+C_{\lambda,M,T}\mathbb{E}^{\mathbf{P}_n}\int_{0}^{T\wedge\tau_M^{(n)}}\|X^{(n)}_t\|_{\mathbb{V}}^2dt
\nonumber \\
\leq &&\!\!\!\!\!\!\!\! C_{\lambda,M,T}\mathbb{E}^{\mathbf{P}_n}\big(1+\|X_{0}^{(n)}\|_{\mathbb{V}}^2\big)+C_{\lambda,M,T}\int_{0}^{T}\mathbb{E}^{\mathbf{P}_n}\Big[\sup_{s\in[0,t\wedge\tau_M^{(n)}]}\|X^{(n)}_s\|_{\mathbb{V}}^2\Big]dt.
\end{eqnarray*}
Hence, by  Gronwall's lemma we obtain
$$\mathbb{E}^{\mathbf{P}_n}\Big[\sup_{t\in[0,T\wedge\tau_{M}^{(n)}]}\|X^{(n)}_t\|_{\mathbb{V}}^2\Big]+2\delta_2\mathbb{E}^{\mathbf{P}_n}\int_0^{T\wedge\tau_{M}^{(n)}}\|X^{(n)}_t\|_{\mathbb{X}}^{\gamma}dt
\leq C_{\lambda,M,T}\big(1+\mathbb{E}\|\xi\|_{\mathbb{V}}^2\big).$$
We complete the proof.
\end{proof}

\subsection{Tightness}\label{sub5.2}
 Note that $\Omega^{n}=C([0,T];\mathbb{H}_n)$ is a closed subset of $\Omega$. We could extend $\mathbf{P}_n$ to a probability measure $\hat{\mathbf{P}}_n$ on $(\Omega,\mathscr{F})$ by
$$\hat{\mathbf{P}}_n(A):=\mathbf{P}_n(A\cap \Omega^{n}),~A\in\mathscr{F}.$$

Throughout the section, the generic point in $\Omega$ is denoted by  $X$.
Recall the space (\ref{es73}) and the associated metric.

In this subsection, we intend to show the tightness of  $\hat{\mathbf{P}}_n$  in $\mathscr{P}(\mathbb{S})$. Firstly, we shall prove the tightness of  $\hat{\mathbf{P}}_n$ in $\mathscr{P}(\mathbb{C}_T)$ by the Aldous criterion \cite{A1}. Secondly, we  shall prove the tightness of  $\hat{\mathbf{P}}_n$ in $\mathscr{P}(L^\alpha([0,T];\mathbb{V}))$ by the  criterion presented in \cite[Theorem 5]{S3} (see also \cite[Lemma 5.2]{RSZ1}).
Then we conclude this subsection by combining Lemmas \ref{lem6} and \ref{lem7} to obtain the tightness of  $\hat{\mathbf{P}}_n$  in $\mathscr{P}(\mathbb{S})$.
\begin{lemma}\label{lem6}
The set of measures $\{\hat{\mathbf{P}}_n\}_{n\geq 1}$ is tight in $\mathscr{P}(\mathbb{C}_T)$.
\end{lemma}
\begin{proof}
Recall the definition (\ref{destop}) of stopping time $\tau_M^{(n)}$. By Lemma \ref{lem3.0}, we have
 \begin{eqnarray}\label{v110}
&&\sup_{n\in\mathbb{N}}\hat{\mathbf{P}}_n\big(\tau_M^{(n)}<T\big)
\nonumber \\
= &&\!\!\!\!\!\!\!\!
\sup_{n\in\mathbb{N}}\hat{\mathbf{P}}_n\left(\Bigg\{\int_0^T\|X_s\|_{\mathbb{V}}^{\alpha}ds>M\Bigg\}\cup\Big\{\sup_{0\leq s\leq T}\|X_s\|_{\mathbb{H}}>M\Big\}\right)
\nonumber \\
\leq &&\!\!\!\!\!\!\!\!
\sup_{n\in\mathbb{N}}\mathbf{P}_n\left(\int_0^T\|X^{(n)}_s\|_{\mathbb{V}}^{\alpha}ds>M\right)
+\sup_{n\in\mathbb{N}}\mathbf{P}_n\left(\sup_{0\leq s\leq T}\|X^{(n)}_s\|_{\mathbb{H}}>M\right)
\nonumber \\
\leq &&\!\!\!\!\!\!\!\!
\frac{1}{M}\sup_{n\in\mathbb{N}}\mathbb{E}^{\mathbf{P}_n}\int_0^T\|X^{(n)}_s\|_{\mathbb{V}}^{\alpha}ds
+\frac{1}{M^2}\sup_{n\in\mathbb{N}}\mathbb{E}^{\mathbf{P}_n}\Big[\sup_{0\leq s\leq T}\|X^{(n)}_s\|_{\mathbb{H}}^2\Big]
\nonumber \\
\leq &&\!\!\!\!\!\!\!\!C_T\Big(\frac{1}{M}+\frac{1}{M^2}\Big),
\end{eqnarray}
where the constant $C_{T}$ is independent of $M$.

\vspace{1mm}
For any $R>0$, by Lemma \ref{sec3lem4} and (\ref{v110}),we have
  \begin{eqnarray}\label{v111}
&&\!\!\!\!\!\!\!\!\sup_{n\in\mathbb{N}}\hat{\mathbf{P}}_n\Big(\sup_{0\leq t\leq T}\|X_t\|_{\mathbb{V}}>R\Big)
\nonumber \\
\leq &&\!\!\!\!\!\!\!\!
\sup_{n\in\mathbb{N}}\mathbf{P}_n\Big(\sup_{0\leq t\leq T}\|X^{(n)}_t\|_{\mathbb{V}}>R,\tau_M^{(n)}\geq T\Big)
\nonumber \\
 &&\!\!\!\!\!\!\!\!+\sup_{n\in\mathbb{N}}\mathbf{P}_n\Big(\sup_{0\leq t\leq T}\|X^{(n)}_t\|_{\mathbb{V}}>R,\tau_M^{(n)}<T\Big)
\nonumber \\
\leq &&\!\!\!\!\!\!\!\!
\sup_{n\in\mathbb{N}}\mathbf{P}_n\Big(\sup_{0\leq t\leq T\wedge\tau_M^{(n)}}\|X^{(n)}_t\|_{\mathbb{V}}>R\Big)+\sup_{n\in\mathbb{N}}\mathbf{P}_n\big(\tau_M^{(n)}<T\big)
\nonumber \\
\leq &&\!\!\!\!\!\!\!\!\frac{1}{R^2}\sup_{n\in\mathbb{N}}\mathbb{E}^{\mathbf{P}_n}\Big[\sup_{0\leq t\leq T\wedge\tau_M^{(n)}}\|X^{(n)}_t\|_{\mathbb{V}}^2\Big]+C_T\Big(\frac{1}{M}+\frac{1}{M^2}\Big)
\nonumber \\
\leq &&\!\!\!\!\!\!\!\!\frac{C_M}{R^2}+C_T\Big(\frac{1}{M}+\frac{1}{M^2}\Big).
\end{eqnarray}
Let $R\rightarrow\infty$ first, then let $M\rightarrow\infty$, it follows  that
$$\lim_{R\rightarrow\infty}\sup_{n\in\mathbb{N}}\hat{\mathbf{P}}_n\left(\sup_{0\leq t\leq T}\|X_t\|_{\mathbb{V}}>R\right)=0.$$

Since the embedding $\mathbb{V}\subset \mathbb{H}$ is compact, by \cite[Theorem 3.1]{Ja86}, it suffices to
show that for every $e\in \mathbb{H}_m$, $m\geq 1$, $\{\langle X^{(n)},e\rangle_{\mathbb{H}}\}_{n\in\mathbb{N}}$ is tight in the space $C([0,T],\mathbb{R})$. Note that by Aldous's criterion \cite{A1} and (\ref{v111}), it remains to prove that for any stopping time $0\leq \zeta_n\leq T$ and $\varepsilon>0$,
\begin{equation}\label{escon1}
\lim_{\Delta\to 0}\sup_{n\in\mathbb{N}}\hat{\mathbf{P}}_n\Big(|\langle X_{\overline{\zeta_n+\Delta}}-X_{\zeta_n},e\rangle_{\mathbb{H}}|>\varepsilon\Big)=0,
\end{equation}
where $\overline{\zeta_n+\Delta}:=T\wedge (\zeta_n+\Delta) \vee 0$. Similar to (\ref{v111}), we know
\begin{eqnarray}\label{sec3es2}
&&\!\!\!\!\!\!\!\!\hat{\mathbf{P}}_n\Big(|\langle X_{\overline{\zeta_n+\Delta}}-X_{\zeta_n},e\rangle_{\mathbb{H}}|>\varepsilon\Big)
\nonumber \\
\leq &&\!\!\!\!\!\!\!\!\mathbf{P}_n\Big(|\langle X^{(n)}_{\overline{\zeta_n+\Delta}}-X^{(n)}_{\zeta_n},e\rangle_{\mathbb{H}}|>\varepsilon,\tau_M^{(n)}\geq T\Big)+\mathbf{P}_n\big(\tau_M^{(n)}<T\big)
\nonumber \\
\leq &&\!\!\!\!\!\!\!\!\frac{1}{\varepsilon^{\frac{\alpha}{\alpha-1}}}\mathbb{E}^{\mathbf{P}_n}\big|\langle X^{(n)}_{\overline{\zeta_n+\Delta}\wedge  \tau_M^{(n)} }-X^{(n)}_{\zeta_n\wedge \tau_M^{(n)}},e\rangle_{\mathbb{H}}\big|^{\frac{\alpha}{\alpha-1}}+\mathbf{P}_n\big(\tau_M^{(n)}<T\big).
\end{eqnarray}

Now we estimate the first term on the right hand side of (\ref{sec3es2}). By (\ref{eqf}) and applying Burkholder-Davis-Gundy's inequality, it follows that
\begin{eqnarray}\label{esq1}
&&\!\!\!\!\!\!\!\!\mathbb{E}^{\mathbf{P}_n}\big|\langle X^{(n)}_{\overline{\zeta_n+\Delta}\wedge  \tau_M^{(n)} }-X^{(n)}_{\zeta_n\wedge \tau_M^{(n)}},e\rangle_{\mathbb{H}}\big|^{\frac{\alpha}{\alpha-1}}
\nonumber \\
\leq &&\!\!\!\!\!\!\!\!C\mathbb{E}^{\mathbf{P}_n}\Bigg\{\int_{\zeta_n\wedge \tau_M^{(n)}}^{\overline{\zeta_n+\Delta}\wedge  \tau_M^{(n)}}|\langle \pi_n \mathcal{A}(s,X^{(n)}_s,\mathscr{L}_{X^{(n)}_s}), e\rangle_{\mathbb{H}}|ds\Bigg\}^{\frac{\alpha}{\alpha-1}}
\nonumber \\
 &&\!\!\!\!\!\!\!\!+C\mathbb{E}^{\mathbf{P}_n}\Bigg\{\int_{\zeta_n\wedge \tau_M^{(n)}}^{\overline{\zeta_n+\Delta}\wedge  \tau_M^{(n)}}\| \pi_n \mathcal{B}(s,X^{(n)}_s,\mathscr{L}_{X^{(n)}_s})\tilde{\pi}_n\|_{L_2(U, \mathbb{H})}^2\|e\|_{\mathbb{H}}^2ds\Bigg\}^{\frac{\alpha}{2(\alpha-1)}}
\nonumber \\
= &&\!\!\!\!\!\!\!\!C\mathbb{E}^{\mathbf{P}_n}\Bigg\{\int_{\zeta_n\wedge \tau_M^{(n)}}^{\overline{\zeta_n+\Delta}\wedge  \tau_M^{(n)}}|{}_{\mathbb{V}^*}\langle \pi_n \mathcal{A}(s,X^{(n)}_s,\mathscr{L}_{X^{(n)}_s}), e\rangle_{\mathbb{V}}|ds\Bigg\}^{\frac{\alpha}{\alpha-1}}
\nonumber \\
 &&\!\!\!\!\!\!\!\!+C\mathbb{E}^{\mathbf{P}_n}\Bigg\{\int_{\zeta_n\wedge \tau_M^{(n)}}^{\overline{\zeta_n+\Delta}\wedge  \tau_M^{(n)}}\| \pi_n \mathcal{B}(s,X^{(n)}_s,\mathscr{L}_{X^{(n)}_s})\tilde{\pi}_n\|_{L_2(U, \mathbb{H})}^2\|e\|_{\mathbb{H}}^2ds\Bigg\}^{\frac{\alpha}{2(\alpha-1)}}
 \nonumber \\
=: &&\!\!\!\!\!\!\!\!\text{(I)}+\text{(II)},
\end{eqnarray}
where we used the fact that $e\in\mathbb{H}_m\subset\mathbb{V}$ in the second step.

\vspace{1mm}
We mention that by Lemma \ref{lem3.0}, it implies
\begin{equation}\label{es11}
\mathscr{L}_{X^{(n)}_t}\in\mathscr{P}_{\alpha}(\mathbb{V})\cap\mathscr{P}_{p}(\mathbb{H}),~dt\text{-a.e.}.
\end{equation}
Thus for $\text{(I)}$, by $\mathbf{(A_3)}$, (\ref{es11}) and Jensen's inequality, we have
\begin{eqnarray*}
\text{(I)}\leq  &&\!\!\!\!\!\!\!\!C |\Delta|^{\frac{1}{\alpha-1}}\cdot \mathbb{E}^{\mathbf{P}_n}\Bigg\{\int_{\zeta_n\wedge \tau_M^{(n)}}^{\overline{\zeta_n+\Delta}\wedge  \tau_M^{(n)}}|{}_{\mathbb{V}^*}\langle \pi_n \mathcal{A}(s,X^{(n)}_s,\mathscr{L}_{X^{(n)}_s}), e\rangle_{\mathbb{V}}|^{\frac{\alpha}{\alpha-1}}ds\Bigg\}
\nonumber \\
\leq &&\!\!\!\!\!\!\!\!C |\Delta|^{\frac{1}{\alpha-1}}\cdot\mathbb{E}^{\mathbf{P}_n}\Bigg\{\int_{0}^{T\wedge \tau_M^{(n)}}\|e\|_{\mathbb{V}}^{\frac{\alpha}{\alpha-1}}\big(1+\|X_{s}^{(n)}\|_{\mathbb{V}}^\alpha+\mathscr{L}_{X^{(n)}_s}(\|\cdot\|_{\mathbb{V}}^{\alpha})\big)
\nonumber \\
 &&\!\!\!\!\!\!\!\! \cdot
\big(1+\|X_{s}^{(n)}\|_{\mathbb{H}}^\beta+\mathscr{L}_{X^{(n)}_s}(\|\cdot\|_{\mathbb{H}}^{\beta})\big)ds\Bigg\}
\nonumber \\
\leq &&\!\!\!\!\!\!\!\!C_{M,T,\|e\|_{\mathbb{V}}} |\Delta|^{\frac{1}{\alpha-1}}\cdot\mathbb{E}^{\mathbf{P}_n}\Bigg\{\int_{0}^{T\wedge \tau_M^{(n)}}\big(1+\|X_{s}^{(n)}\|_{\mathbb{V}}^\alpha+\mathscr{L}_{X^{(n)}_s}(\|\cdot\|_{\mathbb{V}}^{\alpha})\big)ds\Bigg\}.
\end{eqnarray*}
Note that due to Lemma \ref{lem3.0}, we have
$$\mathbb{E}^{\mathbf{P}_n}\Bigg\{\int_{0}^{T\wedge \tau_M^{(n)}}\mathscr{L}_{X^{(n)}_s}(\|\cdot\|_{\mathbb{V}}^{\alpha})ds\Bigg\}\leq \mathbb{E}^{\mathbf{P}_n}\int_{0}^{T}\|X^{(n)}_s\|_{\mathbb{V}}^{\alpha}ds\leq C_{T}\big(1+\mathbb{E}\|\xi\|_{\mathbb{H}}^2\big).$$
Therefore,
\begin{equation}\label{esq2}
\text{(I)}
 \leq C_{M,T,\|e\|_{\mathbb{V}}} |\Delta|^{\frac{1}{\alpha-1}}\big(1+\mathbb{E}\|\xi\|_{\mathbb{H}}^p\big).
\end{equation}
By the growth assumption $(\mathbf{A_3})$, we can get
\begin{eqnarray}\label{esq3}
\text{(II)}\leq  &&\!\!\!\!\!\!\!\!\mathbb{E}^{\mathbf{P}_n}\Bigg\{\int_{\zeta_n\wedge \tau_M^{(n)}}^{(\zeta_n+\Delta)\wedge  \tau_M^{(n)}}\|  \mathcal{B}(s,X^{(n)}_s,\mathscr{L}_{X^{(n)}_s})\|_{L_2(U, \mathbb{H})}^2\|e\|_{\mathbb{H}}^2ds\Bigg\}^{\frac{\alpha}{2(\alpha-1)}}
\nonumber \\
\leq &&\!\!\!\!\!\!\!\!C\mathbb{E}^{\mathbf{P}_n}\Bigg\{\int_{\zeta_n\wedge \tau_M^{(n)}}^{(\zeta_n+\Delta)\wedge  \tau_M^{(n)}}\|e\|_{\mathbb{H}}^2\big(1+\|X_{s}^{(n)}\|_{\mathbb{H}}^2+\mathbb{E}^{\mathbf{P}_n}\|X_{s}^{(n)}\|_{\mathbb{H}}^2\big)ds\Bigg\}^{\frac{\alpha}{2(\alpha-1)}}
\nonumber \\
\leq  &&\!\!\!\!\!\!\!\!C_{M,\|e\|_{\mathbb{H}}}\mathbb{E}^{\mathbf{P}_n}\Bigg\{\int_{\zeta_n\wedge \tau_M^{(n)}}^{(\zeta_n+\Delta)\wedge  \tau_M^{(n)}}\big(1+\mathbb{E}^{\mathbf{P}_n}\|X_{s}^{(n)}\|_{\mathbb{H}}^2\big)ds\Bigg\}^{\frac{\alpha}{2(\alpha-1)}}
\nonumber \\
\leq &&\!\!\!\!\!\!\!\!C_{M,T,\|e\|_{\mathbb{H}}}\mathbb{E}^{\mathbf{P}_n}\Bigg\{\int_{\zeta_n\wedge \tau_M^{(n)}}^{(\zeta_n+\Delta)\wedge  \tau_M^{(n)}}\big(1+\mathbb{E}\|\xi\|_{\mathbb{H}}^2\big)ds\Bigg\}^{\frac{\alpha}{2(\alpha-1)}}
\nonumber \\
 \leq&&\!\!\!\!\!\!\!\!C_{M,T,\|e\|_{\mathbb{H}}}|\Delta|^{\frac{\alpha}{2(\alpha-1)}}\Big(1+\big(\mathbb{E}\|\xi\|_{\mathbb{H}}^2\big)^{\frac{\alpha}{2(\alpha-1)}}\Big),
\end{eqnarray}
where we used Lemma \ref{lem3.0} in the fourth step.

Combining (\ref{esq1})-(\ref{esq3}) gives
\begin{eqnarray}\label{esq4}
\lim_{\Delta\to 0}\sup_{n\in\mathbb{N}}\mathbb{E}^{\mathbf{P}_n}\big|\langle X^{(n)}_{\overline{\zeta_n+\Delta}\wedge  \tau_M^{(n)} }-X^{(n)}_{\zeta_n\wedge \tau_M^{(n)}},e\rangle_{\mathbb{H}}\big|^{\alpha}=0.
\end{eqnarray}
Finally, taking (\ref{v110}), (\ref{sec3es2}) and (\ref{esq4}) into account and  letting $\Delta\to 0$ and then $M\to\infty$ in (\ref{sec3es2}), we conclude (\ref{escon1}) holds.
The proof is completed.
\end{proof}

\begin{lemma}\label{lem7}
The set of measures $\{\hat{\mathbf{P}}_n\}_{n\in\mathbb{N}}$ is tight in $\mathscr{P}(L^\alpha([0,T];\mathbb{V}))$.
\end{lemma}
\begin{proof}
Since the embedding $\mathbb{X}\subset \mathbb{V}$ is compact, by the tightness criterion (cf. \cite[Theorem 5]{S3} and \cite[Lemma 5.2]{RSZ1}), we only need to show the following two claims hold.

(i)
\begin{eqnarray}\label{vv88}
\lim_{\kappa\rightarrow\infty}\sup_{n\in\mathbb{N}}\hat{\mathbf{P}}_n\left(\int_0^T\|X_t\|_{\mathbb{X}}^\alpha dt>\kappa\right)=0.
\end{eqnarray}

(ii) For any $\epsilon>0$,
\begin{eqnarray}\label{vv99}
\lim_{\Delta\rightarrow0^+}\sup_{n\in\mathbb{N}}\hat{\mathbf{P}}_n
\left(\int_0^{T}\|X_{t}-X_{t-\Delta}\|_{\mathbb{H}}^{\alpha}dt>\epsilon\right)=0.
\end{eqnarray}

We first prove claim (i). By Lemma \ref{sec3lem4}, we have
\begin{eqnarray*}
&&\!\!\!\!\!\!\!\!\sup_{n\in\mathbb{N}}\hat{\mathbf{P}}_n\Big(\int_0^T\|X_t\|_{\mathbb{X}}^\alpha dt>\kappa\Big)
\nonumber \\
= &&\!\!\!\!\!\!\!\!
\sup_{n\in\mathbb{N}}\mathbf{P}_n\Big(\int_0^T\|X^{(n)}_t\|_{\mathbb{X}}^\alpha dt>\kappa,\tau_M^{(n)}\geq T\Big)+\sup_{n\in\mathbb{N}}\mathbf{P}_n\Big(\int_0^T\|X^{(n)}_t\|_{\mathbb{X}}^\alpha dt>\kappa,\tau_M^{(n)}<T\Big)
\nonumber \\
\leq &&\!\!\!\!\!\!\!\!
\sup_{n\in\mathbb{N}}\mathbf{P}_n\Big(\int_0^{T\wedge\tau_M^{(n)}}\|X^{(n)}_t\|_{\mathbb{X}}^\alpha dt>\kappa\Big)+\sup_{n\in\mathbb{N}}\mathbf{P}_n\Big(\tau_M^{(n)}<T\Big)
\nonumber \\
\leq &&\!\!\!\!\!\!\!\!\frac{1}{\kappa}\sup_{n\in\mathbb{N}}\mathbb{E}^{\mathbf{P}_n}\int_0^{T\wedge\tau_M^{(n)}}(1+\|X^{(n)}_t\|_{\mathbb{X}}^{\gamma}) dt+C_T\Big(\frac{1}{M}+\frac{1}{M^2}\Big)
\nonumber \\
\leq &&\!\!\!\!\!\!\!\!\frac{C_M}{\kappa}+C_T\Big(\frac{1}{M}+\frac{1}{M^2}\Big).
\end{eqnarray*}
Let $\kappa\rightarrow\infty$ first and then  $M\rightarrow\infty$, we complete the proof of claim (i).

\vspace{2mm}
We now prove claim (ii). Note that
\begin{eqnarray}\label{esq5}
&&\!\!\!\!\!\!\!\!\hat{\mathbf{P}}_n
\left(\int_0^{T}\|X_{t}-X_{t-\Delta}\|_{\mathbb{H}}^{\alpha}dt>\epsilon\right)
\nonumber \\
\leq &&\!\!\!\!\!\!\!\!\mathbf{P}_n
\left(\int_0^{T}\|X^{(n)}_{t}-X^{(n)}_{t-\Delta}\|_{\mathbb{H}}^{\alpha}dt>\epsilon,\tau_M^{(n)}\geq T\right)+\mathbf{P}_n\big(\tau_M^{(n)}<T\big)
\nonumber \\
\leq &&\!\!\!\!\!\!\!\!\frac{1}{\epsilon}\mathbb{E}^{\mathbf{P}_n}
\Bigg\{\int_0^{T}\|X^{(n)}_{t\wedge \tau_M^{(n)}}-X^{(n)}_{(t-\Delta)\wedge \tau_M^{(n)}}\|_{\mathbb{H}}^{\alpha}dt\Bigg\}+\mathbf{P}_n\big(\tau_M^{(n)}<T\big).
\end{eqnarray}
If we can prove
\begin{equation}\label{esq6}
\lim_{\Delta\rightarrow0^+}\sup_{n\in\mathbb{N}}\mathbb{E}^{\mathbf{P}_n}
\Bigg\{\int_0^{T}\|X^{(n)}_{t\wedge \tau_M^{(n)}}-X^{(n)}_{(t-\Delta)\wedge \tau_M^{(n)}}\|_{\mathbb{H}}^{\alpha}dt\Bigg\}=0,
\end{equation}
then by (\ref{v110}), letting $\Delta\to 0^+$ and then $M\to \infty$ in (\ref{esq5}) yields claim (ii). The proof of (\ref{esq6}) is divided into two cases, i.e., $1<\alpha\leq 2$ and $\alpha> 2$.

\vspace{2mm}
\textbf{Case 1.} Let $1<\alpha\leq 2$.  Applying It\^{o}'s formula yields that
\begin{eqnarray}
&&\!\!\!\!\!\!\!\!\|X_{t\wedge \tau_M^{(n)}}^{(n)}-X_{(t-\Delta)\wedge \tau_M^{(n)}}^{(n)}\|_{\mathbb{H}}^{2}
 \nonumber \\
=&&\!\!\!\!\!\!\!\!2\int_{(t-\Delta)\wedge \tau_M^{(n)}} ^{t\wedge \tau_M^{(n)}}{}_{\mathbb{V}^*}\langle \pi_n \mathcal{A}(s,X^{(n)}_s,\mathscr{L}_{X^{(n)}_s}), X_{s}^{(n)}-X_{(t-\Delta)\wedge \tau_M^{(n)}}^{(n)}\rangle_{\mathbb{V}} ds
 \nonumber \\
 &&\!\!\!\!\!\!\!\!
 +\int_{(t-\Delta)\wedge \tau_M^{(n)}} ^{t\wedge \tau_M^{(n)}}\|\pi_n\mathcal{B}(s,X^{(n)}_s,\mathscr{L}_{X^{(n)}_s})\tilde{\pi}_n\|_{L_2(U,\mathbb{H})}^2ds
 \nonumber \\
 &&\!\!\!\!\!\!\!\!+2\int_{(t-\Delta)\wedge \tau_M^{(n)}} ^{t\wedge \tau_M^{(n)}}\langle \pi_n\mathcal{B}(s,X^{(n)}_s,\mathscr{L}_{X^{(n)}_s})dW^{(n)}_s, X_{s}^{(n)}-X_{(t-\Delta)\wedge \tau_M^{(n)}}^{(n)}\rangle_{\mathbb{H}} \nonumber\\
=:&&\!\!\!\!\!\!\!\!\text{(I)}+\text{(II)}+\text{(III)}.  \label{F3.9}
\end{eqnarray}
For the first term $\text{(I)}$, by $\mathbf{(A_3)}$ and (\ref{es11}), there exists a constant $C>0$ such that
\begin{eqnarray}  \label{REGX1}
&&\!\!\!\!\!\!\!\!\mathbb{E}^{\mathbf{P}_n}\int^{T}_{\Delta}\text{(I)}dt\nonumber\\
\leq&&\!\!\!\!\!\!\!\! C\mathbb{E}^{\mathbf{P}_n}\Bigg\{\int^{T}_{\Delta}\int_{(t-\Delta)\wedge \tau_M^{(n)}} ^{t\wedge \tau_M^{(n)}}\|\mathcal{A}(s,X^{(n)}_s,\mathscr{L}_{X^{(n)}_s})\|_{\mathbb{V}^*}
\|X_{s}^{(n)}-X_{(t-\Delta)\wedge \tau_M^{(n)}}^{(n)}\|_{\mathbb{V}} ds dt\Bigg\}\nonumber\\
\leq&&\!\!\!\!\!\!\!\!C\Bigg\{\mathbb{E}^{\mathbf{P}_n}\int^{T}_{\Delta}\int_{(t-\Delta)\wedge \tau_M^{(n)}} ^{t\wedge \tau_M^{(n)}}\|\mathcal{A}(s,X^{(n)}_s,\mathscr{L}_{X^{(n)}_s})\|_{\mathbb{V}^*}^{\frac{\alpha}{\alpha-1}}dsdt\Bigg\}^{\frac{\alpha-1}{\alpha}}
\nonumber\\
&&\!\!\!\!\!\!\!\!\cdot
\Bigg\{\mathbb{E}^{\mathbf{P}_n}\int^{T}_{\Delta}\int_{(t-\Delta)\wedge \tau_M^{(n)}} ^{t\wedge \tau_M^{(n)}}\|X_{s}^{(n)}-X_{(t-\Delta)\wedge \tau_M^{(n)}}^{(n)}\|^\alpha_{\mathbb{V}} dsdt\Bigg\}^{\frac{1}{\alpha}}\nonumber\\
\leq&&\!\!\!\!\!\!\!\!C\Bigg\{\Delta\mathbb{E}^{\mathbf{P}_n}\int^{T\wedge \tau_M^{(n)}}_0\big(1+\|X_{s}^{(n)}\|_{\mathbb{V}}^\alpha+\mathscr{L}_{X^{(n)}_s}(\|\cdot\|_{\mathbb{V}}^{\alpha})\big)
\nonumber\\
&&\!\!\!\!\!\!\!\!\cdot
\big(1+\|X_{s}^{(n)}\|_{\mathbb{H}}^\beta+\mathscr{L}_{X^{(n)}_s}(\|\cdot\|_{\mathbb{H}}^{\beta})\big)ds\Bigg\}^{\frac{\alpha-1}{\alpha}}
\cdot\Bigg\{\Delta\mathbb{E}^{\mathbf{P}_n}\int^{T\wedge \tau_M^{(n)}}_0\|X_{s}^{(n)}\|^\alpha_{\mathbb{V}}ds\Bigg\}^{\frac{1}{\alpha}}\nonumber\\
\leq&&\!\!\!\!\!\!\!\!C_{p,M,T}\Delta\big(1+\mathbb{E}\|\xi\|_{\mathbb{H}}^p\big),
\end{eqnarray}
where we used Lemma \ref{lem3.0} in the  fourth inequality.

\vspace{1mm}
Similarly, for $\text{(II)}$, by condition $\mathbf{(A_3)}$,  we get
\begin{eqnarray}\label{REGX2a}
\mathbb{E}^{\mathbf{P}_n}\int^{T}_{\Delta}\text{(II)}dt\leq&&\!\!\!\!\!\!\!\!C\mathbb{E}^{\mathbf{P}_n}\Bigg\{\int^{T}_{\Delta}\int_{(t-\Delta)\wedge \tau_M^{(n)}} ^{t\wedge \tau_M^{(n)}}\big(1+\|X_{s}^{(n)}\|_{\mathbb{H}}^2+\mathscr{L}_{X^{(n)}_s}(\|\cdot\|_{\mathbb{H}}^2)\big)ds dt\Bigg\}\nonumber\\
\leq&&\!\!\!\!\!\!\!\!C_T\Delta\cdot\mathbb{E}^{\mathbf{P}_n}\Big(\sup_{s\in[0,T]}\big(1+\|X_{s}^{(n)}\|_{\mathbb{H}}^2+\mathbb{E}^{\mathbf{P}_n}\|X_{s}^{(n)}\|_{\mathbb{H}}^2\big)\Big)\nonumber\\
\leq&&\!\!\!\!\!\!\!\!C_{T}\Delta\big(1+\mathbb{E}\|\xi\|_{\mathbb{H}}^2\big).
\end{eqnarray}
For $\text{(III)}$, due to Lemma \ref{lem3.0}, it is easy to see that
\begin{eqnarray}  \label{REGX3}
&&\!\!\!\!\!\!\!\!\mathbb{E}^{\mathbf{P}_n}\int^{T}_{\Delta}\text{(III)}dt
\nonumber\\
=&&\!\!\!\!\!\!\!\!\int^{T}_{\Delta}\mathbb{E}^{\mathbf{P}_n}\Bigg\{\int_{(t-\Delta)\wedge \tau_M^{(n)}} ^{t\wedge \tau_M^{(n)}}\langle \pi_n\mathcal{B}(s,X^{(n)}_s,\mathscr{L}_{X^{(n)}_s})dW^{(n)}_s,X_{s}^{(n)}-X_{(t-\Delta)\wedge \tau_M^{(n)}}^{(n)}\rangle_{\mathbb{H}}\Bigg\}dt \nonumber\\
=&&\!\!\!\!\!\!\!\!0.
\end{eqnarray}
Combining estimates \eref{F3.9}-\eref{REGX3}, we get that
\begin{eqnarray}
\mathbb{E}^{\mathbf{P}_n}\Bigg\{\int^{T}_{\Delta}\|X_{t\wedge \tau_M^{(n)}}^{(n)}-X_{(t-\Delta)\wedge \tau_M^{(n)}}^{(n)}\|_{\mathbb{H}}^2dt\Bigg\}\leq&&\!\!\!\!\!\!\!\!C_{p,M,T}\Delta(1+\mathbb{E}\|\xi\|_{\mathbb{H}}^p). \label{F3.13}
\end{eqnarray}
Note that due to (\ref{F3.13}),
\begin{eqnarray}\label{esq7}
&&\!\!\!\!\!\!\!\!\mathbb{E}^{\mathbf{P}_n}\Bigg\{\int^{T}_0\|X_{t\wedge \tau_M^{(n)}}^{(n)}-X_{(t-\Delta)\wedge \tau_M^{(n)}}^{(n)}\|_{\mathbb{H}}^2dt\Bigg\}\nonumber\\
=&&\!\!\!\!\!\!\!\! \mathbb{E}^{\mathbf{P}_n}\Bigg\{\int^{\Delta}_0\|X_{t\wedge \tau_M^{(n)}}^{(n)}-X_{(t-\Delta)\wedge \tau_M^{(n)}}^{(n)}\|_{\mathbb{H}}^2dt\Bigg\}
\nonumber \\
&&\!\!\!\!\!\!\!\!+\mathbb{E}^{\mathbf{P}_n}\Bigg\{\int^{T}_{\Delta}\|X_{t\wedge \tau_M^{(n)}}^{(n)}-X_{(t-\Delta)\wedge \tau_M^{(n)}}^{(n)}\|_{\mathbb{H}}^2dt\Bigg\}\nonumber\\
\leq&&\!\!\!\!\!\!\!\! C\big(1+\mathbb{E}\|\xi\|_{\mathbb{H}}^2\big)\Delta
+\mathbb{E}^{\mathbf{P}_n}\Bigg\{\int^{T}_{\Delta}\|X_{t\wedge \tau_M^{(n)}}^{(n)}-X_{(t-\Delta)\wedge \tau_M^{(n)}}^{(n)}\|_{\mathbb{H}}^2dt\Bigg\}
\nonumber\\
\leq&&\!\!\!\!\!\!\!\! C_{p,M,T}\Delta\big(1+\mathbb{E}\|\xi\|_{\mathbb{H}}^p\big).
\end{eqnarray}
Therefore, by H\"{o}lder's inequality, we have
\begin{eqnarray*}
&&\!\!\!\!\!\!\!\!\lim_{\Delta\rightarrow0^+}\sup_{n\in\mathbb{N}}\mathbb{E}^{\mathbf{P}_n}
\Bigg\{\int_0^{T}\|X^{(n)}_{t\wedge \tau_M^{(n)}}-X^{(n)}_{(t-\Delta)\wedge \tau_M^{(n)}}\|_{\mathbb{H}}^{\alpha}dt\Bigg\}
\nonumber\\
\leq&&\!\!\!\!\!\!\!\!C_T \Bigg\{\lim_{\Delta\rightarrow0^+}\sup_{n\in\mathbb{N}}\mathbb{E}^{\mathbf{P}_n}
\Bigg\{\int_0^{T}\|X^{(n)}_{t\wedge \tau_M^{(n)}}-X^{(n)}_{(t-\Delta)\wedge \tau_M^{(n)}}\|_{\mathbb{H}}^{2}dt\Bigg\}\Bigg\}^{\frac{\alpha}{2}}=0.
\end{eqnarray*}

\textbf{Case 2.} Let $\alpha> 2$.  Applying It\^{o}'s formula to $\|\cdot\|_{\mathbb{H}}^{\alpha}$ yields
\begin{eqnarray}
&&\!\!\!\!\!\!\!\!\|X_{t\wedge \tau_M^{(n)}}^{(n)}-X_{(t-\Delta)\wedge \tau_M^{(n)}}^{(n)}\|_{\mathbb{H}}^{\alpha}
 \nonumber \\
=&&\!\!\!\!\!\!\!\!\alpha\int_{(t-\Delta)\wedge \tau_M^{(n)}} ^{t\wedge \tau_M^{(n)}}\Big(\|X_{s}^{(n)}-X_{(t-\Delta)\wedge \tau_M^{(n)}}^{(n)}\|_{\mathbb{H}}^{\alpha-2}
\nonumber \\
 &&\!\!\!\!\!\!\!\!\cdot
{}_{\mathbb{V}^*}\langle \pi_n\mathcal{A}(s,X^{(n)}_s,\mathscr{L}_{X^{(n)}_s}), X_{s}^{(n)}-X_{(t-\Delta)\wedge \tau_M^{(n)}}^{(n)}\rangle_{\mathbb{V}} \Big)ds
 \nonumber \\
 &&\!\!\!\!\!\!\!\!
 +\frac{\alpha}{2}\int_{(t-\Delta)\wedge \tau_M^{(n)}} ^{t\wedge \tau_M^{(n)}}\Big(\|X_{s}^{(n)}-X_{(t-\Delta)\wedge \tau_M^{(n)}}^{(n)}\|_{\mathbb{H}}^{\alpha-2}
 \nonumber \\
 &&\!\!\!\!\!\!\!\!
 \cdot\|\pi_n\mathcal{B}(s,X^{(n)}_s,\mathscr{L}_{X^{(n)}_s})\tilde{\pi}_n\|_{L_2(U,\mathbb{H})}^2\Big)ds
 \nonumber \\
 &&\!\!\!\!\!\!\!\!+
 \frac{\alpha(\alpha-2)}{2}  \int_{(t-\Delta)\wedge \tau_M^{(n)}}^{t\wedge \tau_M^{(n)}}\Big(\|X_{s}^{(n)}-X_{(t-\Delta)\wedge \tau_M^{(n)}}^{(n)}\|_{\mathbb{H}}^{\alpha-4}
 \nonumber \\
 &&\!\!\!\!\!\!\!\!\cdot
\|\big(\pi_n \mathcal{B}(s,X^{(n)}_s,\mathscr{L}_{X^{(n)}_s})\tilde{\pi}_n\big)^*(X_{s}^{(n)}-X_{(t-\Delta)\wedge \tau_M^{(n)}}^{(n)})\|_{U}^2\Big)ds
 \nonumber \\
 &&\!\!\!\!\!\!\!\!+\alpha\int_{(t-\Delta)\wedge \tau_M^{(n)}} ^{t\wedge \tau_M^{(n)}}\|X_{s}^{(n)}-X_{(t-\Delta)\wedge \tau_M^{(n)}}^{(n)}\|_{\mathbb{H}}^{\alpha-2}
 \nonumber \\
 &&\!\!\!\!\!\!\!\!\cdot
 \langle \pi_n\mathcal{B}(s,X^{(n)}_s,\mathscr{L}_{X^{(n)}_s})dW^{(n)}_s, X_{s}^{(n)}-X_{(t-\Delta)\wedge \tau_M^{(n)}}^{(n)}\rangle_{\mathbb{H}} \nonumber\\
=:&&\!\!\!\!\!\!\!\!\text{(I)}+\text{(II)}+\text{(III)}+\text{(IV)}.  \label{F3.0}
\end{eqnarray}
For the first term $\text{(I)}$, similar to the proof of \textbf{Case 1}, we can get
\begin{eqnarray*}
&&\!\!\!\!\!\!\!\!\mathbb{E}^{\mathbf{P}_n}\int^{T}_{\Delta}\text{(I)}dt\nonumber\\
\leq&&\!\!\!\!\!\!\!\! C\mathbb{E}^{\mathbf{P}_n}\Bigg\{\int^{T}_{\Delta}\int_{(t-\Delta)\wedge \tau_M^{(n)}} ^{t\wedge \tau_M^{(n)}}\|X_{s}^{(n)}-X_{(t-\Delta)\wedge \tau_M^{(n)}}^{(n)}\|_{\mathbb{H}}^{\alpha-2}
\nonumber\\
&&\!\!\!\!\!\!\!\!\cdot
\|\mathcal{A}(s,X^{(n)}_s,\mathscr{L}_{X^{(n)}_s})\|_{\mathbb{V}^*}
\|X_{s}^{(n)}-X_{(t-\Delta)\wedge \tau_M^{(n)}}^{(n)}\|_{\mathbb{V}} ds dt\Bigg\}\nonumber\\
\leq&&\!\!\!\!\!\!\!\!C\Bigg\{\mathbb{E}^{\mathbf{P}_n}\int^{T}_{\Delta}\int_{(t-\Delta)\wedge \tau_M^{(n)}} ^{t\wedge \tau_M^{(n)}}\|\mathcal{A}(s,X^{(n)}_s,\mathscr{L}_{X^{(n)}_s})\|_{\mathbb{V}^*}^{\frac{\alpha}{\alpha-1}}dsdt\Bigg\}^{\frac{\alpha-1}{\alpha}}
\nonumber\\
&&\!\!\!\!\!\!\!\!\cdot
\Bigg\{\mathbb{E}^{\mathbf{P}_n}\int^{T}_{\Delta}\int_{(t-\Delta)\wedge \tau_M^{(n)}} ^{t\wedge \tau_M^{(n)}}\|X_{s}^{(n)}-X_{(t-\Delta)\wedge \tau_M^{(n)}}^{(n)}\|^\alpha_{\mathbb{V}} dsdt\Bigg\}^{\frac{1}{\alpha}}
\nonumber\\
\leq&&\!\!\!\!\!\!\!\!C_{p,M,T}\Delta\big(1+\mathbb{E}\|\xi\|_{\mathbb{H}}^p\big).
\end{eqnarray*}
Similarly, we easily deduce that
$$\mathbb{E}^{\mathbf{P}_n}\int^{T}_{\Delta}\text{(II)}+\text{(III)}dt\leq C_{p,M,T}\Delta\big(1+\mathbb{E}\|\xi\|_{\mathbb{H}}^p\big)$$
and
$$\mathbb{E}^{\mathbf{P}_n}\int^{T}_{\Delta}\text{(IV)}dt=0.$$
Hence, in view of (\ref{esq7}), we can conclude
$$\sup_{n\in\mathbb{N}}\mathbb{E}^{\mathbf{P}_n}
\Bigg\{\int_0^{T}\|X^{(n)}_{t\wedge \tau_M^{(n)}}-X^{(n)}_{(t-\Delta)\wedge \tau_M^{(n)}}\|_{\mathbb{H}}^{\alpha}dt\Bigg\}\leq C_{p,M,T}\Delta\big(1+\mathbb{E}\|\xi\|_{\mathbb{H}}^p\big),$$
which implies that (\ref{esq6}) holds.
We complete the proof of this lemma.
\end{proof}

\vspace{2mm}
We now have the following corollary.
\begin{corollary}\label{coro1}
The set of measures $\{\hat{\mathbf{P}}_n\}_{n\in\mathbb{N}}$ is tight in $\mathscr{P}(\mathbb{S})$.
\end{corollary}

\begin{proof}
According to Lemmas \ref{lem6} and \ref{lem7}, for any sequence $\{\hat{\mathbf{P}}_n\}_{n\in\mathbb{N}}$ we can find a subsequence, still denoted by $\{\hat{\mathbf{P}}_n\}_{n\in\mathbb{N}}$, such that $\{\hat{\mathbf{P}}_n\}_{n\in\mathbb{N}}$ converges weakly in $\mathscr{P}(\mathbb{S})$. Therefore, the assertion follows.
\end{proof}

\subsection{Proof of Theorems \ref{th1} and \ref{th3}}\label{sec2.4}
Since $\{\hat{\mathbf{P}}_n\}_{n\in\mathbb{N}}$ is tight in $\mathscr{P}(\mathbb{S})$, by Skorohod's representation theorem, there exists a probability space $(\tilde{\Omega},\tilde{\mathscr{F}},\tilde{\mathbb{P}})$, and on this space, $\mathbb{S}$-valued random variables $\tilde{X}^{(n)},\tilde{X}$ (here choosing a subsequence if necessary) such that

\vspace{2mm}
(i) the law of $\tilde{X}^{(n)}$ under $\tilde{\mathbb{P}}$ is equal to $\hat{\mathbf{P}}_n$.

\vspace{2mm}
(ii) the following convergence holds
\begin{equation}\label{es80}
\tilde{X}^{(n)}\to \tilde{X}~\text{in}~\mathbb{S},~\tilde{\mathbb{P}}\text{-a.s.},~\text{as}~n\to\infty,
\end{equation}
and $\tilde{X}$ has the law $\mathbf{P}$.

\vspace{2mm}

To prove Theorem \ref{th1}, we intend to show that  $\mathbf{P}$ is a martingale solution  of Eq.~(\ref{eqSPDE})  in the sense of Definition \ref{de3}. Before starting the proof, we first briefly provide a roadmap concerning this lengthy proof. In what follows, our objective is to prove $(M1)$ and $(M2)$ in Definition \ref{de3} hold, which implies $\mathbf{P}\in\mathscr{M}_{\mu_0}^{\mathcal{A},\mathcal{B}}$. By the a priori estimates and applying Fatou's Lemma, we have the moment estimate \eref{es119}, then due to the growth condition $\mathbf{(A_3)}$ we can see that $(M1)$ holds. Next, in order to show that $\mathcal{M}_l(t,X,\mu)$ is a continuous $(\mathscr{F}_t)$-martingale, we first truncate the nonlinear term in $\mathcal{M}_l(t,X,\mu)$ and prove the continuity of the truncated term in Lemma \ref{lem8}.
Then, by identifying the martingale property of $\mathcal{M}_l(t,X,\mu)$ and characterizing the associated quadratic variation process (see Lemmas \ref{lem9} and \ref{lem10}), we are able to show that $(M2)$ holds.

\vspace{2mm}

Note that by Lemmas \ref{lem3.0} and \ref{lem5}, we obtain
\begin{equation}\label{es16}
\mathbb{E}^{\tilde{\mathbb{P}}}\Big[\sup_{t\in[0,T]}\|\tilde{X}^{(n)}_t\|_{\mathbb{H}}^p\Big]+\mathbb{E}^{\tilde{\mathbb{P}}}\Bigg\{\int_0^T\|\tilde{X}^{(n)}_t\|_{\mathbb{V}}^{\alpha}dt\Bigg\}^{r}+\mathbb{E}^{\tilde{\mathbb{P}}}\int_0^T\|\tilde{X}^{(n)}_t\|_{\mathbb{V}}^{\alpha}\|\tilde{X}^{(n)}_t\|_{\mathbb{H}}^{p-2}dt
\leq C_{p,T},
\end{equation}
for some constants $r>1$, where $C_{p,T}$ is independent of $n$.

\vspace{1mm}
Since
$$\mathscr{L}_{\tilde{X}^{(n)}|_{\tilde{\mathbb{P}}}}\to\mathscr{L}_{\tilde{X}|_{\tilde{\mathbb{P}}}}~\text{weakly in}~ \mathscr{P}(\mathbb{S})~ (\text{here selecting a subsequence if necessary}),$$
from \cite[Theorem 5.5]{CD1} and (\ref{es16}), it leads to the following convergence
\begin{equation}\label{es18}
\mathscr{L}_{\tilde{X}^{(n)}|_{\tilde{\mathbb{P}}}}\to\mathscr{L}_{\tilde{X}|_{\tilde{\mathbb{P}}}}~\text{in}~\mathscr{P}_{2}(\mathbb{C}_T)\cap \mathscr{P}_{\alpha}(L^{\alpha}([0,T];\mathbb{V})).
\end{equation}

Now we denote
\begin{equation}\label{eqs2}
\mathbb{K}:=\mathscr{P}_{2}(\mathbb{C}_T)\cap \mathscr{P}_{\alpha}(L^{\alpha}([0,T];\mathbb{V})),
\end{equation}
and
$$\Psi(T,X):=\sup_{t\in[0,T]}\|X_t\|_{\mathbb{H}}^p+\Bigg\{\int_0^{T}\|X_t\|_{\mathbb{V}}^{\alpha}dt\Bigg\}^r+\int_0^{T}\|X_t\|_{\mathbb{V}}^{\alpha}\|X_t\|_{\mathbb{H}}^{p-2}dt.$$
First,  by (\ref{es16}) and Fatou's lemma, we obtain
\begin{eqnarray*}
\mathbb{E}^{\mathbf{P}}\Big[\sup_{t\in[0,T]}\|X_t\|_{\mathbb{H}}^p\Big]=&&\!\!\!\!\!\!\!\!\mathbb{E}^{\tilde{\mathbb{P}}}\Big[\sup_{t\in[0,T]}\|\tilde{X}_t\|_{\mathbb{H}}^p\Big]
\nonumber\\
\leq&&\!\!\!\!\!\!\!\!\liminf_{n\to\infty}\mathbb{E}^{\tilde{\mathbb{P}}}\Big[\sup_{t\in[0,T]}\|\tilde{X}^{(n)}_t\|_{\mathbb{H}}^p\Big]
\nonumber\\
<&&\!\!\!\!\!\!\!\!\infty.
\end{eqnarray*}
On the other hand, we note that  $\|\cdot\|_{\mathbb{V}}$ are  lower semicontinuous in $\mathbb{H}$,
thus we also have
\begin{eqnarray}\label{eslow}
\mathbb{E}^{\mathbf{P}}\Bigg\{\int_0^{T}\|X_t\|_{\mathbb{V}}^{\alpha}dt\Bigg\}^r=&&\!\!\!\!\!\!\!\!\mathbb{E}^{\tilde{\mathbb{P}}}\Bigg\{\int_0^{T}\|\tilde{X}_t\|_{\mathbb{V}}^{\alpha}dt\Bigg\}^r
\nonumber\\
\leq&&\!\!\!\!\!\!\!\!\mathbb{E}^{\tilde{\mathbb{P}}}\Bigg\{\int_0^{T}\liminf_{n\to\infty}\|\tilde{X}^{(n)}_t\|_{\mathbb{V}}^{\alpha}dt\Bigg\}^r
\nonumber\\
\leq&&\!\!\!\!\!\!\!\!\liminf_{n\to\infty}\mathbb{E}^{\tilde{\mathbb{P}}}\Bigg\{\int_0^{T}\|\tilde{X}^{(n)}_t\|_{\mathbb{V}}^{\alpha}dt\Bigg\}^r
\nonumber\\
<&&\!\!\!\!\!\!\!\!\infty.
\end{eqnarray}
Similarly, we deduce
\begin{equation*}
\mathbb{E}^{\mathbf{P}}\int_0^{T}\|X_t\|_{\mathbb{V}}^{\alpha}\|X_t\|_{\mathbb{H}}^{p-2}dt
<\infty.
\end{equation*}
We can conclude that
\begin{equation}\label{es119}
\mathbb{E}^{\mathbf{P}}\big[\Psi(T,X)\big]
<\infty.
\end{equation}

By the growth condition $(\mathbf{A_3})$, we have
\begin{eqnarray*}
&&\!\!\!\!\!\!\!\!\int_0^T\|\mathcal{A}(s,X_s,\mu_s)\|_{\mathbb{V}^*}ds
\nonumber \\
\leq  &&\!\!\!\!\!\!\!\!C_T\Bigg(\int_0^T\|\mathcal{A}(s,X_s,\mu_s)\|_{\mathbb{V}^*}^{\frac{\alpha}{\alpha-1}}ds\Bigg)^{\frac{\alpha-1}{\alpha}}
\nonumber \\
\leq  &&\!\!\!\!\!\!\!\!C_T\Bigg(\int_0^T\big(1+\|X_s\|_{\mathbb{V}}^{\alpha}+\mu_s(\|\cdot\|_{\mathbb{V}}^{\alpha})\big)\big(1+\|X_s\|_{\mathbb{H}}^{\beta}+\mu_s(\|\cdot\|_{\mathbb{H}}^{\beta})\big)ds\Bigg)^{\frac{\alpha-1}{\alpha}}
\nonumber \\
\leq  &&\!\!\!\!\!\!\!\!C_T+\Bigg(\int_0^T\|X_s\|_{\mathbb{V}}^{\alpha}ds+\int_0^T\mathbb{E}^{\mathbf{P}}\|X_s\|_{\mathbb{V}}^{\alpha}ds+\int_0^T\|X_s\|_{\mathbb{V}}^{\alpha}\|X_s\|_{\mathbb{H}}^{\beta}ds
\nonumber \\
 &&\!\!\!\!\!\!\!\!+\int_0^T\|X_s\|_{\mathbb{V}}^{\alpha}\mathbb{E}^{\mathbf{P}}\|X_s\|_{\mathbb{H}}^{\beta}ds+\int_0^T\|X_s\|_{\mathbb{H}}^{\beta}\mathbb{E}^{\mathbf{P}}\|X_s\|_{\mathbb{V}}^{\alpha}ds
 \nonumber \\
&&\!\!\!\!\!\!\!\!
  +
  \int_0^T\mathbb{E}^{\mathbf{P}}\|X_s\|_{\mathbb{H}}^{\beta}\cdot\mathbb{E}^{\mathbf{P}}\|X_s\|_{\mathbb{V}}^{\alpha}ds\Bigg)^{\frac{\alpha-1}{\alpha}}
\end{eqnarray*}
and
\begin{eqnarray*}
\int_0^T\|\mathcal{B}(s,X_s,\mu_s)\|_{L_2(U;{\mathbb{H}})}^2ds
\leq C\int_0^T\big(1+\|X_s\|_{\mathbb{H}}^{2}+\mathbb{E}^{\mathbf{P}}\|X_s\|_{\mathbb{H}}^{2}\big)ds.
\end{eqnarray*}
Consequently, in view of (\ref{es119}) we deduce that
$$\int_0^T\|\mathcal{A}(s,X_s,\mu_s)\|_{\mathbb{V}^*}ds+\int_0^T\|\mathcal{B}(s,X_s,\mu_s)\|_{L_2(U;{\mathbb{H}})}^2ds<\infty~~\mathbf{P}\text{-a.s.},$$
namely, $(M1)$ in Definition \ref{de3} holds.

\vspace{2mm}
Next, we intend to prove $(M2)$ for  $\mathbf{P}$. Fix $l\in\mathbb{V}$. Specifically,
 we would like to show
$$\mathcal{M}_l(t,X,\mu):={}_{\mathbb{V}^*}\langle X_t,l\rangle_{\mathbb{V}}-{}_{\mathbb{V}^*}\langle X_0,l\rangle_{\mathbb{V}}-\int_0^t{}_{\mathbb{V}^*}\langle \mathcal{A}(s,X_s,\mu_s),l\rangle_{\mathbb{V}}ds,~t\in[0,T],$$
where $\mu_s=\mathbf{P}\circ \pi_s^{-1}$,  is a continuous $(\mathscr{F}_t)$-martingale w.r.t.~$\mathbf{P}$, whose quadratic variation process
is given by
$$\langle \mathcal{M}_l\rangle(t,X,\mu)=\int_0^t\|\mathcal{B}(s,X_s,\mu_s)^*l\|_U^2ds,~t\in[0,T].$$

To this end, for any $(t,w,\nu)\in[0,T]\times\mathbb{S}\times\mathbb{K}$, we denote
\begin{eqnarray*}
\varpi(t,w,\nu):=\varpi^{(1)}(t,w)+\varpi^{(2)}(t,w,\nu),
\end{eqnarray*}
where
$$\varpi^{(1)}(t,w):={}_{\mathbb{V}^*}\langle w_t,l\rangle_{\mathbb{V}}-{}_{\mathbb{V}^*}\langle w_0,l\rangle_{\mathbb{V}},$$
$$\varpi^{(2)}(t,w,\nu):=-\int_0^t{}_{\mathbb{V}^*}\langle \mathcal{A}(s,w_s,\nu_s),l\rangle_{\mathbb{V}}ds.$$
Moreover,
for any $R>0$ and $(t,w,\nu)\in[0,T]\times\mathbb{S}\times\mathbb{K}$, we  set
\begin{eqnarray}\label{es20}
\varpi_R(t,w,\nu)=\varpi^{(1)}_R(t,w)+\varpi^{(2)}_R(t,w,\nu),
\end{eqnarray}
where
$$\varpi^{(1)}_R(t,w):={}_{\mathbb{V}^*}\langle w_t,l\rangle_{\mathbb{V}}\cdot\chi_R({}_{\mathbb{V}^*}\langle w_t,l\rangle_{\mathbb{V}})-{}_{\mathbb{V}^*}\langle w_0,l\rangle_{\mathbb{V}}\cdot\chi_R({}_{\mathbb{V}^*}\langle w_0,l\rangle_{\mathbb{V}}),$$
$$\varpi^{(2)}_R(t,w,\nu):=-\int_0^t{}_{\mathbb{V}^*}\langle \mathcal{A}(s,w_s,\nu_s),l\rangle_{\mathbb{V}}\cdot\chi_R({}_{\mathbb{V}^*}\langle \mathcal{A}(s,w_s,\nu_s),l\rangle_{\mathbb{V}})ds.$$
Here  $\chi_R\in C^{\infty}_c(\mathbb{R})$ is a cut-off function with
$$\chi_R(r)=\begin{cases} 1,~~~~|r|\leq R&\quad\\
0,~~~~|r|>2R.&\quad\end{cases}$$

We have the following continuity property of $\varpi_R$.
\begin{lemma}\label{lem8} For any $t\in[0,T]$, $w^{(n)},w\in \mathbb{S}$ and $\nu^{(n)},\nu\in\mathbb{K}$ with $(w^{(n)},\nu^{(n)})\to(w,\nu)$ in $\mathbb{S}\times\mathbb{K}$, as $n\to \infty$,
we have
$$\lim_{n\to \infty} \varpi_R(t,w^{(n)},\nu^{(n)})=\varpi_R(t,w,\nu).$$
\end{lemma}

\begin{proof}
In fact,
if  $w^{(n)},w\in \mathbb{S}$ and $\nu^{(n)},\nu\in\mathbb{K}$ with $(w^{(n)},\nu^{(n)})\to(w,\nu)$ in $\mathbb{S}\times\mathbb{K}$, as $n\to \infty$, we have
\begin{eqnarray}
&&w^{(n)}\to w~~\text{in}~\mathbb{C}_T;
\nonumber \\
&&
w^{(n)}\to w~~\text{in}~L^{\alpha}([0,T];\mathbb{V}), \label{es34}
\end{eqnarray}
and
\begin{eqnarray}
&&\nu^{(n)}_t\to \nu_t~~\text{in}~\mathscr{P}_{2}(\mathbb{H}),~~t\in[0,T];\label{es35}
 \\
&&\nu^{(n)}_t\to \nu_t~~\text{in}~\mathscr{P}_{\alpha}(\mathbb{V}),~~dt\text{-a.e.}.\label{es14}
\end{eqnarray}
Notice that (\ref{es35})-(\ref{es14}) imply
\begin{equation}\label{es15}
\nu^{(n)}_t\to \nu_t~~\text{in}~\mathbb{M}_1,~~dt\text{-a.e.},
\end{equation}
where we recall $\mathbb{M}_1=\mathscr{P}_{2}(\mathbb{H})\cap \mathscr{P}_{\alpha}(\mathbb{V})$.

\vspace{1mm}
Then it is straightforward that
$$\lim_{n\to\infty}\varpi^{(1)}_R(t,w^{(n)})=\varpi^{(1)}_R(t,w).$$
Since $(x,\mu)\mapsto{}_{\mathbb{V}^*}\langle A(t,x,\mu),l\rangle_{\mathbb{V}}\cdot\chi_R({}_{\mathbb{V}^*}\langle A(t,x,\mu),l\rangle_{\mathbb{V}})$ is  bounded and $\chi_R(\cdot)$ is continuous, then by $(\mathbf{A_1})$, (\ref{es34}), (\ref{es15}) and the dominated convergence theorem, we have for any $t\in[0,T]$,
$$\lim_{n\to\infty}\varpi^{(2)}_R(t,w^{(n)},\nu^{(n)})=\varpi^{(2)}_R(t,w,\nu).$$
Therefore, the lemma follows.
\end{proof}

\vspace{2mm}
Now we give the following two important lemmas.

\begin{lemma}\label{lem9} Under the assumptions in Theorem \ref{th1}, for any $0\leq s<t\leq T$,
\begin{equation}\label{es43}
\mathbb{E}^{\mathbf{P}}\big[\mathcal{M}_l(t,X,\mu)\big|\mathscr{F}_s\big]=\mathcal{M}_l(s,X,\mu).
\end{equation}
\end{lemma}
\begin{proof}
Note that
\begin{eqnarray}\label{es24}
&&\!\!\!\!\!\!\!\!\mathbb{E}^{\tilde{\mathbb{P}}}|\mathcal{M}_l(t,\tilde{X}^{(n)},\tilde{\mu}^{(n)})-\mathcal{M}_l(t,\tilde{X},\tilde{\mu})|
\nonumber \\
=&&\!\!\!\!\!\!\!\!\mathbb{E}^{\tilde{\mathbb{P}}}|\varpi(t,\tilde{X}^{(n)},\tilde{\mu}^{(n)})-\varpi_R(t,\tilde{X}^{(n)},\tilde{\mu}^{(n)})|
\nonumber \\
&&\!\!\!\!\!\!\!\!
+\mathbb{E}^{\tilde{\mathbb{P}}}|\varpi_R(t,\tilde{X}^{(n)},\tilde{\mu}^{(n)})-\varpi_R(t,\tilde{X},\tilde{\mu})|
\nonumber \\
&&\!\!\!\!\!\!\!\!+\mathbb{E}^{\tilde{\mathbb{P}}}|\varpi_R(t,\tilde{X},\tilde{\mu})-\varpi(t,\tilde{X},\tilde{\mu})|,
\end{eqnarray}
where
$$\tilde{\mu}:=\mathscr{L}_{\tilde{X}|_{\tilde{\mathbb{P}}}},~~ \tilde{\mu}^{(n)}:=\mathscr{L}_{\tilde{X}^{(n)}|_{\tilde{\mathbb{P}}}}.$$
First  by (\ref{es80}), (\ref{es18}), Lemma \ref{lem8} and the dominated convergence theorem, we obtain
\begin{equation}\label{es39}
\lim_{n\to\infty}\mathbb{E}^{\tilde{\mathbb{P}}}|\varpi_R(t,\tilde{X}^{(n)},\tilde{\mu}^{(n)})-\varpi_R(t,\tilde{X},\tilde{\mu})|=0.
\end{equation}
For the first term on right hand side of (\ref{es24}),  we know
\begin{eqnarray}\label{es25}
&&\!\!\!\!\!\!\!\!\mathbb{E}^{\tilde{\mathbb{P}}}|\varpi(t,\tilde{X}^{(n)},\tilde{\mu}^{(n)})-\varpi_R(t,\tilde{X}^{(n)},\tilde{\mu}^{(n)})|
\nonumber \\
\leq&&\!\!\!\!\!\!\!\!\mathbb{E}^{\tilde{\mathbb{P}}}|\varpi^{(1)}(t,\tilde{X}^{(n)},\tilde{\mu}^{(n)})-\varpi^{(1)}_R(t,\tilde{X}^{(n)},\tilde{\mu}^{(n)})|
\nonumber \\
&&\!\!\!\!\!\!\!\!+\mathbb{E}^{\tilde{\mathbb{P}}}|\varpi^{(2)}(t,\tilde{X}^{(n)},\tilde{\mu}^{(n)})-\varpi^{(2)}_R(t,\tilde{X}^{(n)},\tilde{\mu}^{(n)})|.
\end{eqnarray}
By $\mathbf{(A_2)}$, it follows that
\begin{eqnarray*}
&&\!\!\!\!\!\!\!\!\mathbb{E}^{\tilde{\mathbb{P}}}|\varpi^{(2)}(t,\tilde{X}^{(n)},\tilde{\mu}^{(n)})-\varpi_R^{(2)}(t,\tilde{X}^{(n)},\tilde{\mu}^{(n)})|
\nonumber \\
\leq&&\!\!\!\!\!\!\!\! \mathbb{E}^{\tilde{\mathbb{P}}}\Bigg\{\int_0^T|{}_{\mathbb{V}^*}\langle \mathcal{A}(s,\tilde{X}^{(n)}_s,\tilde{\mu}^{(n)}_s),l\rangle_{\mathbb{V}}|\cdot \mathbf{1}_{\big\{|{}_{\mathbb{V}^*}\langle \mathcal{A}(s,\tilde{X}^{(n)}_s,\tilde{\mu}^{(n)}_s),l\rangle_{\mathbb{V}}|\geq R\big\}}ds\Bigg\}
\nonumber \\
\leq&&\!\!\!\!\!\!\!\!\|l\|_{\mathbb{V}}\Bigg\{\Big(\int_0^T\mathbb{E}^{\tilde{\mathbb{P}}}\|\mathcal{A}(s,\tilde{X}^{(n)}_s,\tilde{\mu}^{(n)}_s)\|_{\mathbb{V}^*}^{\frac{\alpha}{\alpha-1}}ds\Big)^{\frac{\alpha-1}{\alpha}}
\nonumber \\
&&\!\!\!\!\!\!\!\!
\cdot \Big(\int_0^T{\tilde{\mathbb{P}}}\big(|{}_{\mathbb{V}^*}\langle \mathcal{A}(s,\tilde{X}^{(n)}_s,\tilde{\mu}^{(n)}_s),l\rangle_{\mathbb{V}}|\geq R\big)ds\Big)^{\frac{1}{\alpha}}\Bigg\}
\nonumber \\
\leq&&\!\!\!\!\!\!\!\!\|l\|_{\mathbb{V}}\Bigg\{\mathbb{E}^{\tilde{\mathbb{P}}}\int_0^T\|\mathcal{A}(s,\tilde{X}^{(n)}_s,\tilde{\mu}^{(n)}_s)\|_{{\mathbb{V}}^*}^{\frac{\alpha}{\alpha-1}}ds\Bigg\}\Big/R^{\frac{1}{\alpha-1}}
\nonumber \\
\leq&&\!\!\!\!\!\!\!\!C_{\|l\|_{\mathbb{V}}}\Bigg\{\mathbb{E}^{\tilde{\mathbb{P}}}\int_0^T\big(1+\|\tilde{X}^{(n)}_s\|_{\mathbb{V}}^{\alpha}+ \mathbb{E}^{\tilde{\mathbb{P}}}\|\tilde{X}^{(n)}_s\|_{\mathbb{V}}^{\alpha}\big)
\nonumber \\
&&\!\!\!\!\!\!\!\!
\cdot\big(1+ \|\tilde{X}^{(n)}_s\|_{\mathbb{H}}^{\beta}+\mathbb{E}^{\tilde{\mathbb{P}}}\|\tilde{X}^{(n)}_s\|_{\mathbb{H}}^{\beta}\big)ds\Bigg\}\Big/R^{\frac{1}{\alpha-1}}
\nonumber \\
\leq&&\!\!\!\!\!\!\!\!C_{p,T,\|l\|_{\mathbb{V}}}\Big/R^{\frac{1}{\alpha-1}},
\end{eqnarray*}
where the last step is due to (\ref{es16}).  Thus it is straightforward that
\begin{equation}\label{es27}
\lim_{R\to\infty}\sup_{n\in\mathbb{N}}\sup_{t\in[0,T]}\mathbb{E}^{\tilde{\mathbb{P}}}|\varpi^{(2)}(t,\tilde{X}^{(n)},\tilde{\mu}^{(n)})-\varpi_R^{(2)}(t,\tilde{X}^{(n)},\tilde{\mu}^{(n)})|=0.
\end{equation}
Similarly,  we can also get
\begin{equation}\label{es28}
\lim_{R\to\infty}\sup_{n\in\mathbb{N}}\sup_{t\in[0,T]}\mathbb{E}^{\tilde{\mathbb{P}}}|\varpi^{(1)}(t,\tilde{X}^{(n)})-\varpi_R^{(1)}(t,\tilde{X}^{(n)})|=0.
\end{equation}
Combining (\ref{es25})-(\ref{es28}), we deduce
\begin{equation}\label{es40}
\lim_{R\to\infty}\sup_{n\in\mathbb{N}}\sup_{t\in[0,T]}\mathbb{E}^{\tilde{\mathbb{P}}}|\varpi(t,\tilde{X}^{(n)},\tilde{\mu}^{(n)})-\varpi_R(t,\tilde{X}^{(n)},\tilde{\mu}^{(n)})|=0.
\end{equation}
A similar argument shows
\begin{equation}\label{es41}
\lim_{R\to\infty}\sup_{n\in\mathbb{N}}\sup_{t\in[0,T]}\mathbb{E}^{\tilde{\mathbb{P}}}|\varpi(t,\tilde{X},\tilde{\mu})-\varpi_R(t,\tilde{X},\tilde{\mu})|=0.
\end{equation}
Hence, in view of (\ref{es39})-(\ref{es41}), it follows that for any $t\in[0,T]$,
\begin{equation}\label{es41re}
\lim_{n\to\infty}\mathbb{E}^{\tilde{\mathbb{P}}}|\mathcal{M}_l(t,\tilde{X}^{(n)},\tilde{\mu}^{(n)})-\mathcal{M}_l(t,\tilde{X},\tilde{\mu})|=0.
\end{equation}
Let $\Phi(\cdot)$ be any bounded continuous and $\mathscr{F}_s$-measurable real valued function on $\mathbb{S}$. Due to \eref{es41re}, we deduce
\begin{eqnarray*}
&&\!\!\!\!\!\!\!\!\mathbb{E}^{\mathbf{P}}\Big[\Big(\mathcal{M}_l(t,X,\mu)-\mathcal{M}_l(s,X,\mu)\Big)\Phi(X)\Big]
\nonumber \\
=&&\!\!\!\!\!\!\!\!
\mathbb{E}^{\tilde{\mathbb{P}}}\Big[\Big(\mathcal{M}_l(t,\tilde{X},\tilde{\mu})-\mathcal{M}_l(s,\tilde{X},\tilde{\mu})\Big)\Phi(\tilde{X})\Big]
\nonumber \\
=&&\!\!\!\!\!\!\!\!
\lim_{n\to\infty}\mathbb{E}^{\tilde{\mathbb{P}}}\Big[\Big(\mathcal{M}_l(t,\tilde{X}^{(n)},\tilde{\mu}^{(n)})-\mathcal{M}_l(s,\tilde{X}^{(n)},\tilde{\mu}^{(n)})\Big)\Phi(\tilde{X}^{(n)})\Big]
\nonumber \\
=&&\!\!\!\!\!\!\!\!
\lim_{n\to\infty}\mathbb{E}^{\hat{\mathbf{P}}_n}\Big[\Big(\mathcal{M}_l(t,X,\mu^{(n)})-\mathcal{M}_l(s,X,\mu^{(n)})\Big)\Phi(X)\Big]
\nonumber \\
=&&\!\!\!\!\!\!\!\!0,
\end{eqnarray*}
where $\mu^{(n)}_t:=\hat{\mathbf{P}}_n \circ \pi^{-1}_t$, which implies the assertion by the arbitrariness of $\Phi$.
\end{proof}

\begin{lemma}\label{lem10} Under the assumptions in Theorem \ref{th1}, for any $0\leq s<t\leq T$,
\begin{eqnarray}\label{es44}
&&\!\!\!\!\!\!\!\!\mathbb{E}^{\mathbf{P}}\Big[\mathcal{M}_l^2(t,X,\mu)-\int_0^t\|\mathcal{B}(r,X_r,\mu_r)^*l\|_{U}^2dr   \Big|\mathscr{F}_s\Big]
\nonumber \\
=&&\!\!\!\!\!\!\!\!\mathcal{M}_l^2(s,X,\mu)-\int_0^s\|\mathcal{B}(r,X_r,\mu_r)^*l\|_{U}^2dr.
\end{eqnarray}
\end{lemma}
\begin{proof}
Note that by $\mathbf{(A_3)}$ and (\ref{es16}),
\begin{eqnarray*}
\sup_{n\in\mathbb{N}}\mathbb{E}^{\tilde{\mathbb{P}}}|\mathcal{M}_l(t,\tilde{X}^{(n)},\tilde{\mu}^{(n)})|^p
\leq&&\!\!\!\!\!\!\!\!C\sup_{n\in\mathbb{N}}\mathbb{E}^{\tilde{\mathbb{P}}}\Bigg\{\int_0^T\|\mathcal{B}(t,\tilde{X}^{(n)}_t,\tilde{\mu}^{(n)}_t)^*l\|_{U}^2dt\Bigg\}^{\frac{p}{2}}
\nonumber \\
\leq&&\!\!\!\!\!\!\!\!C_T\sup_{n\in\mathbb{N}}\mathbb{E}^{\tilde{\mathbb{P}}}\int_0^T\|\mathcal{B}(t,\tilde{X}^{(n)}_t,\tilde{\mu}^{(n)}_t)^*l\|_{U}^pdt
\nonumber \\
<&&\!\!\!\!\!\!\!\!\infty.
\end{eqnarray*}
Since $p>2$, by (\ref{es41re}) and using Vitali's convergence theorem, we have
$$\lim_{n\to\infty}\mathbb{E}^{\tilde{\mathbb{P}}}|\mathcal{M}_l(t,\tilde{X}^{(n)},\tilde{\mu}^{(n)})-\mathcal{M}_l(t,\tilde{X},\tilde{\mu})|^2=0,$$
and
$$\lim_{n\to\infty}\mathbb{E}^{\tilde{\mathbb{P}}}\Big|\int_0^t\|\big(\mathcal{B}(r,\tilde{X}^{(n)}_r,\tilde{\mu}^{(n)}_r)-\mathcal{B}(r,\tilde{X}_r,\tilde{\mu}_r)\big)^*l\|_{U}^2dr\Big|=0.$$
Thus following the similar argument for proving (\ref{es43}), we obtain (\ref{es44}).
\end{proof}

\vspace{2mm}
\textbf{Proof of Theorem \ref{th1}.} From Lemmas \ref{lem9}-\ref{lem10}, we can conclude $\mathbf{P}\in\mathscr{M}_{\mu_0}^{\mathcal{A},\mathcal{B}}$. Hence, (\ref{eqSPDE}) has a martingale solution. Moreover, the estimate (\ref{esq370}) holds from (\ref{es119}). We complete the proof of Theorem \ref{th1}.  \hspace{\fill}$\Box$

\vspace{2mm}
The following modified version of Yamada-Watanabe theorem plays an important role in proving the existence of strong solutions to mean field SPDEs from weak solutions, which has been established recently in \cite[Lemma 2.1]{HW1} for the case of finite dimensions. In the sequel, we use $\mathscr{L}_{X_t}|_{\mathbb{P}}$ to stress the law of $X_t$ under the probability measure $\mathbb{P}$.

\begin{lemma}\label{lem112} $($Modified Yamada-Watanabe Theorem$)$
Suppose that the MVSPDE
\begin{equation}\label{eq13}
dX_t=\mathcal{A}(t,X_t,\mathscr{L}_{X_t})dt+\mathcal{B}(t,X_t,\mathscr{L}_{X_t})dW_t
\end{equation}
has a weak solution $(X,W)$ under the probability measure $\mathbb{P}$. Let $\mu_t=\mathscr{L}_{X_t}|_{\mathbb{P}}$, $t\in[0,T]$.
If the following SPDE
\begin{equation}\label{eq3}
d\tilde{X}_t=\mathcal{A}(t,\tilde{X}_t,\mu_t)dt+\mathcal{B}(t,\tilde{X}_t,\mu_t)dW_t
\end{equation}
 has pathwise uniqueness with initial value $\tilde{X}_0\sim\mu_0$, then (\ref{eq13}) has a strong solution starting at $\tilde{X}_0$. Moreover, if (\ref{eq13}) has pathwise uniqueness for any initial value $X_0\sim\mu_0$, then it has a unique weak solution for the initial law $\mu_0$.
\end{lemma}
\begin{proof}
We separate the proof into the following two steps.

\medskip
\textbf{Strong existence.} First it is easy to see that $(X,W)$ under $\mathbb{P}$ is a weak solution of SPDE (\ref{eq3}) with initial law $\mu_0$ due to  $\mu_t=\mathscr{L}_{X_t}|_{\mathbb{P}}$. Using the infinite-dimensional version of Yamada-Watanabe theorem from \cite{RSZ},  the pathwise uniqueness of (\ref{eq3}) implies the strong (resp.~weak) existence and uniqueness of (\ref{eq3}) with initial value $\tilde{X}_0$ (resp.~the initial law $\mu_0$). This leads to $\mathscr{L}_{\tilde{X}_t}=\mu_t$, $t\in[0,T]$. Hence $\tilde{X}_t$ is a strong solution of (\ref{eq13}).

\medskip
\textbf{Weak uniqueness.} Given two weak solutions $(X,W)$ and $(\bar{X},\bar{W})$ w.r.t.~the stochastic basis $(\Omega,\mathscr{F},\{\mathscr{F}_t\}_{t\in[0,T]},\mathbb{P})$ and $(\bar{\Omega},\bar{\mathscr{F}},\bar{\{\mathscr{F}_t\}}_{t\in[0,T]},\bar{\mathbb{P}})$, respectively, such that $\mathscr{L}_{X_0}|_{\mathbb{P}}=\mathscr{L}_{\bar{X}_0}|_{\bar{\mathbb{P}}}$.  Note that $\bar{X}_t$ solves the following equation
\begin{equation}\label{26}
d\bar{X}_t=\mathcal{A}(t,\bar{X}_t,\mathscr{L}_{\bar{X}_t}|_{\bar{\mathbb{P}}})dt+\mathcal{B}(t,\bar{X}_t,\mathscr{L}_{{\bar{X}_t}} |_{\bar{\mathbb{P}}})d\bar{W}_t.
\end{equation}
Our aim is to verify $\mathscr{L}_{X_t}|_{\mathbb{P}}=\mathscr{L}_{\bar{X}_t}|_{\bar{\mathbb{P}}}$, $t\in[0,T]$.

By the argument in the  proof of strong existence, the following SPDE
\begin{equation}\label{25}
dY_t=\mathcal{A}(t,Y_t,\mu_t)dt+\mathcal{B}(t,Y_t,\mu_t)d\bar{W}_t,~Y_0=\bar{X}_0,
\end{equation}
has a unique solution under $(\bar{\Omega},\bar{\mathscr{F}},\bar{\{\mathscr{F}_t\}}_{t\in[0,T]},\bar{\mathbb{P}})$, thus the weak uniqueness  also holds. We note that $(X,W)$ also solves
$$dX_t=\mathcal{A}(t,X_t,\mu_t)dt+\mathcal{B}(t,X_t,\mu_t)dW_t,~\mathscr{L}_{X_0}|_{\mathbb{P}}=\mathscr{L}_{\bar{X}_0}|_{\bar{\mathbb{P}}}.$$
The weak uniqueness of (\ref{25}) implies that
\begin{equation*}
\mathscr{L}_{X_t}|_{\mathbb{P}}=\mathscr{L}_{Y_t}|_{\bar{\mathbb{P}}},~t\in[0,T].
\end{equation*}
Then it is obvious that (\ref{25}) reduces to
$$dY_t=\mathcal{A}(t,Y_t,\mathscr{L}_{Y_t}|_{\bar{\mathbb{P}}})dt+\mathcal{B}(t,Y_t,\mathscr{L}_{Y_t}|_{\bar{\mathbb{P}}})d\bar{W}_t,~Y_0=\bar{X}_0.$$
By the assumption, (\ref{26}) is pathwisely unique, then it follows that $Y_t=\bar{X}_t$, $t\in[0,T]$. Consequently, we conclude that $$\mathscr{L}_{X_t}|_{\mathbb{P}}=\mathscr{L}_{Y_t}|_{\bar{\mathbb{P}}}=\mathscr{L}_{\bar{X}_t}|_{\bar{\mathbb{P}}},~t\in[0,T].$$
The proof is completed.
\end{proof}

\vspace{2mm}

Now we are in the position to complete the proof of Theorem \ref{th3}.

\vspace{2mm}

\textbf{Proof of Theorem \ref{th3}.} By Theorem \ref{th1} and the martingale representation theorem, we know that (\ref{eqSPDE}) has a weak solution
$$(\Omega,\mathscr{F},\{\mathscr{F}_t\}_{t\in[0,T]},\mathbb{P}; X,W).$$
According to Lemma \ref{lem112}, for the existence of strong solutions to (\ref{eqSPDE}) it suffices to prove that the following decoupled SPDE
\begin{equation}\label{eqdec}
d\bar{X}_t=\mathcal{A}^{\mu}(t,\bar{X}_t)dt+\mathcal{B}^{\mu}(t,\bar{X}_t)dW_t,
\end{equation}
where $\mu_t:=\mathscr{L}_{X_t}$,  $\mathcal{A}^{\mu}(t,\cdot):=A(t,\cdot,\mu_t)$ (similarly, $\mathcal{B}^{\mu}(t,\cdot)$), has pathwise uniqueness.

\vspace{1mm}
Indeed,  let $\bar{X},\bar{Y}$ be two solutions of (\ref{eqdec}) with same initial value $\xi\in L^p(\Omega,\mathscr{F}_0,\mathbb{P};\mathbb{H})\cap L^2(\Omega,\mathscr{F}_0,\mathbb{P};\mathbb{V})$. First,
we note that the estimate (\ref{es45})  holds for $\bar{X},\bar{Y}$, i.e.,
\begin{eqnarray}
&&\mathbb{E}\Big[\sup_{t\in[0,T]}\|\bar{X}_t\|_{{\mathbb{H}}}^p\Big]+\mathbb{E}\int_0^T(1+\|\bar{X}_t\|_{{\mathbb{H}}}^{p-2})\|\bar{X}_t\|_{\mathbb{V}}^{\alpha}dt<\infty,\label{es19}
\\
&&\mathbb{E}\Big[\sup_{t\in[0,T]}\|\bar{Y}_t\|_{{\mathbb{H}}}^p\Big]+\mathbb{E}\int_0^T(1+\|\bar{Y}_t\|_{{\mathbb{H}}}^{p-2})\|\bar{Y}_t\|_{\mathbb{V}}^{\alpha}dt<\infty.\label{es22}
\end{eqnarray}
Setting
$$\phi(t):=C+\rho(\bar{X}_t,\mu_t)+\eta(\bar{Y}_t,\mu_t).$$
By It\^{o}'s formula and the product rule we have
\begin{eqnarray*}
&&\!\!\!\!\!\!\!\!e^{-\int_0^t\phi(s)ds}\|\bar{X}_t-\bar{Y}_t\|_{\mathbb{H}}^2
\nonumber\\
\leq&&\!\!\!\!\!\!\!\!\int_0^te^{-\int_0^s\phi(r)dr}
\Big(2{}_{{\mathbb{V}}^*}\langle \mathcal{A}^{\mu}(s,\bar{X}_s)-\mathcal{A}^{\mu}(s,\bar{Y}_s),\bar{X}_s-\bar{Y}_s\rangle_{\mathbb{V}}
\nonumber\\
&&\!\!\!\!\!\!\!\!
+\|\mathcal{B}^{\mu}(s,\bar{X}_s)-\mathcal{B}^{\mu}(s,\bar{Y}_s)\|_{L_2(U;\mathbb{H})}^2
-\phi(s)\|X_s-Y_s\|_{\mathbb{H}}^2\Big)ds
\nonumber\\
&&\!\!\!\!\!\!\!\!+2\int_0^te^{-\int_0^s\phi(r)dr}
\langle \bar{X}_s-\bar{Y}_s,\big(\mathcal{B}^{\mu}(s,\bar{X}_s)-\mathcal{B}^{\mu}(s,\bar{Y}_s)\big)dW_s\rangle_{\mathbb{H}}
\nonumber\\
\leq&&\!\!\!\!\!\!\!\!2\int_0^te^{-\int_0^s\phi(r)dr}
\langle \bar{X}_s-\bar{Y}_s,\big(\mathcal{B}^{\mu}(s,\bar{X}_s)-\mathcal{B}^{\mu}(s,\bar{Y}_s)\big)dW_s\rangle_{\mathbb{H}},
\end{eqnarray*}
where we used $\mathbf{(A_5)}$ in the last step. Taking expectation on both sides of the above inequality, by (\ref{es19}) and (\ref{es22}) we have
\begin{equation*}
\mathbb{E}\Big[e^{-\int_0^t\phi(s)ds}\|\bar{X}_t-\bar{Y}_t\|_{\mathbb{H}}^2\Big]\leq 0.
\end{equation*}
Note that by (\ref{es19}) and (\ref{es22}), it follows that
$$\int_0^T\rho(\bar{X}_s,\mu_s)+\eta(\bar{Y}_s,\mu_s)ds<\infty,~\mathbb{P}\text{-a.s.}.$$
Consequently, we deduce that
$$\bar{X}_t=\bar{Y}_t,~\mathbb{P}\text{-a.s.},~t\in[0,T],$$
which implies the pathwise uniqueness of (\ref{eqdec}) by the path continuity on $\mathbb{H}$. \hspace{\fill}$\Box$

\subsection{Proof of Theorem \ref{th2}}\label{proof2}
In this part, we aim to show the pathwise uniqueness of solutions to (\ref{eqSPDE}) on $\mathbb{H}$.

Let $X,Y$ be two solutions of (\ref{eqSPDE}) with initial values $X_0=Y_0=\xi\in L^p(\Omega,\mathscr{F}_0,\mathbb{P};{\mathbb{H}})\cap L^2(\Omega,\mathscr{F}_0,\mathbb{P};\mathbb{V})$, with $p\in[\vartheta,\infty)\cap (2,\infty)$, which fulfill that $\mathbb{P}$-a.s.,
\begin{eqnarray*}
&&X_t=\xi+\int_0^t\mathcal{A}(s,X_s,\mathscr{L}_{X_s})ds+\int_0^t\mathcal{B}(s,X_s,\mathscr{L}_{X_s})dW_s,~t\in[0,T],
\nonumber\\
&&Y_t=\xi+\int_0^t\mathcal{A}(s,Y_s,\mathscr{L}_{Y_s})ds+\int_0^t\mathcal{B}(s,Y_s,\mathscr{L}_{Y_s})dW_s,~t\in[0,T].
\end{eqnarray*}
Recall (\ref{es45}), we have
\begin{eqnarray}
&&\mathbb{E}\Big[\sup_{t\in[0,T]}\|X_t\|_{{\mathbb{H}}}^p\Big]+\mathbb{E}\int_0^T(1+\|X_t\|_{{\mathbb{H}}}^{p-2})\|X_t\|_{\mathbb{V}}^{\alpha}dt<\infty,\label{es57}
\\
&&\mathbb{E}\Big[\sup_{t\in[0,T]}\|Y_t\|_{{\mathbb{H}}}^p\Big]+\mathbb{E}\int_0^T(1+\|Y_t\|_{{\mathbb{H}}}^{p-2})\|Y_t\|_{\mathbb{V}}^{\alpha}dt<\infty.\label{es58}
\end{eqnarray}

\vspace{2mm}
\textbf{Case 1.} Suppose that $(\mathbf{A'_5})$ holds. Applying It\^{o}'s formula, by $(\mathbf{A'_5})$, (\ref{es57})-(\ref{es58}), we can get that for any $t\in[0,T]$,
\begin{eqnarray*}
&&\!\!\!\!\!\!\!\!\|X_t-Y_t\|_{\mathbb{H}}^2
\nonumber\\
=&&\!\!\!\!\!\!\!\!\int_0^t
\Big(2{}_{{\mathbb{V}}^*}\langle \mathcal{A}(s,X_s,\mathscr{L}_{X_s})-\mathcal{A}(s,Y_s,\mathscr{L}_{Y_s}),X_s-Y_s\rangle_{\mathbb{V}}
\nonumber\\
&&\!\!\!\!\!\!\!\!
+\|\mathcal{B}(s,X_s,\mathscr{L}_{X_s})-\mathcal{B}(s,Y_s,\mathscr{L}_{Y_s})\|_{L_2(U;\mathbb{H})}^2\Big)ds+2\mathcal{M}_t
\nonumber\\
\leq&&\!\!\!\!\!\!\!\! C\int_0^t \big(1+\rho(0,\mathscr{L}_{X_s})+\eta(0,\mathscr{L}_{Y_s})\big)\|X_s-Y_s\|_{\mathbb{H}}^2ds
\nonumber\\
&&\!\!\!\!\!\!\!\! +C\int_0^t \big(1+\rho(X_s,\mathscr{L}_{X_s})+\eta(Y_s,\mathscr{L}_{Y_s})\big)\mathbb{E}\|X_s-Y_s\|_{\mathbb{H}}^2ds+2\mathcal{M}_t,
\end{eqnarray*}
where
$$\mathcal{M}_t:=\int_0^t
\langle X_s-Y_s,\big(\mathcal{B}(s,X_s,\mathscr{L}_{X_s})-\mathcal{B}(s,Y_s,\mathscr{L}_{Y_s})\big)dW_s\rangle_{\mathbb{H}}.$$
Due to (\ref{conb}) and (\ref{es57})-(\ref{es58}), we deduce
\begin{equation*}
\mathbb{E}\|X_t-Y_t\|_{\mathbb{H}}^2
\leq C\int_0^t \big(1+\mathbb{E}\rho(X_s,\mathscr{L}_{X_s})+\mathbb{E}\eta(Y_s,\mathscr{L}_{Y_s})\big)\mathbb{E}\|X_s-Y_s\|_{\mathbb{H}}^2ds.
\end{equation*}
Note that by (\ref{esq22}) and (\ref{es57})-(\ref{es58}) again, we obtain
\begin{equation}\label{es8}
\mathbb{E}\int_0^T\big(\rho(X_t,\mathscr{L}_{X_t})+\eta(Y_t,\mathscr{L}_{Y_t})\big)dt<\infty.
\end{equation}
Therefore, by Gronwall's lemma
\begin{equation*}
\mathbb{E}\|X_t-Y_t\|_{\mathbb{H}}^2\leq 0,
\end{equation*}
which implies
 $$X_t=Y_t,~\mathbb{P}\text{-a.s.},~t\in[0,T].$$
 Thus the pathwise uniqueness follows from the path continuity on $\mathbb{H}$.

\vspace{2mm}
\textbf{Case 2.} Suppose that $(\mathbf{A''_5})$   holds. For any $R>R_0$, we define a stopping time
$$\tau_R:=\tau_R^X\wedge\tau_R^Y=\inf\Bigg\{t\in[0,T]:\Big\{\|X_t\|_{\mathbb{H}}+\int_0^t\|X_s\|_{\mathbb{V}}^{\alpha}ds\Big\}\vee \Big\{\|Y_t\|_{\mathbb{H}}+\int_0^t\|Y_s\|_{\mathbb{V}}^{\alpha}ds\Big\}\geq R\Bigg\}.$$
Using It\^{o}'s formula, by (\ref{es57})-(\ref{es58}) we obtain that for any $t\in[0,T]$,
\begin{eqnarray}\label{es211}
&&\!\!\!\!\!\!\!\!\|X_{t\wedge\tau_R}-Y_{t\wedge\tau_R}\|_{\mathbb{H}}^2
\nonumber\\
\leq&&\!\!\!\!\!\!\!\!\int_0^{t\wedge\tau_R}
\Big(2{}_{{\mathbb{V}}^*}\langle \mathcal{A}(s,X_s,\mathscr{L}_{X_s})-\mathcal{A}(s,Y_s,\mathscr{L}_{Y_s}),X_s-Y_s\rangle_{\mathbb{V}}
\nonumber\\
&&\!\!\!\!\!\!\!\!
+\|\mathcal{B}(s,X_s,\mathscr{L}_{X_s})-\mathcal{B}(s,Y_s,\mathscr{L}_{Y_s})\|_{L_2(U,\mathbb{H})}^2\Big)ds+2\mathcal{M}_{t\wedge\tau_R}
\nonumber\\
\leq&&\!\!\!\!\!\!\!\!C\int_0^{t\wedge\tau_R}\big(1+\rho(X_s,\mathscr{L}_{X_s})+\eta(Y_s,\mathscr{L}_{Y_s})\big)
\nonumber\\
&&\!\!\!\!\!\!\!\!\cdot
\big(\|X_s-Y_s\|_{\mathbb{H}}^2+\mathcal{W}_{2,R}(\mathscr{L}_{X^s},\mathscr{L}_{Y^s})^2\big)ds
+2\mathcal{M}_{t\wedge\tau_R}
\nonumber\\
\leq&&\!\!\!\!\!\!\!\!C\int_0^{t\wedge\tau_R}\big(1+\rho(X_s,\mathscr{L}_{X_s})+\eta(Y_s,\mathscr{L}_{Y_s})\big)
\nonumber\\
&&\!\!\!\!\!\!\!\!\cdot\big(\|X_s-Y_s\|_{\mathbb{H}}^2+\mathbb{E}\Big[\sup_{r\in[0,s\wedge \tau_R]}\|X_r-Y_r\|_{\mathbb{H}}^2\Big]\big)ds
+2\mathcal{M}_{t\wedge\tau_R},
\end{eqnarray}
where $\mathcal{M}_t$ is a local martingale given by
$$\mathcal{M}_t:=\int_0^t\langle X_s-Y_s, \big(\mathcal{B}(s,X_s,\mathscr{L}_{X_s})-\mathcal{B}(s,Y_s,\mathscr{L}_{Y_s})\big)dW_s\rangle_{\mathbb{H}},$$
and we used the fact that
$$\mathcal{W}_{2,R}(\mathscr{L}_{\xi^t},\mathscr{L}_{\zeta^t})^2\leq \mathbb{E}\Big[\sup_{s\in[0,t\wedge \tau_R]}\|\xi_s-\zeta_s\|_{\mathbb{H}}^2\Big].$$

Using Burkholder-Davis-Gundy's inequality and by (\ref{es57})-(\ref{es58}), it follows that
\begin{eqnarray*}
&&\!\!\!\!\!\!\!\!\mathbb{E}\Big[\sup_{t\in[0,T]}|\mathcal{M}_{t\wedge\tau_R}|\Big]
\nonumber\\
\leq&&\!\!\!\!\!\!\!\!\mathbb{E}\Bigg\{\int_0^{T\wedge\tau_R}\|X_t-Y_t\|_{\mathbb{H}}^2
\|\mathcal{B}(t,X_t,\mathscr{L}_{X_t})-\mathcal{B}(t,Y_t,\mathscr{L}_{Y_t})\|_{L_2(U,\mathbb{H})}^2dt\Bigg\}^{\frac{1}{2}}
\nonumber\\
\leq&&\!\!\!\!\!\!\!\!\frac{1}{2}\mathbb{E}\Big[\sup_{t\in[0,T\wedge\tau_R]}\|X_{t}-Y_{t}\|_{\mathbb{H}}^2\Big]+C\mathbb{E}\int_0^{T\wedge\tau_R}\big(1+\rho(X_t,\mathscr{L}_{X_t})+\eta(Y_t,\mathscr{L}_{Y_t})\big)\|X_t-Y_t\|_{\mathbb{H}}^2dt
\nonumber\\
&&\!\!\!\!\!\!\!\!
+C\mathbb{E}\int_0^{T\wedge\tau_R}\big(1+\rho(X_t,\mathscr{L}_{X_t})+\eta(Y_t,\mathscr{L}_{Y_t})\big)\cdot\mathbb{E}\Big[\sup_{s\in[0,t\wedge \tau_R]}\|X_s-Y_s\|_{\mathbb{H}}^2\Big]dt.
\end{eqnarray*}
Now taking $\sup_{t\in[0,T]}$ and expectation on both sides of (\ref{es211}), we have
\begin{eqnarray*}
&&\!\!\!\!\!\!\!\!\mathbb{E}\Big[\sup_{t\in[0,T\wedge\tau_R]}\|X_{t}-Y_{t}\|_{\mathbb{H}}^2\Big]
\nonumber\\
\leq&&\!\!\!\!\!\!\!\!
 C\mathbb{E}\int_0^{T\wedge\tau_R}\big(1+\rho(X_t,\mathscr{L}_{X_t})+\eta(Y_t,\mathscr{L}_{Y_t})\big)\|X_t-Y_t\|_{\mathbb{H}}^2dt
\nonumber\\
&&\!\!\!\!\!\!\!\!
+C\mathbb{E}\int_0^{T\wedge\tau_R}\big(1+\rho(X_t,\mathscr{L}_{X_t})+\eta(Y_t,\mathscr{L}_{Y_t})\big)\cdot\mathbb{E}\Big[\sup_{s\in[0,t\wedge \tau_R]}\|X_s-Y_s\|_{\mathbb{H}}^2\Big]dt.
\end{eqnarray*}
Combining stochastic Gronwall's lemma (cf.~Lemma \ref{appen2} in Appendix) and (\ref{es57})-(\ref{es58}), it follows that
\begin{eqnarray}\label{es23}
&&\!\!\!\!\!\!\!\!\mathbb{E}\Big[\sup_{t\in[0,T\wedge\tau_R]}\|X_{t}-Y_{t}\|_{\mathbb{H}}^2\Big]
\nonumber\\
\leq &&\!\!\!\!\!\!\!\! C_R\int_0^{T}\mathbb{E}\big(1+\rho(X_t,\mathscr{L}_{X_t})+\eta(Y_t,\mathscr{L}_{Y_t})\big)\cdot\mathbb{E}\Big[\sup_{s\in[0,t\wedge \tau_R]}\|X_s-Y_s\|_{\mathbb{H}}^2\Big]dt.
\end{eqnarray}
Therefore, in view of (\ref{es8}) and applying Gronwall's lemma to (\ref{es23}) implies
\begin{equation*}
\mathbb{E}\Big[\sup_{t\in[0,T\wedge\tau_R]}\|X_{t}-Y_{t}\|_{\mathbb{H}}^2\Big]\leq  0.
\end{equation*}
Using Fatou's lemma, it leads to
\begin{equation*}
\mathbb{E}\Big[\sup_{t\in[0,T]}\|X_{t}-Y_{t}\|_{\mathbb{H}}^2\Big]\leq \liminf_{R\to\infty}\mathbb{E}\Big[\sup_{t\in[0,T\wedge\tau_R]}\|X_{t}-Y_{t}\|_{\mathbb{H}}^2\Big]\leq 0.
\end{equation*}
We complete the proof.  \hspace{\fill}$\Box$

\section{Proof of mean field limit}\label{sec meam}
In this section, we aim to prove the mean field limit for the weakly interacting SPDE system (\ref{eqi}). To this end, in Subsection \ref{Sub6.1}, we  construct the Galerkin approximating system of (\ref{eqi}) and establish their  uniform in $(n,N)$-estimates. In Subsection \ref{sec4.2}, we first establish the joint tightness of the approximation sequence in $\mathbb{C}_T(\mathbb{H})\cap L^{\alpha}([0,T];\mathbb{V})$ with respect to both the dimension $n$ and the number of particles $N$. Then, we pass the limit of the Galerkin interacting  system $\{X^{(n),N,i}\}_{1\leq i\leq N}$,  as $n$ tends to infinity, for any fixed $N\geq2$.
In Subsection \ref{sec4.3}, we  characterize the limit of the empirical law $\mathcal{S}^N$ by letting $N$ tends to infinity and complete the proof of Theorem \ref{th4}.

\subsection{Galerkin scheme}\label{Sub6.1}
We first establish
the following  weakly interacting system (on $\mathbb{H}_n$) by Galerkin scheme,
\begin{equation}\label{eqiG}
\left\{ \begin{aligned}
&dX^{(n),N,i}_t=\pi_n\mathcal{A}(t,X^{(n),N,i}_t,\mathcal{S}^{(n),N}_t)dt+\pi_n\mathcal{B}(t,X^{(n),N,i}_t,\mathcal{S}^{(n),N}_t)dW_t^{(n),i},\\
&X^{(n),N,i}_0=\xi^{(n),i},
\end{aligned} \right.
\end{equation}
where $1\leq i\leq N$, $\xi^{(n),i}:=\pi_n \xi^{i}$,
 $$W^{(n),i}_t:=\tilde{\pi}_nW^i_t=\sum\limits_{k=1}^{n}\langle W^i_t,g_k\rangle_{U}g_k,~n\in\mathbb{N},~1\leq i\leq N,$$
and
$$\mathcal{S}^{(n),N}_t:=\frac{1}{N}\sum_{j=1}^N\delta_{X^{(n),N,j}_t},~n\in\mathbb{N}.$$
Under the assumptions $\mathbf{(A_0)}$-$\mathbf{(A_4)}$ and $\mathbf{(A_5^*)}$, by \cite[Theorem 3.1.1]{LR1} it is easy to see that the finite dimensional system (\ref{eqiG}) has a unique (strong) solution $X^{(n),N}:=(X^{(n),N,1},\ldots,X^{(n),N,N})$. We mention that by the condition $\mathbf{(A_6)}$ and Remark \ref{remark1}, the initial vector $X^{(n),N}_0:=(\xi^{(n),1},\ldots,\xi^{(n),N})$ is also symmetric and so is $X^{(n),N}$.

In the sequel, we present some moment bounds of the particles $X^{(n),N,i},i=1,\ldots,N$, which are uniform in $(n,N)$.
\begin{lemma}\label{lem11}
Suppose that the assumptions in Theorem \ref{th4} hold. For any $q\geq 2$, there exist some constants $C_{q,T}>0$ independent of $(n,N)$ such that
\begin{eqnarray}\label{es116}
&&\mathbb{E}\Big[\sup_{t\in[0,T]}\frac{1}{N}\sum_{i=1}^N\|X^{(n),N,i}_t\|_{\mathbb{H}}^q\Big]
\nonumber\\&&+\mathbb{E}\Bigg\{\int_0^{T}\frac{1}{N}\sum_{i=1}^N\Big(\|X^{(n),N,i}_t\|_{\mathbb{H}}^{q-2}\|X^{(n),N,i}_t\|_{\mathbb{V}}^{\alpha}\Big)dt\Bigg\}\leq C_{q,T},
\end{eqnarray}
\begin{equation}\label{es117}
\mathbb{E}\Big[\sup_{t\in[0,T]}\|X^{(n),N,1}_t\|_{\mathbb{H}}^q\Big]+\mathbb{E}\Bigg\{\int_0^{T}\|X^{(n),N,1}_t\|_{\mathbb{H}}^{q-2}\|X^{(n),N,1}_t\|_{\mathbb{V}}^{\alpha}dt\Bigg\}\leq C_{q,T}.
\end{equation}
\end{lemma}

\begin{proof}
Applying It\^{o}'s formula, we can get that for any $t\in[0,T]$,
\begin{eqnarray}\label{es112}
&&\!\!\!\!\!\!\!\!\|X^{(n),N,i}_t\|_{\mathbb{H}}^q
\nonumber \\
=&&\!\!\!\!\!\!\!\!\|\xi^{(n),i}\|_{\mathbb{H}}^q+\frac{q(q-2)}{2}\int_0^t\|X^{(n),N,i}_s\|_{\mathbb{H}}^{q-4}\|\mathcal{B}(s,X^{(n),N,i}_s,\mathcal{S}^N_s)^*X^{(n),N,i}_s\|_U^2ds
\nonumber \\
&&\!\!\!\!\!\!\!\!+\frac{q}{2}\int_0^t\|X^{(n),N,i}_s\|_{\mathbb{H}}^{q-2}\Big(2{}_{\mathbb{V}^*}\langle \mathcal{A}(s,X^{(n),N,i}_s,\mathcal{S}^N_s),X^{(n),N,i}_s\rangle_{\mathbb{V}}
\nonumber \\
&&\!\!\!\!\!\!\!\!
+\|\mathcal{B}(s,X^{(n),N,i}_s,\mathcal{S}^N_s)\|_{L_2(U;{\mathbb{H}})}^2\Big)ds
\nonumber \\
&&\!\!\!\!\!\!\!\!+q\int_0^t\|X^{(n),N,i}_s\|_{\mathbb{H}}^{q-2}\langle X^{(n),N,i}_s,\pi_n\mathcal{B}(s,X^{(n),N,i}_s,\mathcal{S}^N_s)dW^{(n),i}_s\rangle_{\mathbb{H}}
\nonumber \\
=:&&\!\!\!\!\!\!\!\!\|\xi^{(n),i}\|_{\mathbb{H}}^q+\text{(I)}+\text{(II)}+\text{(III)}.
\end{eqnarray}
For $\text{(I)}$, by (\ref{conb}) and Young's inequality, it follows that
\begin{eqnarray}\label{es113}
\mathbb{E}\Big[\sup_{t\in[0,T]}\frac{1}{N}\sum_{i=1}^N\text{(I)}\Big]
\leq&&\!\!\!\!\!\!\!\!C_{q,T}+C_q\mathbb{E}\Bigg\{\int_0^T\frac{1}{N}\sum_{i=1}^N\|X^{(n),N,i}_t\|_{\mathbb{H}}^qdt\Bigg\}
\nonumber \\
&&\!\!\!\!\!\!\!\!
+C_p\mathbb{E}\Bigg\{\int_0^T\frac{1}{N^2}\sum_{i,j=1}^N\Big(\|X^{(n),N,i}_t\|_{\mathbb{H}}^{q-2}\|X^{(n),N,j}_t\|_{\mathbb{H}}^{2}\Big)dt\Bigg\}
\nonumber \\
\leq&&\!\!\!\!\!\!\!\!C_{q,T}+C_{q}\mathbb{E}\Bigg\{\int_0^T\frac{1}{N}\sum_{i=1}^N\|X^{(n),N,i}_t\|_{\mathbb{H}}^qdt\Bigg\}.
\end{eqnarray}
Using $\mathbf{(A_2)}$, by a similar calculation, we also get
\begin{eqnarray}\label{es114}
&&\!\!\!\!\!\!\!\!\mathbb{E}\Big[\sup_{t\in[0,T]}\frac{1}{N}\sum_{i=1}^N\text{(II)}\Big]
\nonumber \\
\leq&&\!\!\!\!\!\!\!\!C_{q,T}+C_q\mathbb{E}\Bigg\{\int_0^T\frac{1}{N}\sum_{i=1}^N\|X^{(n),N,i}_t\|_{\mathbb{H}}^qdt\Bigg\}
\nonumber \\
&&\!\!\!\!\!\!\!\!
-C_q\delta_1\mathbb{E}\Bigg\{\int_0^T\frac{1}{N}\sum_{i=1}^N\Big(\|X^{(n),N,i}_t\|_{\mathbb{H}}^{q-2}\|X^{(n),N,i}_t\|_{\mathbb{V}}^{\alpha}\Big)dt\Bigg\}.
\end{eqnarray}
To estimate  $\text{(III)}$, 
applying  Burkholder-Davis-Gundy's inequality, it follows that
\begin{eqnarray}\label{es115}
&&\!\!\!\!\!\!\!\!\mathbb{E}\Big[\sup_{t\in[0,T]}\frac{1}{N}\sum_{i=1}^N\text{(III)}\Big]
\nonumber \\
\leq&&\!\!\!\!\!\!\!\!\frac{C_q}{N}\sum_{i=1}^N\mathbb{E}\Bigg\{\int_0^{T}\|X^{(n),N,i}_t\|_{\mathbb{H}}^{2q-2}\|\mathcal{B}(t,X^{(n),N,i}_t,\mathcal{S}^{(n),N}_t)\|_{L_2(U;{\mathbb{H}})}^2dt\Bigg\}^{\frac{1}{2}}
\nonumber \\
\leq&&\!\!\!\!\!\!\!\!C_{q,T}+C_{q}\mathbb{E}\Bigg\{\int_0^{T}\frac{1}{N}\sum_{i=1}^N\|X^{(n),N,i}_t\|_{\mathbb{H}}^qdt\Bigg\}
+\frac{1}{2}\mathbb{E}\Big[\sup_{t\in[0,T]} \frac{1}{N}\sum_{i=1}^N\|X^{(n),N,i}_t\|_{\mathbb{H}}^q\Big].
\end{eqnarray}
Combining (\ref{es112})-(\ref{es115}), by a standard localization argument,  it is easy to deduce that
\begin{eqnarray*}
&&\!\!\!\!\!\!\!\!\mathbb{E}\Big[\sup_{t\in[0,T]}\frac{1}{N}\sum_{i=1}^N\|X^{(n),N,i}_t\|_{\mathbb{H}}^q\Big]
+C_q\delta_1\mathbb{E}\Bigg\{\int_0^{T}\frac{1}{N}\sum_{i=1}^N\Big(\|X^{(n),N,i}_t\|_{\mathbb{H}}^{q-2}\|X^{(n),N,i}_t\|_{\mathbb{V}}^{\alpha}\Big)dt\Bigg\}
\nonumber \\
\leq&& \!\!\!\!\!\!\!\! C_{q,T}\Big(1+\frac{1}{N}\sum_{i=1}^N\mathbb{E}\|\xi^{(n),i}\|_{\mathbb{H}}^q\Big)+C_q\mathbb{E}\Bigg\{\int_0^{T}\frac{1}{N}\sum_{i=1}^N\|X^{(n),N,i}_t\|_{\mathbb{H}}^qdt\Bigg\}.
\end{eqnarray*}
By Gronwall's lemma, it leads to
\begin{eqnarray*}
&&\!\!\!\!\!\!\!\!\mathbb{E}\Big[\sup_{t\in[0,T]}\frac{1}{N}\sum_{i=1}^N\|X^{(n),N,i}_t\|_{\mathbb{H}}^q\Big]+C_q\delta_1\mathbb{E}\Bigg\{\int_0^{T}\frac{1}{N}\sum_{i=1}^N\Big(\|X^{(n),N,i}_t\|_{\mathbb{H}}^{q-2}\|X^{(n),N,i}_t\|_{\mathbb{V}}^{\alpha}\Big)dt\Bigg\}
\nonumber \\
\leq&&\!\!\!\!\!\!\!\!C_{q,T}\Big(1+\frac{1}{N}\sum_{i=1}^N\mathbb{E}\|\xi^{(n),i}\|_{\mathbb{H}}^q\Big)\leq C_{q,T},
\end{eqnarray*}
where we  used  (\ref{c3}) and the condition that the law of $X_0^N$ is symmetric.
We have proved (\ref{es116}).

\vspace{1mm}
Now we prove (\ref{es117}). Note that by the argument above, we can also get
\begin{eqnarray*}
&&\!\!\!\!\!\!\!\!\mathbb{E}\Big[\sup_{t\in[0,T]}\|X^{(n),N,1}_t\|_{\mathbb{H}}^q\Big]+C_q\delta_1\mathbb{E}\Bigg\{\int_0^{T}\|X^{(n),N,1}_s\|_{\mathbb{H}}^{q-2}\|X^{(n),N,1}_s\|_{\mathbb{V}}^{\alpha}ds\Bigg\}
\nonumber \\
\leq &&\!\!\!\!\!\!\!\!C_{q,T}(1+\|\xi^{(n),1}\|_{\mathbb{H}}^q)+C_{q,T}\mathbb{E}\Bigg\{\int_0^T\frac{1}{N}\sum_{i=1}^N\|X^{(n),N,i}_t\|_{\mathbb{H}}^qdt\Bigg\}+C_{q,T}\mathbb{E}\int_0^T\|X^{(n),N,1}_t\|_{\mathbb{H}}^qdt.
\end{eqnarray*}
Using (\ref{es116}) and Gronwall's lemma, it follows that
\begin{eqnarray*}
\mathbb{E}\Big[\sup_{t\in[0,T]}\|X^{(n),N,1}_t\|_{\mathbb{H}}^q\Big]+C_q\delta_1\mathbb{E}\Bigg\{\int_0^{T}\|X^{(n),N,1}_t\|_{\mathbb{H}}^{q-2}\|X^{(n),N,1}_t\|_{\mathbb{V}}^{\alpha}dt\Bigg\}
\leq C_{q,T}.
\end{eqnarray*}
We complete the proof.
\end{proof}

\begin{lemma}\label{lem3}
Suppose that the assumptions in Theorem \ref{th4} hold. For any $r>1$,  there exist $C_{r,T}>0$ independent of $(n,N)$ such that
\begin{equation}\label{es1112}
\frac{1}{N}\sum_{i=1}^N\mathbb{E}\Bigg\{\int_0^T\|X^{(n),N,i}_t\|_{\mathbb{V}}^{\alpha} dt\Bigg\}^{r}+\mathbb{E}\Bigg\{\int_0^T\|X^{(n),N,1}_t\|_{\mathbb{V}}^{\alpha} dt\Bigg\}^{r}\leq C_{r,T}.
\end{equation}

\end{lemma}

\begin{proof}
Notice that in view of the proof of  (\ref{es116}), for any $t\in[0,T]$,
\begin{eqnarray*}
\int_0^t\|X^{(n),N,i}_t\|_{\mathbb{V}}^{\alpha} ds
\leq&&\!\!\!\!\!\!\!\!C_T(1+\|\xi^{(n),i}\|_{\mathbb{H}}^2)
+C\int_0^t\| X^{(n),N,i}_s\|_{\mathbb{H}}^2ds+C\int_0^t \frac{1}{N}\sum_{j=1}^N\|X^{(n),N,j}_s\|_{\mathbb{H}}^2ds
\nonumber \\
&&\!\!\!\!\!\!\!\!
+ C\int_0^t \langle X^{(n),N,i}_s,\pi_n\mathcal{B}(s,X^{(n),N,i}_s,\mathcal{S}^N_s)dW^{(n),i}_s\rangle_{\mathbb{H}}.
\end{eqnarray*}
Thus
\begin{eqnarray}\label{es9}
\mathbb{E}\Bigg\{\int_0^t\|X^{(n),N,i}_t\|_{\mathbb{V}}^{\alpha} ds\Bigg\}^r
\leq&&\!\!\!\!\!\!\!\!C_{r,T}\big(1+\mathbb{E}\|\xi^{(n),i}\|_{\mathbb{H}}^{2r}\big)
+C_{r,T}\int_0^t\mathbb{E}\| X^{(n),N,i}_s\|_{\mathbb{H}}^{2r}ds
\nonumber \\
&&\!\!\!\!\!\!\!\!+C_{r,T}\int_0^t \mathbb{E}\Big[\frac{1}{N}\sum_{j=1}^N\|X^{(n),N,j}_s\|_{\mathbb{H}}^{2r}\Big]ds
\nonumber \\
&&\!\!\!\!\!\!\!\!
+ C\mathbb{E}\Big|\int_0^t \langle X^{(n),N,i}_s,\pi_n\mathcal{B}(s,X^{(n),N,i}_s,\mathcal{S}^N_s)dW^{(n),i}_s\rangle_{\mathbb{H}}\Big|^r.
\end{eqnarray}
By Burkholder-Davis-Gundy's inequality, due to (\ref{es116}), it follows that
\begin{eqnarray}\label{es10}
&&\!\!\!\!\!\!\!\!\frac{1}{N}\sum_{i=1}^N\mathbb{E}\Big|\int_0^t \langle X^{(n),N,i}_s,\pi_n\mathcal{B}(s,X^{(n),N,i}_s,\mathcal{S}^N_s)dW^{(n),i}_s\rangle_{\mathbb{H}}\Big|^r
\nonumber \\
\leq&&\!\!\!\!\!\!\!\!\frac{1}{N}\sum_{i=1}^N\mathbb{E}\Bigg\{\int_0^t\|X^{(n),N,i}_s\|_{\mathbb{H}}^2
\|\mathcal{B}(s,X^{(n),N,i}_s,\mathcal{S}^N_s)\|_{L_2(U,\mathbb{H})}^2ds\Bigg\}^{\frac{r}{2}}
\nonumber \\
\leq&&\!\!\!\!\!\!\!\!C_r\mathbb{E}\Big[\frac{1}{N}\sum_{i=1}^N\sup_{s\in[0,t]}\|X^{(n),N,i}_s\|_{\mathbb{H}}^{2r}\Big]+
C_r\mathbb{E}\Bigg\{\int_0^t\frac{1}{N}\sum_{i=1}^N\|X^{(n),N,i}_s\|_{\mathbb{H}}^{2r}ds\Bigg\}.
\end{eqnarray}
Combining estimates (\ref{es9})-(\ref{es10}), we have
\begin{eqnarray*}
\frac{1}{N}\sum_{i=1}^N\mathbb{E}\Bigg\{\int_0^T\|X^{(n),N,i}_t\|_{\mathbb{V}}^{\alpha} dt\Bigg\}^{r}
\leq&&\!\!\!\!\!\!\!\! C_{r,T}\Big(1+\mathbb{E}\Big(\frac{1}{N}\sum_{i=1}^N\|\xi^{(n),i}\|_{\mathbb{H}}^{2r}\Big)\Big)
\nonumber \\
&&\!\!\!\!\!\!\!\!+C_{r,T}\mathbb{E}\Big[\sup_{t\in[0,T]}\frac{1}{N}\sum_{i=1}^N\| X^{(n),N,i}_t\|_{\mathbb{H}}^{2r}\Big]
\nonumber \\
\leq&&\!\!\!\!\!\!\!\!C_{r,T}.
\end{eqnarray*}
By a similar argument, one can also get
\begin{eqnarray*}
\mathbb{E}\Bigg\{\int_0^T\|X^{(n),N,1}_t\|_{\mathbb{V}}^{\alpha} dt\Bigg\}^{r}\leq C_{r,T},
\end{eqnarray*}
which completes the proof.
\end{proof}

Set a stopping time $\tau_M^{(n),N}$ given by
\begin{eqnarray*}
\tau_M^{(n),N}:=&&\!\!\!\!\!\!\!\!\inf\Bigg\{t\geq0:\|X^{(n),N,1}_t\|_{\mathbb{H}}+\frac{1}{N}\sum_{j=1}^N\|X^{(n),N,j}_t\|_{\mathbb{H}}^{\lambda}>M\Bigg\}
\nonumber \\
&&\!\!\!\!\!\!\!\!\wedge\inf\Bigg\{t\geq0:\int_0^t \|X^{(n),N,1}_s\|_{\mathbb{V}}^{\alpha}ds+\int_0^t\frac{1}{N}\sum_{j=1}^N\|X^{(n),N,j}_s\|_{\mathbb{V}}^{\alpha}ds>M\Bigg\}\wedge T,~M>0.
\end{eqnarray*}

\begin{lemma}\label{lem2}
Suppose that the assumptions in Theorem \ref{th4} hold, then for any $M>0$ there exists a constant $C_{M}>0$ such that for any $n,N\in\mathbb{N}$,
\begin{eqnarray*}
\mathbb{E}\Big[\sup_{t\in[0,T\wedge\tau_M^{(n),N}]}\|X^{(n),N,1}_t\|_{\mathbb{V}}^2\Big]+\mathbb{E}\int_0^{T\wedge\tau_M^{(n),N}}\|X^{(n),N,1}_t\|_{\mathbb{X}}^\gamma dt
\leq C_{M,T}\big(1+\mathbb{E}\|\xi^1\|_{\mathbb{V}}^2\big).
\end{eqnarray*}
\end{lemma}
\begin{proof}
By It\^{o}'s formula, we have
\begin{eqnarray}\label{es12}
&&\!\!\!\!\!\!\!\!\|X^{(n),N,1}_t\|_{\mathbb{V}}^2
\nonumber \\
=&&\!\!\!\!\!\!\!\!\|X^{(n),N,1}_0\|_{\mathbb{V}}^2+\int_0^t\big(2\langle \mathcal{A}(s,X^{(n),N,1}_s,\mathcal{S}^{(n),N}_s),X^{(n),N,1}_s\rangle_{\mathbb{V}}
\nonumber \\
&&\!\!\!\!\!\!\!\!
+\|\pi_n\mathcal{B}(s,X^{(n),N,1}_s,\mathcal{S}^{(n),N}_s)\tilde{\pi}_n\|_{L_2(U,{\mathbb{V}})}^2\big)ds
\nonumber \\
&&\!\!\!\!\!\!\!\!+2\int_0^t\langle X^{(n),N,1}_s,\pi_n\mathcal{B}(s,X^{(n),N,1}_s,\mathcal{S}^{(n),N}_s)dW^{(n),1}_s\rangle_{\mathbb{V}}
\nonumber \\
\leq&&\!\!\!\!\!\!\!\!\|X^{(n),N,1}_0\|_{\mathbb{V}}^2-\delta_2\int_0^t\|X^{(n),N,1}_s\|_{\mathbb{X}}^{\gamma}ds
\nonumber \\
&&\!\!\!\!\!\!\!\!
+C\Bigg\{\int_0^t\|X^{(n),N,1}_s\|_{\mathbb{V}}^{\alpha+2}\|X^{(n),N,1}_s\|_{\mathbb{H}}^{\lambda}ds
\nonumber \\
&&\!\!\!\!\!\!\!\!
+\int_0^t\big(1+\|X^{(n),N,1}_s\|_{\mathbb{V}}^{\alpha}+ \frac{1}{N}\sum_{j=1}^N\|X^{(n),N,j}_s\|_{\mathbb{V}}^{\alpha}\big)
\nonumber \\
&&\!\!\!\!\!\!\!\!\cdot
\big(1+\|X^{(n),N,1}_s\|_{\mathbb{V}}^{2}+\|X^{(n),N,1}_s\|_{\mathbb{H}}^{\lambda}+\frac{1}{N}\sum_{j=1}^N\|X^{(n),N,j}_s\|_{\mathbb{H}}^{\lambda}\big)ds\Bigg\}
\nonumber \\
&&\!\!\!\!\!\!\!\!+2\Big|\int_0^t\langle X^{(n),N,1}_s,\pi_n\mathcal{B}(s,X^{(n),N,1}_s,\mathcal{S}^{(n),N}_s)dW^{(n),1}_s\rangle_{\mathbb{V}}\Big|
\nonumber \\
=:&&\!\!\!\!\!\!\!\!\|X^{(n),N,1}_0\|_{\mathbb{V}}^2-\delta_2\int_0^t\|X^{(n),N,1}_s\|_{\mathbb{X}}^{\gamma}ds+\text{(I)}+\text{(II)}.
\end{eqnarray}
For (I), we deduce that for any $t\in[0,T\wedge\tau_M^{(n),N}]$,
\begin{equation}\label{es13}
\text{(I)}\leq C_{M,T}+C_{M}\int_0^t\big(1+\|X^{(n),N,1}_s\|_{\mathbb{V}}^{\alpha}+\frac{1}{N}\sum_{j=1}^N\|X^{(n),N,j}_s\|_{\mathbb{V}}^{\alpha}\big)\|X^{(n),N,1}_s\|_{\mathbb{V}}^{2}ds.
\end{equation}
Using (\ref{es13}) and Gronwall's lemma to (\ref{es12}), it follows that for any $t\in[0,T\wedge\tau_M^{(n),N}]$,
\begin{eqnarray}\label{espro1}
&&\!\!\!\!\!\!\!\!\|X^{(n),N,1}_t\|_{\mathbb{V}}^2+\delta_2\int_0^{t}\|X^{(n),N,1}_t\|_{\mathbb{X}}^{\gamma}dt
\nonumber \\
~~~~~~~~~~~\leq &&\!\!\!\!\!\!\!\! C_{M,T}\big(1+\|\xi^1\|_{\mathbb{V}}^2+\text{(II)}\big)\cdot
 \exp\Bigg\{2\int_0^{t}\big(1+\|X^{(n),N,1}_s\|_{\mathbb{V}}^{\alpha}+\frac{1}{N}\sum_{j=1}^N\|X^{(n),N,j}_s\|_{\mathbb{V}}^{\alpha}\big)ds\Bigg\}
\nonumber \\
~~~~~~~~~~~\leq &&\!\!\!\!\!\!\!\! C_{M,T}\big(1+\|\xi^1\|_{\mathbb{V}}^2+\text{(II)}\big).
\end{eqnarray}
Using  Burkholder-Davis-Gundy's inequality, we obtain
\begin{eqnarray}\label{espro2}
&&\!\!\!\!\!\!\!\!\mathbb{E}\Bigg\{\sup_{t\in[0, T\wedge\tau_M^{(n),N}]}\text{(II)}\Bigg\}
\nonumber \\
\leq&&\!\!\!\!\!\!\!\!C\mathbb{E}\Bigg\{\int_0^{T\wedge\tau_M^{(n),N}}\|X^{(n),N,1}_s\|_{\mathbb{V}}^{2}\|\mathcal{B}(s,X^{(n),N,1}_s,\mathcal{S}^{(n),N}_s)\|_{L_2(U,{\mathbb{V}})}^2ds\Bigg\}^{\frac{1}{2}}
\nonumber \\
\leq&&\!\!\!\!\!\!\!\!\frac{1}{2}\mathbb{E}\Big[\sup_{t\in[0,T\wedge\tau_M^{(n),N}]} \|X^{(n),N,1}_t\|_{\mathbb{V}}^2\Big]
\nonumber \\
&&\!\!\!\!\!\!\!\!
+C\mathbb{E}\Bigg\{\int_0^{T\wedge\tau_M^{(n),N}}\big(1+\|X^{(n),N,1}_t\|_{\mathbb{V}}^{2}+\frac{1}{N}\sum_{j=1}^N\|X^{(n),N,j}_t\|_{\mathbb{V}}^{\alpha}\big)
\nonumber \\
&&\!\!\!\!\!\!\!\!\cdot
\big(1+\|X^{(n),N,1}_t\|_{\mathbb{H}}^{\lambda}+\frac{1}{N}\sum_{j=1}^N\|X^{(n),N,j}_t\|_{\mathbb{H}}^{\lambda}\big)dt\Bigg\}.
\nonumber \\
\leq&&\!\!\!\!\!\!\!\!C_{M,T}+\frac{1}{2}\mathbb{E}\Big[\sup_{t\in[0,T\wedge\tau_M^{(n),N}]} \|X^{(n),N,1}_t\|_{\mathbb{V}}^2\Big]
\nonumber \\
&&\!\!\!\!\!\!\!\!+C_M\mathbb{E}\int_0^{T\wedge\tau_M^{(n),N}}\|X^{(n),N,1}_t\|_{\mathbb{V}}^{\alpha}ds +C_M\mathbb{E}\int_0^{T\wedge\tau_M^{(n),N}}\frac{1}{N}\sum_{j=1}^N\|X^{(n),N,j}_s\|_{\mathbb{V}}^{\alpha}ds .
\end{eqnarray}
Combining (\ref{espro1})-(\ref{espro2}), we obtain
\begin{eqnarray*}
&&\!\!\!\!\!\!\!\!\mathbb{E}\Big[\sup_{t\in[0,T\wedge\tau_M^{(n),N}]}\|X^{(n),N,1}_t\|_{\mathbb{V}}^2\Big]+2\delta_2\mathbb{E}\int_0^{T\wedge\tau_M^{(n),N}}\|X^{(n),N,1}_t\|_{\mathbb{X}}^{\gamma}dt
\nonumber \\
\leq &&\!\!\!\!\!\!\!\!C_{M,T}\big(1+\mathbb{E}\|\xi^1\|_{\mathbb{V}}^2\big)+C_{M,T}\mathbb{E}\int_0^{T\wedge\tau_M^{(n),N}}\|X^{(n),N,1}_t\|_{\mathbb{V}}^{2}ds.
\end{eqnarray*}
Hence, by Gronwall's lemma we have
$$\mathbb{E}\Big[\sup_{t\in[0,T\wedge\tau_M^{(n),N}]}\|X^{(n),N,1}_t\|_{\mathbb{V}}^2\Big]+2\delta_2\mathbb{E}\int_0^{T\wedge\tau_M^{(n),N}}\|X^{(n),N,1}_t\|_{\mathbb{X}}^{\gamma}dt
\leq C_{M,T}\big(1+\mathbb{E}\|\xi^1\|_{\mathbb{V}}^2\big).$$
We complete the proof.
\end{proof}

\subsection{Passage to the original system}\label{sec4.2}
For the convenience of reading, we first provide a structural description of the argument in this subsection. The aim of this subsection is to pass to the limit of the Galerkin system $\{X^{(n),N,i},n,N\in\mathbb{N}\}$ to the original system $\{X^{N,i},N\in\mathbb{N}\}$ of (\ref{eqi}), as $n\to\infty$. To this end, we first investigate the tightness of laws of the sequence $\{X^{(n),N,1},n,N\in\mathbb{N}\}$ in the state space $\mathbb{S}$. Then, by the exchangeability property, we can deduce that
the laws of $\{(X^{(n),N,1},\ldots,X^{(n),N,N}),n\in\mathbb{N}\}$ are  tight in $\mathbb{S}^{\otimes N}$, which indicates that there exists a weakly convergent subsequence. Next, by employing the martingale approach we can show that the limit of this weakly convergent subsequence is a martingale  solution to the original system (\ref{eqi}), which together with the weak uniqueness of (\ref{eqi}) and the uniform integrability imply the convergence of the whole  sequence w.r.t.~certain Wasserstein metrics.

\begin{lemma}\label{lem12}
Under the assumptions in Theorem \ref{th4}, we have

\vspace{1mm}
(i) for $N\geq 2$ fixed, the sequence  $\{X^{(n),N,1},n\in\mathbb{N}\}$ is  tight in $\mathbb{S}$,

\vspace{2mm}
(ii) the sequence  $\{X^{(n),N,1},n,N\in\mathbb{N}\}$ is  tight in $\mathbb{S}$.
\end{lemma}
\begin{proof}
We first prove the claim (ii). Since the proof is exactly similar to that of
Lemma \ref{coro1}, we only include the proof of tightness of $\{X^{(n),N,1},n,N\in\mathbb{N}\}$ in $\mathbb{C}_T$ for reader's understanding.

Recall the definition of stopping time $\tau_M^{(n),N}$. By Lemma \ref{lem11}, we can get
  \begin{eqnarray}\label{es2}
&&\!\!\!\!\!\!\!\!\sup_{n,N\in\mathbb{N}}\mathbb{P}\big(\tau_M^{(n),N}<T\big)
\nonumber \\
\leq&&\!\!\!\!\!\!\!\! \sup_{n,N\in\mathbb{N}}\Bigg\{\mathbb{E}\Big[\sup_{t\in[0,T]}\|X^{(n),N,1}_t\|_{\mathbb{H}}\Big]+\mathbb{E}\Big[\sup_{t\in[0,T]}\frac{1}{N}\sum_{j=1}^N\|X^{(n),N,j}_t\|_{\mathbb{H}}^{\lambda}\Big]
\nonumber \\
&&\!\!\!\!\!\!\!\!+\mathbb{E}\int_0^T \|X^{(n),N,1}_s\|_{\mathbb{V}}^{\alpha}ds+\mathbb{E}\int_0^T\frac{1}{N}\sum_{j=1}^N\|X^{(n),N,j}_s\|_{\mathbb{V}}^{\alpha}ds\Bigg\}
\Big/M
\nonumber \\
\leq&&\!\!\!\!\!\!\!\! C_{p,T}/M,
  \end{eqnarray}
where the constant $C_{p,T}$ is independent of $M$.

\vspace{1mm}
Then for any $R>0$, by Lemma \ref{lem2}, we have
  \begin{eqnarray}
&&\!\!\!\!\!\!\!\!\sup_{n,N\in\mathbb{N}}\mathbb{P}\Big(\sup_{0\leq t\leq T}\|X^{(n),N,1}_t\|_{\mathbb{V}}>R\Big)
\nonumber \\
\leq &&\!\!\!\!\!\!\!\!
\sup_{n,N\in\mathbb{N}}\mathbb{P}\Big(\sup_{0\leq t\leq T\wedge\tau_M^{(n),N}}\|X^{(n),N,1}_t\|_{\mathbb{V}}>R\Big)+\sup_{n,N\in\mathbb{N}}\mathbb{P}\big(\tau_M^{(n),N}<T\big)
\nonumber \\
\leq &&\!\!\!\!\!\!\!\!\frac{1}{R^2}\sup_{n,N\in\mathbb{N}}\mathbb{E}\Big[\sup_{0\leq t\leq T\wedge\tau_M^{(n),N}}\|X^{(n),N,1}_t\|_{\mathbb{V}}^2\Big]+\frac{C_{p,T}}{M}
\nonumber \\
\leq &&\!\!\!\!\!\!\!\!\frac{C_{p,M,T}}{R^2}+\frac{C_{p,T}}{M}.
\end{eqnarray}
Let $R\rightarrow\infty$ first, then let $M\rightarrow\infty$, we can see that
$$\lim_{R\rightarrow\infty}\sup_{n,N\in\mathbb{N}}\mathbb{P}\left(\sup_{0\leq t\leq T}\|X^{(n),N,1}_t\|_{\mathbb{V}}>R\right)=0.$$
As stated in Lemma \ref{lem6}, it suffices to
show that for every $e\in \mathbb{H}_m$, $m\geq 1$ and any stopping time $0\leq \zeta_n\leq T$ and $\varepsilon>0$,
\begin{equation}\label{es3}
\lim_{\Delta\to 0}\sup_{n,N\in\mathbb{N}}\mathbb{P}\Big(\big|\langle X^{(n),N,1}_{\overline{\zeta_n+\Delta}}-X^{(n),N,1}_{\zeta_n},e\rangle_{\mathbb{H}}\big|>\varepsilon\Big)=0,
\end{equation}
where $\overline{\zeta_n+\Delta}:=T\wedge (\zeta_n+\Delta) \vee 0$.

 Note that
\begin{eqnarray}\label{es1}
&&\!\!\!\!\!\!\!\!\mathbb{P}\Big(\big|\langle X^{(n),N,1}_{\overline{\zeta_n+\Delta}}-X^{(n),N,1}_{\zeta_n},e\rangle_{\mathbb{H}}\big|>\varepsilon\Big)
\nonumber \\
\leq &&\!\!\!\!\!\!\!\!\mathbb{P}\Big(\big|\langle X^{(n),N,1}_{\overline{\zeta_n+\Delta}\wedge  \tau_M^{(n),N}}-X^{(n),N,1}_{\zeta_n\wedge \tau_M^{(n),N}},e\rangle_{\mathbb{H}}\big|>\varepsilon\Big)+\mathbb{P}\big(\tau_M^{(n),N}<T\big)
\nonumber \\
\leq &&\!\!\!\!\!\!\!\!\mathbb{P}\Big(\Big| \int_{\zeta_n\wedge \tau_M^{(n),N}}^{\overline{\zeta_n+\Delta}\wedge \tau_M^{(n),N}}\langle\pi_n \mathcal{A}(s,X^{(n),N,1}_s,\mathcal{S}^{(n),N}_s),e\rangle_{\mathbb{H}}ds\Big|>\varepsilon\Big)
\nonumber \\
 &&\!\!\!\!\!\!\!\!+\mathbb{P}\Big(\Big| \int_{\zeta_n\wedge \tau_M^{(n),N}}^{\overline{\zeta_n+\Delta}\wedge \tau_M^{(n),N}}\langle\pi_n \mathcal{B}(s,X^{(n),N,1}_s,\mathcal{S}^{(n),N}_s)dW^{(n),1}_s,e\rangle_{\mathbb{H}}\Big|>\varepsilon\Big)
\nonumber \\
 &&\!\!\!\!\!\!\!\!+\mathbb{P}\big(\tau_M^{(n),N}<T\big)
\nonumber \\
\leq &&\!\!\!\!\!\!\!\!\frac{1}{\varepsilon^{\frac{\alpha}{\alpha-1}}}\mathbb{E}\Big| \int_{\zeta_n\wedge \tau_M^{(n),N}}^{\overline{\zeta_n+\Delta}\wedge  \tau_M^{(n),N}}{}_{\mathbb{V}^*}\langle\pi_n \mathcal{A}(s,X^{(n),N,1}_s,\mathcal{S}^{(n),N}_s),e\rangle_{\mathbb{V}}ds\Big|^{\frac{\alpha}{\alpha-1}}
\nonumber \\
 &&\!\!\!\!\!\!\!\!+\frac{1}{\varepsilon^2}\mathbb{E}\Big|\int_{\zeta_n\wedge \tau_M^{(n),N}}^{\overline{\zeta_n+\Delta}\wedge \tau_M^{(n),N}}\langle\pi_n \mathcal{B}(s,X^{(n),N,1}_s,\mathcal{S}^{(n),N}_s)dW^{(n),1}_s,e\rangle_{\mathbb{H}}\Big|^2
\nonumber \\
 &&\!\!\!\!\!\!\!\!+\mathbb{P}\big(\tau_M^{(n),N}<T\big)
\nonumber \\
=: &&\!\!\!\!\!\!\!\!\frac{1}{\varepsilon^{\frac{\alpha}{\alpha-1}}}\text{(I)}+\frac{1}{\varepsilon^2}\text{(II)}+\mathbb{P}\big(\tau_M^{(n),N}<T\big).
\end{eqnarray}
For $\text{(I)}$, since $e\in\mathbb{H}_m\subset\mathbb{V}$, by H\"{o}lder's inequality we have
\begin{eqnarray}\label{sec4esq2}
\text{(I)}\leq  &&\!\!\!\!\!\!\!\!C |\Delta|^{\frac{1}{\alpha-1}}\cdot \mathbb{E}\Bigg\{\int_{\zeta_n\wedge \tau_M^{(n),N}}^{\overline{\zeta_n+\Delta}\wedge  \tau_M^{(n),N}}|{}_{\mathbb{V}^*}\langle \pi_n \mathcal{A}(s,X^{(n),N,1}_s,\mathcal{S}^{(n),N}_s), e\rangle_{\mathbb{V}}|^{\frac{\alpha}{\alpha-1}}ds\Bigg\}
\nonumber \\
\leq &&\!\!\!\!\!\!\!\!C_{\|e\|_{\mathbb{V}}} |\Delta|^{\frac{1}{\alpha-1}}\cdot\mathbb{E}\Bigg\{\int_{0}^{T\wedge \tau_M^{(n),N}}\big(1+\|X_{s}^{(n),N,1}\|_{\mathbb{V}}^\alpha+\frac{1}{N}\sum_{j=1}^N \|X_{s}^{(n),N,j}\|_{\mathbb{V}}^{\alpha}\big)
\nonumber \\
&&\!\!\!\!\!\!\!\!\cdot
\big(1+\|X_{s}^{(n),N,1}\|_{\mathbb{H}}^{\beta}+\frac{1}{N}\sum_{j=1}^N \|X_{s}^{(n),N,j}\|_{\mathbb{H}}^{\beta}\big)
 ds\Bigg\}
\nonumber \\
 \leq&&\!\!\!\!\!\!\!\!C_{M,T,\|e\|_{\mathbb{V}}} |\Delta|^{\frac{1}{\alpha-1}}.
\end{eqnarray}
For $\text{(II)}$, we can get
\begin{eqnarray}\label{sec4esq3}
\text{(II)}\leq  &&\!\!\!\!\!\!\!\!C\mathbb{E}\Bigg\{\int_{\zeta_n\wedge \tau_M^{(n),N}}^{\overline{\zeta_n+\Delta}\wedge  \tau_M^{(n),N}}\|e\|_{\mathbb{H}}^2\Big(1+\|X_{s}^{(n),N,1}\|_{\mathbb{H}}^2+\frac{1}{N}\sum_{j=1}^N \|X_{s}^{(n),N,j}\|_{\mathbb{H}}^2\Big)ds\Bigg\}
\nonumber \\
 \leq&&\!\!\!\!\!\!\!\!C_{M,T,\|e\|_{\mathbb{H}}}|\Delta|.
\end{eqnarray}
Combining (\ref{es2}), (\ref{es1})-(\ref{sec4esq3}) and  letting $\Delta\to 0$ and then $M\to\infty$ implies (\ref{es3}) holds.

\vspace{1mm}
It is clear that the assertion (ii) trivially implies (i).  Thus we finish the proof.
\end{proof}

From now on, let us fix $N\geq 2$.
By Lemma \ref{lem12} (i), we have obtained that the family $\{X^{(n),N,1},n\in\mathbb{N}\}$ is  tight in $\mathbb{S}$. By the exchangeability property, we deduce that
$\{(X^{(n),N,1},\ldots,X^{(n),N,N}),n\in\mathbb{N}\}$ is  tight in $\mathbb{S}^{\otimes N}$. Consequently,
$$\big\{\big((X^{(n),N,1},W^1),\ldots,(X^{(n),N,N},W^N)\big),n\in\mathbb{N}\big\}~\text{is  tight in}~ (\mathbb{S}\times C([0,T];U_1))^{\otimes N},$$
 where $U_1$ is a Hilbert space such that the embedding $U\subset U_1$ is Hilbert-Schmidt. For any sequence $\{n\}_{n\in\mathbb{N}}$, there exists a subsequence  denoted by $\{n_k\}_{k\in\mathbb{N}}$ such that the family
 \begin{eqnarray}\label{conver1}
 &&\!\!\!\!\!\!\!\!\big((X^{(n_k),N,1},W^1),\ldots,(X^{(n_k),N,N},W^N)\big)~\text{converges in law}~
 \nonumber\\&&\!\!\!\!\!\!\!\!\text{in}~(\mathbb{S}\times C([0,T];U_1))^{\otimes N},~~\text{as}~k\to\infty.
\end{eqnarray}
Applying Skorokhod representation theorem, there exists a probability space still denoted  by $(\Omega,\mathscr{F},\mathbb{P})$, and on this space, $(\mathbb{S}\times C([0,T];U_1))^{\otimes N}$-valued random variables $$\big((X^{(n_k),N,1},W^{(n_k),1}),\ldots,(X^{(n_k),N,N},W^{(n_k),N})\big)$$  such that
\begin{equation*}\label{conver2}
\big((X^{(n_k),N,1},W^{(n_k),1}),\ldots,(X^{(n_k),N,N},W^{(n_k),N})\big)\to\big((X^{N,1},W^{1}),\ldots,(X^{N,N},W^{N})\big),~\mathbb{P}\text{-a.s.},
\end{equation*}
$\text{in}~(\mathbb{S}\times C([0,T];U_1))^{\otimes N},~\text{as}~k\to\infty$. Next, by the martingale problem approach (following from similar argument as the proof of Theorem \ref{th1}) and the martingale representation theorem, we can prove
\begin{equation}\label{eqi3}
\Xi^N:=\big((X^{N,1},W^{1}),\ldots,(X^{N,N},W^{N})\big)
\end{equation}
is a weak solution to the coupled system (\ref{eqi}).
Then Theorem \ref{th5} implies  $\Xi^N$ is the unique weak solution of  (\ref{eqi}) and thus the convergence (\ref{conver1}) holds for whole sequence $\{n\}_{n\in\mathbb{N}}$.

 \vspace{1mm}
 Moreover, by Lemmas \ref{lem11} and \ref{lem3}, the lower semicontinuity of norms (cf.~(\ref{eslow})) and Fatou's lemma, it follows that there exists $C_{p,T}>0$ independent of $N$ and  $r>1$ such that
\begin{equation}\label{es4}
\mathbb{E}\Big[\sup_{t\in[0,T]}\frac{1}{N}\sum_{i=1}^N\|X^{N,i}_t\|_{\mathbb{H}}^p\Big]+\frac{1}{N}\sum_{i=1}^N\mathbb{E}\Bigg\{\int_0^T\|X^{N,i}_t\|_{\mathbb{V}}^{\alpha} dt\Bigg\}^{r} \leq C_{p,T},
\end{equation}
\begin{equation}\label{es5}
\mathbb{E}\Big[\sup_{t\in[0,T]}\|X^{N,1}_t\|_{\mathbb{H}}^p\Big]+\mathbb{E}\Bigg\{\int_0^T\|X^{N,1}_t\|_{\mathbb{V}}^{\alpha} dt\Bigg\}^{r}\leq C_{p,T},
\end{equation}
and
\begin{eqnarray}
&&\!\!\!\!\!\!\!\!\mathbb{E}\Big\{\int_0^{T}\frac{1}{N}\sum_{i=1}^N\Big(\|X^{N,i}_t\|_{\mathbb{H}}^{p-2}\|X^{N,i}_t\|_{\mathbb{V}}^{\alpha}\Big)dt\Big\}
\nonumber \\ &&\!\!\!\!\!\!\!\!+\mathbb{E}\Big\{\int_0^{T}\|X^{N,1}_t\|_{\mathbb{H}}^{p-2}\|X^{N,1}_t\|_{\mathbb{V}}^{\alpha}dt\Big\}\leq C_{p,T},\label{es6}
\end{eqnarray}
where $p$ is defined in Theorem \ref{th4}.

\vspace{2mm}
In the next subsection, we want to characterize the limit of the empirical law $\mathcal{S}^N$ and prove Theorem \ref{th4}.

\subsection{Passage to the limit of $\mathcal{S}^N$}\label{sec4.3}
In this part, we consider the $N$-interacting system $X^{N}$ built in (\ref{eqi3}) with initial value $X_0^N=(\xi^1,\ldots,\xi^N)$, and  derive the tightness of $\{\mathcal{S}^{N},N\in\mathbb{N}\}$ in $\mathbb{K}:=\mathscr{P}_{2}(\mathbb{C}_T)\cap \mathscr{P}_{\alpha}(L^{\alpha}([0,T];\mathbb{V}))$ by the following  lemma.

\begin{lemma}\label{pro13}
Under the assumptions in Theorem \ref{th4}, the sequence  $\{\mathcal{S}^{N},N\in\mathbb{N}\}$ is  tight in $\mathbb{K}$.
\end{lemma}

\begin{proof}
 Recall that for each $N\geq 2$, $(X^{N,i})_{1\leq i\leq N}$ has been obtained as a limit point (in law) of $(X^{(n),N,i})_{1\leq i\leq N}$ as $n\to \infty$.
 In view of Lemma \ref{lem12} (ii), the family $\{X^{N,1},N\in\mathbb{N}\}$ is thus tight in $\mathbb{S}$.
 According to the classical result from  Sznitman \cite{S1} that the tightness of  $\{\mathcal{S}^{N},N\in\mathbb{N}\}$ in $\mathscr{P}(\mathbb{S})$ is equivalent to the tightness of $\{X^{N,1},N\in\mathbb{N}\}$ in $\mathbb{S}$.

\vspace{2mm}
To proceed,  we prove that $\{\mathcal{S}^{N},N\in\mathbb{N}\}$ is  tight in $\mathbb{K}$. The proof contains twofold:

\vspace{2mm}
(i) the tightness  in $\mathscr{P}_{2}(\mathbb{C}_T)$;

\vspace{2mm}
(ii)  the tightness  in $\mathscr{P}_{\alpha}(L^{\alpha}([0,T];\mathbb{V}))$.

\vspace{2mm}
\textbf{Step 1.} In this step, we prove claim (i). Since $\{\mathcal{S}^{N},N\in\mathbb{N}\}$ is  tight in $\mathscr{P}(\mathbb{S})$ (thus in $\mathscr{P}(\mathbb{C}_T)$) and $\{\mathcal{S}^{N},N\in\mathbb{N}\}\subset \mathscr{P}_{2}(\mathbb{C}_T)$, $\mathbb{P}$-a.s., there exists a set $\mathscr{R}_\varepsilon^1\subset\mathscr{P}_{2}(\mathbb{C}_T)$ that is relatively compact in $\mathscr{P}(\mathbb{C}_T)$ such that
\begin{equation}\label{es1118}
\mathbb{P}(\mathcal{S}^{N}\notin \mathscr{R}_\varepsilon^1)\leq \varepsilon.
\end{equation}
Now for some $r_1>2$, let
$$\lambda_m:=m^{\frac{1}{r_1-2}}2^{\frac{m}{r_1-2}},~~~~~\kappa_m:=\frac{\varepsilon m}{\sup_{N\in\mathbb{N}}\mathbb{E}\Big[\sup_{t\in[0,T]}\|X^{N,1}_t\|_{{\mathbb{H}}}^{r_1}\Big]\vee 1},$$
 and set
$$\Upsilon_\varepsilon^1:=\bigcap_{m\in\mathbb{N}}\Bigg\{\nu\in\mathscr{P}_{2}(\mathbb{C}_T):\int  d_{T}^1(w,0)^2\mathbf{1}_{\{d_{T}^1(w,0)\geq \lambda_m\}}\nu(dw)<\frac{1}{\kappa_m}\Bigg\}.$$
Then we deduce from (\ref{es4}) that
\begin{eqnarray}\label{es1119}
\mathbb{P}(\mathcal{S}^{N}\notin\Upsilon_\varepsilon^1)\leq &&\!\!\!\!\!\!\!\! \sum_{m=1}^\infty\mathbb{P}\Big(\frac{1}{N}\sum_{i=1}^Nd_{T}^1(X^{N,i},0)^2\mathbf{1}_{\{d_{T}^1(X^{N,i},0)\geq \lambda_m\}}\geq\frac{1}{\kappa_m}\Big)
\nonumber \\
\leq&&\!\!\!\!\!\!\!\!\sum_{m=1}^\infty \frac{\kappa_m}{N} \sum_{i=1}^N \mathbb{E}\Big[d_{T}^1(X^{N,i},0)^2\mathbf{1}_{\{d_{T}^1(X^{N,i},0)\geq \lambda_m\}}\Big]
\nonumber \\
\leq&&\!\!\!\!\!\!\!\!\sum_{m=1}^\infty\frac{\kappa_m}{\lambda_m^{r_1-2}N}\sum_{i=1}^N\mathbb{E}\Big[\sup_{t\in[0,T]}\|X^{N,i}_t\|_{{\mathbb{H}}}^{r_1}\Big]
\nonumber \\
\leq&&\!\!\!\!\!\!\!\!\varepsilon.
\end{eqnarray}
Thus combining (\ref{es1118})-(\ref{es1119}), it is easy to see
\begin{equation*}
\mathbb{P}(\mathcal{S}^{N}\notin \mathscr{R}_\varepsilon^1\cap\Upsilon_\varepsilon^1)\leq 2\varepsilon.
\end{equation*}
According to Lemma \ref{lem00} in Appendix, the set $\mathscr{R}_\varepsilon^1\cap\Upsilon_\varepsilon^1$ is relatively compact in $\mathscr{P}_{2}(\mathbb{C}_T)$, therefore we conclude that the sequence $\{\mathcal{S}^{N},N\in\mathbb{N}\}$ is  tight in $\mathscr{P}_{2}(\mathbb{C}_T)$.

\vspace{2mm}
\textbf{Step 2.} Similar to Step 1, we prove claim (ii). Since $\{\mathcal{S}^{N},N\in\mathbb{N}\}$ is  tight in $\mathscr{P}(L^{\alpha}([0,T];\mathbb{V}))$  and $\{\mathcal{S}^{N},N\in\mathbb{N}\}\subset \mathscr{P}_{\alpha}(L^{\alpha}([0,T];\mathbb{V}))$, $\mathbb{P}$-a.s., there exists a set $\mathscr{R}_\varepsilon^2\subset\mathscr{P}_{\alpha}(L^{\alpha}([0,T];\mathbb{V}))$ that is relatively compact in $\mathscr{P}(L^{\alpha}([0,T];\mathbb{V}))$ such that
\begin{equation}\label{es31}
\mathbb{P}(\mathcal{S}^{N}\notin \mathscr{R}_\varepsilon^2)\leq \varepsilon.
\end{equation}
Now for some $r_2>1$, let
$$\lambda_m:=m^{\frac{1}{r_2-1}}2^{\frac{m}{r_2-1}},~~~~~\kappa_m:=\frac{\varepsilon m}{\sup_{N\in\mathbb{N}}\mathbb{E}\Big\{\int_0^T\|X^{N,1}_t\|_{{\mathbb{V}}}^{\alpha}dt\Big\}^{r_2}\vee 1},$$
 and set
$$\Upsilon_\varepsilon^2:=\bigcap_{m\in\mathbb{N}}\Bigg\{\nu\in\mathscr{P}_{\alpha}(L^{\alpha}([0,T];\mathbb{V})):\int  d_{T}^2(w,0)^\alpha\mathbf{1}_{\{d_{T}^2(w,0)^{\alpha}\geq \lambda_m\}}\nu(dw)<\frac{1}{\kappa_m}\Bigg\}.$$
It follows from (\ref{es4}) that
\begin{eqnarray}\label{es32}
\mathbb{P}(\mathcal{S}^{N}\notin\Upsilon_\varepsilon^2)\leq &&\!\!\!\!\!\!\!\! \sum_{m=1}^\infty\mathbb{P}\Big(\frac{1}{N}\sum_{i=1}^Nd_{T}^2(X^{N,i},0)^\alpha\mathbf{1}_{\{d_{T}^2(X^{N,i},0)^{\alpha}\geq \lambda_m\}}\geq\frac{1}{\kappa_m}\Big)
\nonumber \\
\leq&&\!\!\!\!\!\!\!\!\sum_{m=1}^\infty \frac{\kappa_m}{N} \sum_{i=1}^N \mathbb{E}\Big[d_{T}^2(X^{N,i},0)^\alpha\mathbf{1}_{\{d_{T}^2(X^{N,i},0)^{\alpha}\geq \lambda_m\}}\Big]
\nonumber \\
\leq&&\!\!\!\!\!\!\!\!\sum_{m=1}^\infty\frac{\kappa_m}{\lambda_m^{r_2-1}N}\sum_{i=1}^N\mathbb{E}\Bigg\{\int_0^T\|X^{N,i}_t\|_{{\mathbb{V}}}^{\alpha}dt\Bigg\}^{r_2}
\nonumber \\
\leq&&\!\!\!\!\!\!\!\!\varepsilon.
\end{eqnarray}
Combining (\ref{es31})-(\ref{es32}), it leads to
\begin{equation*}
\mathbb{P}(\mathcal{S}^{N}\notin \mathscr{R}_\varepsilon^2\cap\Upsilon_\varepsilon^2)\leq 2\varepsilon.
\end{equation*}
By Lemma \ref{lem00}, the set $\mathscr{R}_\varepsilon^2\cap\Upsilon_\varepsilon^2$ is relatively compact in $\mathscr{P}_{\alpha}(L^{\alpha}([0,T];\mathbb{V}))$. Thus the sequence $\{\mathcal{S}^{N},N\in\mathbb{N}\}$ is  tight in $\mathscr{P}_{\alpha}(L^{\alpha}([0,T];\mathbb{V}))$.
\end{proof}

\vspace{2mm}
In view of Lemma \ref{pro13},  we can get that for any sequence $\{N\}_{N\in\mathbb{N}}$, there exists a subsequence (in the sequel, for convenience, we still denoted by $\{N\}_{N\in\mathbb{N}}$) such that $\mathcal{S}^{N}$ converges weakly to $\mathcal{S}$  in $\mathbb{K}$, i.e.,
$$\text{the laws of}~\mathcal{S}^{N}~\text{converges weakly to the law of}~\mathcal{S}~\text{in}~\mathscr{P}(\mathbb{K}).$$
In addition, there exists $r_1'>2$ such that
\begin{eqnarray*}
\mathbb{E}\Bigg\{\int d_T^1(w,0)^{r_1'}\mathcal{S}(dw)\Bigg\}\leq &&\!\!\!\!\!\!\!\! \mathbb{E}\Bigg\{\liminf_{M\to\infty}\int (d_T^1(w,0)^{r_1'}\wedge M)\mathcal{S}(dw)\Bigg\}
\nonumber \\
\leq &&\!\!\!\!\!\!\!\! \liminf_{M\to\infty}\liminf_{N\to\infty}\mathbb{E}\Bigg\{\int (d_T^1(w,0)^{r_1'}\wedge M)\mathcal{S}^{N}(dw)\Bigg\}
\nonumber \\
\leq&&\!\!\!\!\!\!\!\! C\sup_{N\in\mathbb{N}}\mathbb{E}\Big[\sup_{t\in[0,T]}\frac{1}{N}\sum_{i=1}^N\|X^{N,i}_t\|_{\mathbb{H}}^{r_1'}\Big]
\nonumber \\
<&&\!\!\!\!\!\!\!\!\infty,
\end{eqnarray*}
where we used Fatou's lemma in the first and second inequalities,  the last step is due to (\ref{es4}). Similarly, there exists $r_2'>\alpha$ such that
\begin{equation*}
\mathbb{E}\Bigg\{\int d_T^2(w,0)^{r_2'}\mathcal{S}(dw)\Bigg\}<\infty.
\end{equation*}
Hence, it follows that
\begin{equation}\label{es1132}
\mathcal{S}\in \mathscr{P}_{r_1'}(\mathbb{C}_T)\cap \mathscr{P}_{r_2'}(L^{\alpha}([0,T];\mathbb{V}))\subset \mathbb{K},~\mathbb{P}\text{-a.s.}.
\end{equation}

Now, for any $(w,\nu,l)\in \mathbb{S}\times\mathbb{K}\times \mathbb{V}$, we recall the following notations defined in subsection \ref{sec2.4},
\begin{equation}\label{eq11m}
\mathcal{M}_l(t,w,\nu),\varpi_R(t,w,\nu),\varpi(t,w,\nu).
\end{equation}
Fix $m\in\mathbb{N}$. Let $\Phi_1,\ldots,\Phi_m:\mathbb{H}\to\mathbb{R}$ be a real-valued function sequence in $C_b(\mathbb{H};\mathbb{R})$. Let $\nu\in\mathbb{K}$, $0\leq s<t\leq T$, $0\leq s_1<\cdots<s_m\leq s$, we define  maps
\begin{eqnarray*}
&&\!\!\!\!\!\!\!\!\Psi_R(\nu):=\int \Big(\varpi_R(t,w,\nu)-\varpi_R(s,w,\nu)\Big)\Phi_1(w_{s_1})\cdots\Phi_m(w_{s_m})\nu(dw),
\nonumber\\
&&\!\!\!\!\!\!\!\!
\Psi(\nu):=\int \Big(\varpi(t,w,\nu)-\varpi(s,w,\nu)\Big)\Phi_1(w_{s_1})\cdots\Phi_m(w_{s_m})\nu(dw).
\end{eqnarray*}

By Lemma \ref{lem8} and the fact that  $\Phi_1,\ldots,\Phi_m$ are bounded and continuous, one can easily prove \begin{lemma}\label{lem13}
For any $t\in[0,T]$,  $\nu^n,\nu\in\mathbb{K}$ with $\nu^n\to\nu$ in $\mathbb{K}$, as $n\to \infty$,
we have
$$\lim_{n\to \infty} \Psi_R(\nu^n)=\Psi_R(\nu).$$
\end{lemma}

From (\ref{es1132}), we have $\mathbb{P}$-a.s.
\begin{eqnarray}
&&\!\!\!\!\!\!\!\!~~~~~~\Psi_R(\mathcal{S}^N)=\frac{1}{N}\sum_{i=1}^{N}\Big(\varpi_R(t,X^{N,i},\mathcal{S}^N)-\varpi_R(s,X^{N,i},\mathcal{S}^N)\Big)\Phi_1(X^{N,i}_{s_1})\cdots\Phi_m(X^{N,i}_{s_m}),
\nonumber\\
&&\!\!\!\!\!\!\!\!
~~~~~~~~~~\Psi(\mathcal{S}^N)=\frac{1}{N}\sum_{i=1}^{N}\Big(\varpi(t,X^{N,i},\mathcal{S}^N)-\varpi(s,X^{N,i},\mathcal{S}^N)\Big)\Phi_1(X^{N,i}_{s_1})\cdots\Phi_m(X^{N,i}_{s_m}),\label{es1129}
\\
&&\!\!\!\!\!\!\!\!
~~~~~~\Psi_R(\mathcal{S})=\int \Big(\varpi_R(t,w,\mathcal{S})-\varpi_R(s,w,\mathcal{S})\Big)\Phi_1(w_{s_1})\cdots\Phi_m(w_{s_m})\mathcal{S}(dw),
\nonumber\\
&&\!\!\!\!\!\!\!\!
~~~~~~~~~~\Psi(\mathcal{S})=\int \Big(\varpi(t,w,\mathcal{S})-\varpi(s,w,\mathcal{S})\Big)\Phi_1(w_{s_1})\cdots\Phi_m(w_{s_m})\mathcal{S}(dw).\label{es1130}
\end{eqnarray}
In the sequel, we proceed to prove that $\mathbb{P}$-a.s.~$\Psi(\mathcal{S})=0$.

\begin{lemma}\label{lem4}
Under the assumptions in Theorem \ref{th4}, we have
\begin{equation}\label{es21}
\mathbb{E}|\Psi(\mathcal{S}^N)|\to\mathbb{E}|\Psi(\mathcal{S})|~~\text{as}~N\to\infty,
\end{equation}
and
\begin{equation}\label{es26}
\mathbb{E}|\Psi(\mathcal{S}^N)|^2\to0~~\text{as}~N\to\infty.
\end{equation}
\end{lemma}
\begin{proof}
\textbf{Step 1.}  We first prove (\ref{es21}). Note that
\begin{eqnarray}\label{es1124}
\mathbb{E}|\Psi(\mathcal{S}^N)|-\mathbb{E}|\Psi(\mathcal{S})|=&&\!\!\!\!\!\!\!\!\mathbb{E}\Big[|\Psi(\mathcal{S}^N)|-|\Psi_R(\mathcal{S}^N)|\Big]+\mathbb{E}\Big[|\Psi_R(\mathcal{S}^N)|-|\Psi_R(\mathcal{S})|\Big]
\nonumber \\
&&\!\!\!\!\!\!\!\!+\mathbb{E}\Big[|\Psi_R(\mathcal{S})|-|\Psi(\mathcal{S})|\Big].
\end{eqnarray}
By Lemma \ref{lem13} and the dominated convergence theorem,
$$\lim_{N\to\infty}\mathbb{E}|\Psi_R(\mathcal{S}^N)-\Psi_R(\mathcal{S})|=0,$$
which implies
\begin{equation}\label{es69}
\lim_{N\to\infty}\mathbb{E}|\Psi_R(\mathcal{S}^N)|=\mathbb{E}|\Psi_R(\mathcal{S})|.
\end{equation}
For the first term on right hand side of (\ref{es1124}), since $\Phi_1,\ldots,\Phi_m$ are bounded, it follows that
\begin{eqnarray}\label{es1125}
&&\!\!\!\!\!\!\!\!\mathbb{E}|\Psi_R(\mathcal{S}^{N})-\Psi(\mathcal{S}^N)|
\nonumber \\
\leq&&\!\!\!\!\!\!\!\!C\sup_{t\in[0,T]}\Bigg\{\frac{1}{N}\sum_{i=1}^N\mathbb{E}\big|\varpi_R^{(1)}(t,X^{N,i})-\varpi^{(1)}(t,X^{N,i})\big|\Bigg\}
\nonumber \\
&&\!\!\!\!\!\!\!\!+C\sup_{t\in[0,T]}\Bigg\{\frac{1}{N}\sum_{i=1}^N\mathbb{E}\big|\varpi_R^{(2)}(t,X^{N,i},\mathcal{S}^{N})-\varpi^{(2)}(t,X^{N,i},\mathcal{S}^{N})\big|\Bigg\}.~~~~
\end{eqnarray}
By  $\mathbf{(A_3)}$ and (\ref{es4})-(\ref{es6}),
\begin{eqnarray*}
&&\!\!\!\!\!\!\!\!\frac{1}{N}\sum_{i=1}^N\mathbb{E}\big|\varpi_R^{(2)}(t,X^{N,i},\mathcal{S}^{N})-\varpi^{(2)}(t,X^{N,i},\mathcal{S}^{N})\big|
\nonumber \\
\leq&&\!\!\!\!\!\!\!\!\frac{1}{N}\sum_{i=1}^N\mathbb{E}\Bigg\{\int_0^T|{}_{{\mathbb{V}}^*}\langle \mathcal{A}(s,X^{N,i}_s,\mathcal{S}^{N}_s),l\rangle_{\mathbb{V}}|\cdot \mathbf{1}_{\big\{|{}_{{\mathbb{V}}^*}\langle \mathcal{A}(s,X^{N,i}_s,\mathcal{S}^{N}_s),l\rangle_{\mathbb{V}}|\geq R\big\}}ds\Bigg\}
\nonumber \\
\leq&&\!\!\!\!\!\!\!\!\|l\|_{\mathbb{V}}\frac{1}{N}\sum_{i=1}^N\Bigg\{\Big(\int_0^T\mathbb{E}\|\mathcal{A}(s,X^{N,i}_s,\mathcal{S}^{N}_s)\|_{{\mathbb{V}}^*}^{\frac{\alpha}{\alpha-1}}ds\Big)^{\frac{\alpha-1}{\alpha}}
\nonumber \\
&&\!\!\!\!\!\!\!\!
\cdot \Big(\int_0^T\mathbb{P}\big(|{}_{{\mathbb{V}}^*}\langle \mathcal{A}(s,X^{N,i}_s,\mathcal{S}^{N}_s),l\rangle_{\mathbb{V}}|\geq R\big)ds\Big)^{\frac{1}{\alpha}}\Bigg\}
\nonumber \\
\leq&&\!\!\!\!\!\!\!\!\Bigg\{\|l\|_{\mathbb{V}}\frac{1}{N}\sum_{i=1}^N\mathbb{E}\int_0^T\|\mathcal{A}(s,X^{N,i}_s,\mathcal{S}^{N}_s)\|_{{\mathbb{V}}^*}^{\frac{\alpha}{\alpha-1}}ds\Bigg\}\Big/R^{\frac{1}{\alpha-1}}
\nonumber \\
\leq&&\!\!\!\!\!\!\!\!\Bigg\{C_{\|l\|_{\mathbb{V}}}\mathbb{E}\int_0^T\Big(1+\frac{1}{N}\sum_{i=1}^N\big(\|X^{N,i}_s\|_{\mathbb{V}}^{\alpha}+\|X^{N,i}_s\|_{\mathbb{H}}^{\beta}\|X^{N,i}_s\|_{\mathbb{V}}^{\alpha}+\|X^{N,i}_s\|_{\mathbb{H}}^{\beta}\big)\Big)ds\Bigg\}\Big/R^{\frac{1}{\alpha-1}}
\nonumber \\
\leq&&\!\!\!\!\!\!\!\!C_{p,T,\|l\|_{\mathbb{V}}}\Big/R^{\frac{1}{\alpha-1}},
\end{eqnarray*}
where we have used the fact that the law of $X^N$ is symmetric in the fourth step.

  Thus we can get
\begin{equation}\label{es1127}
\lim_{R\to\infty}\sup_{N\in\mathbb{N}}\sup_{t\in[0,T]}\Bigg\{\frac{1}{N}\sum_{i=1}^N\mathbb{E}\big|\varpi_R^{(2)}(t,X^{N,i},\mathcal{S}^{N})-\varpi^{(2)}(t,X^{N,i},\mathcal{S}^{N})\big|\Bigg\}=0.
\end{equation}
Similarly, we also infer that
\begin{equation}\label{es1128}
\lim_{R\to\infty}\sup_{N\in\mathbb{N}}\sup_{t\in[0,T]}\Bigg\{\frac{1}{N}\sum_{i=1}^N\mathbb{E}\big|\varpi_R^{(1)}(t,X^{N,i})-\varpi^{(1)}(t,X^{N,i})\big|\Bigg\}=0.
\end{equation}
Combining (\ref{es1125})-(\ref{es1128}), we have
\begin{equation}\label{es71}
\lim_{R\to\infty}\sup_{N\in\mathbb{N}}\mathbb{E}|\Psi_R(\mathcal{S}^{N})-\Psi(\mathcal{S}^N)|=0.
\end{equation}
A similar argument shows
\begin{equation}\label{es72}
\lim_{R\to\infty}\mathbb{E}|\Psi_R(\mathcal{S})-\Psi(\mathcal{S})|=0.
\end{equation}
Now, combining (\ref{es69}) and (\ref{es71})-(\ref{es72}), we complete the proof of (\ref{es21}).

\vspace{2mm}
\textbf{Step 2.}  Note that
\begin{eqnarray}\label{es77}
\Psi(\mathcal{S}^N)=\frac{1}{N}\sum_{i=1}^{N}\big(\mathcal{M}_t^{N,i}-\mathcal{M}_s^{N,i}\big)\Phi_1(X^{N,i}_{s_1})\cdots \Phi_m(X^{N,i}_{s_m}),
\end{eqnarray}
where for each $l\in\mathbb{V}$,
\begin{eqnarray*}
\mathcal{M}_t^{N,i}:=&&\!\!\!\!\!\!\!\!{}_{\mathbb{V}^*}\langle X^{N,i}_t,l\rangle_{\mathbb{V}}-{}_{\mathbb{V}^*}\langle \xi^i,l\rangle_{\mathbb{V}}-\int_0^t{}_{\mathbb{V}^*}\langle\mathcal{A}(s,X^{N,i}_s,\mathcal{S}^N_s),l\rangle_{\mathbb{V}} ds
\nonumber \\
=&&\!\!\!\!\!\!\!\!\langle\int_0^t\mathcal{B}(s,X^{N,i}_s,\mathcal{S}^N_s) dW^i_s,l\rangle_{\mathbb{H}},
\end{eqnarray*}
which turns out to be  a square integrable martingale. Therefore, we can get
\begin{eqnarray}\label{es75}
\mathbb{E}\Big[\mathcal{M}_t^{N,i}\mathcal{M}_s^{N,k}|\mathscr{F}_s^N\Big]=\mathbb{E}\Big[\mathcal{M}_s^{N,i}\mathcal{M}_s^{N,k}|\mathscr{F}_s^N\Big],
\end{eqnarray}
where $\mathscr{F}_s^N:=\sigma\Big\{X_r^{N,i}:r\leq s,1\leq i\leq N\Big\}$.

Since for $i\neq k$, $W^i$ and $W^k$ are independent, it follows that
$$\langle \mathcal{M}^{N,i},\mathcal{M}^{N,k}\rangle_t=0,$$
which implies that $\mathcal{M}^{N,i}\mathcal{M}^{N,k}$ is a martingale for $i\neq k$, i.e.,
\begin{eqnarray}\label{es76}
\mathbb{E}\Big[\mathcal{M}_t^{N,i}\mathcal{M}_t^{N,k}|\mathscr{F}_s^N\Big]=\mathbb{E}\Big[\mathcal{M}_s^{N,i}\mathcal{M}_s^{N,k}|\mathscr{F}_s^N\Big].
\end{eqnarray}
Combining (\ref{es75})-(\ref{es76}), for $i\neq k$, we have
$$\mathbb{E}\Big[\big(\mathcal{M}_t^{N,i}-\mathcal{M}_s^{N,i}\big)\big(\mathcal{M}_t^{N,k}-\mathcal{M}_s^{N,k}\big)\Big]=0.$$
Recall (\ref{es77}), noting that $\Phi_1,\ldots,\Phi_m$ are bounded, it follows that
\begin{eqnarray}\label{es74}
\mathbb{E}|\Psi(\mathcal{S}^N)|^2\leq&&\!\!\!\!\!\!\!\! \frac{1}{N^2}\sum_{i,k=1}^{N}\mathbb{E}\Big[\big(\mathcal{M}_t^{N,i}-\mathcal{M}_s^{N,i}\big)\big(\mathcal{M}_t^{N,k}-\mathcal{M}_s^{N,k}\big)\Big].
\nonumber \\
=&&\!\!\!\!\!\!\!\!\frac{1}{N^2}\sum_{i=1}^{N}\mathbb{E}|\mathcal{M}_t^{N,i}-\mathcal{M}_s^{N,i}|^2
\nonumber \\
\leq&&\!\!\!\!\!\!\!\!\frac{C_{T,\|l\|_{\mathbb{H}}}}{N^2}\sum_{i=1}^{N}\Big(1+\mathbb{E}\big[\sup_{t\in[0,T]}\|X^{N,1}_t\|_{\mathbb{H}}^2\big]\Big)
\nonumber \\
\leq&&\!\!\!\!\!\!\!\!\frac{C_{T,\|l\|_{\mathbb{H}}}}{N}\to 0 ~\text{as}~N\to\infty,
\end{eqnarray}
which implies (\ref{es26}) holds.
\end{proof}

\vspace{2mm}
 \textbf{Proof of Theorem \ref{th4}.} Due to Lemma \ref{lem4}, it clear that $\mathbb{P}$-a.s.~$\Psi(\mathcal{S})=0$.  Note that since $\mathcal{S}^N_0$ converges to $\mu_0$  in probability by condition $\mathbf{(A_6)}$, we have
\begin{equation}\label{es63}
\mathcal{S}_0=\mu_0,~~\mathbb{P}\text{-a.s.}.
\end{equation}

Since $C_b(\mathbb{H};\mathbb{R})$ is separable, one can find a countable subsect $\mathscr{D}_0\subset C_b(\mathbb{H};\mathbb{R})$.
Let
$$\tilde{\Omega}:=\Big\{\tilde{\omega}\in\Omega:\mathcal{S}(\tilde{\omega})\in\mathbb{K}~\text{and}~ \Psi(\mathcal{S})=0~\text{for}~\Phi_1,\ldots,\Phi_m\in\mathscr{D}_0\Big\}.$$
Then $\mathbb{P}(\tilde{\Omega})=1$, and for all $\tilde{\omega}\in \tilde{\Omega}$,
$$\int \Big(\varpi(t,w,\mathcal{S}(\tilde{\omega}))-\varpi(s,w,\mathcal{S}(\tilde{\omega}))\Big)\Phi_1(w_{s_1})\cdots\Phi_m(w_{s_m})\mathcal{S}(\tilde{\omega})(dw)=0,$$
which means
$$\mathcal{M}_l(t,w,\mathcal{S}(\tilde{\omega}))={}_{{\mathbb{V}}^*}\langle w_{t},l\rangle_{\mathbb{V}}-{}_{{\mathbb{V}}^*}\langle w_0,l\rangle_{\mathbb{V}}-\int_0^{t}{}_{{\mathbb{V}}^*}\langle \mathcal{A}(s,w_s,\mathcal{S}_s(\tilde{\omega})),l\rangle_{\mathbb{V}}ds$$
is a $\mathcal{S}(\tilde{\omega})$-martingale. Once we can show that the  quadratic variation process  of $\mathcal{M}_l(t,w,\mathcal{S}(\tilde{\omega}))$ is
\begin{equation}\label{es62}
\int_0^{t}\|\mathcal{B}(s,w_s,\mathcal{S}_s(\tilde{\omega}))^*l\|_U^2ds,
\end{equation}
then $\mathcal{S}(\tilde{\omega})$ is a solution of martingale problem (\ref{eqSPDE}) with initial law $\mu_0$, i.e., $\mathcal{S}(\tilde{\omega})\in \mathscr{M}_{\mu_0}^{\mathcal{A},\mathcal{B}}$.

\vspace{1mm}
We now prove (\ref{es62}). In fact, it suffices to show that
\begin{eqnarray}\label{es1131}
&&\!\!\!\!\!\!\!\!\mathbb{E}\Big|\int \Big(\varpi^2(t,w,\mathcal{S})-\varpi^2(s,w,\mathcal{S})-\int_s^t\|\mathcal{B}(r,w_r,\mathcal{S}_r)^*l\|_U^2dr\Big)
\nonumber \\
&&\!\!\!\!\!\!\!\!~~~~~~~~\cdot\Phi_1(w_{s_1})\cdots\Phi_m(w_{s_m})\mathcal{S}(dw)\Big|=0.
\end{eqnarray}
Recall the definition of $\varpi$, by (\ref{es4}), H\"{o}lder's inequality and Burkholder-Davis-Gundy's inequality, it leads to
\begin{eqnarray*}
\mathbb{E}\Big|\frac{1}{N}\sum_{i=1}^{N}\varpi^2(t,X^{N,i},\mathcal{S}^N)\Big|^{r}\leq&&\!\!\!\!\!\!\!\!\frac{C}{N}\sum_{i=1}^{N}\mathbb{E}\Big|\varpi(t,X^{N,i},\mathcal{S}^N)\Big|^{2r}
\nonumber \\
\leq&&\!\!\!\!\!\!\!\!C_{r,T},
\end{eqnarray*}
for some $r>1$. As the proof of Lemma \ref{lem4}, we can show that
$$\lim_{N\to\infty}\mathbb{E}\Big|\int \varpi^2(t,w,\mathcal{S}^N)\mathcal{S}^N(dw)-\int \varpi^2(t,w,\mathcal{S})\mathcal{S}(dw)\Big|=0,$$
and
$$\lim_{N\to\infty}\mathbb{E}\Big|\int \int_0^t\|\mathcal{B}(s,w_s,\mathcal{S}^N_s)\|_{L_2(U;{\mathbb{H}})}^2ds\mathcal{S}^N(dw)-\int \int_0^t\|\mathcal{B}(s,w_s,\mathcal{S}_s)\|_{L_2(U;{\mathbb{H}})}^2ds\mathcal{S}(dw)\Big|=0,$$
which implies that (\ref{es1131}) holds.

\vspace{2mm}
In conclusion, for any sequence $(N)_{N\in\mathbb{N}}$ there exists a subsequence $(N_k)_{k\in\mathbb{N}}$ such that $\mathcal{S}^{N_k}$ converges weakly to a martingale solution $\mathcal{S}$ of (\ref{eqSPDE}) with initial law $\mu_0$ in $\mathbb{K}$, as $k\to\infty$. Thus the first assertion in Theorem \ref{th4} follows.

\vspace{2mm}
Furthermore, if  one of  $(\mathbf{A'_5})$-$(\mathbf{A''_5})$ hold, by  Theorem \ref{th8} the martingale solution of (\ref{eqSPDE}) is unique, and thus the convergence above holds for whole sequence $\{N\}_{N\in\mathbb{N}}$, i.e., there is a unique martingale solution $\Gamma\in \mathscr{M}_{\mu_0}^{\mathcal{A},\mathcal{B}}$ of (\ref{eqSPDE}) such that
$$\mathcal{S}=\Gamma,~\mathbb{P}\text{-a.s.}.$$
Therefore the limit of $\{\mathbb{P}\circ (\mathcal{S}^{N})^{-1}\}_{N\in\mathbb{N}}$ is identified as $\delta_{\Gamma}$, which implies
$$\mathcal{W}_{2,\mathbb{C}_T}(\mathcal{S}^{N},\Gamma)^2+\mathcal{W}_{\alpha,L^{\alpha}([0,T];\mathbb{V})}(\mathcal{S}^{N},\Gamma)^{\alpha}\to 0,~\text{in probability}.$$
In order to get the convergence (\ref{es111}), it suffices to show that the family
\begin{equation}\label{eq42}
\{\mathcal{W}_{2,\mathbb{C}_T}(\mathcal{S}^{N},\Gamma)^2+\mathcal{W}_{\alpha,L^{\alpha}([0,T];\mathbb{V})}(\mathcal{S}^{N},\Gamma)^{\alpha}\}_{N\in\mathbb{N}} \end{equation}
is uniformly integrable.
Note that,
in view of (\ref{es1132}), we know
$$\Gamma\in \mathscr{P}_{r_1'}(\mathbb{C}_T)\cap \mathscr{P}_{r_2'}(L^{\alpha}([0,T];\mathbb{V})).$$
This combining with energy estimates (\ref{es4})-(\ref{es6}) implies the uniform integrability of family (\ref{eq42}).
 We complete the second assertion in Theorem \ref{th4} and thus we finish the proof.
 \hspace{\fill}$\Box$

\section{Appendix}

In this section, we collect two lemmas that are used in the proof of our main results.

\begin{lemma}\label{lem00}(cf.~\cite[Corollary 5.6 ]{CD1})
If $(E,d)$ is a complete separable metric space, for any $ p>1$, any subset $\mathscr{K}\subset\mathscr{P}_p(E)$, relatively compact for
the topology of weak convergence of probability measures, any $x_0\in E$, and any
sequences $(\lambda_m)_{m\geq 1}$ and $(\kappa_m)_{m\geq 1}$ of positive real numbers tending to $\infty$ with $m$, the
set:
$$\mathscr{K}\cap\Bigg\{\mu\in\mathscr{P}_p(E):\forall m\geq 1,\int_{\{d(x_0,x)\geq \lambda_m\}}d(x_0,x)^pd\mu(x)<\frac{1}{\kappa_m}\Bigg\},$$
is relatively compact for the Wasserstein distance $\mathcal{W}_p$.
\end{lemma}

We recall the following stochastic Gronwall's lemma and its proof.
\begin{lemma}\label{appen2} (cf. \cite[Lemma
5.3]{GZ1})
Fix $T>0$. Assume that
$X,Y,Z,R:[0,T)\times\Omega\rightarrow\mathbb{R}$ are real-valued,
non-negative stochastic process. Let $\tau<T$ be a stopping time
such that $$\mathbb{E}\int_0^\tau(RX+Z)ds<\infty.$$ Assume,
moreover, that for some fixed constant $\kappa$,
$$\int_0^\tau{R}ds<\kappa~~~\mathbb{P}\text{-a.s.}$$
Suppose that for all stopping times $0\leq\tau_a<\tau_b\leq\tau$
\begin{eqnarray}{\mathbb{E}}\left(\sup_{t\in[\tau_a,\tau_b]}X+\int_{\tau_a}^{\tau_b}Yds\right)\leq c_0{\mathbb{E}}\left(X(\tau_a)+\int_{\tau_a}^{\tau_b}(RX+Z)ds\right),\label{a.1}
\end{eqnarray}
where $c_0$ is a constant independent of the choice of
$\tau_a,\tau_b$. Then
\begin{eqnarray}{\mathbb{E}}\left(\sup_{t\in[0,\tau]}X+\int_{0}^{\tau}Yds\right)\leq c{\mathbb{E}}\left(X(0)+\int_{0}^{\tau}Zds\right),\label{a.2}
\end{eqnarray}
where $c=c_{c_o,T,\kappa}$.
\end{lemma}

\begin{proof}
We include the proof for reader's convenience. Choose a sequence of stopping times $0 = \tau_0 < \tau_1 < ... < \tau_N <
\tau_{N+1} = \tau$, so that
\begin{eqnarray}
\int_{\tau_{k-1}}^{\tau_k}Rds\leq \frac{1}{2c_0}~~\mathbb{P}\text{-a.s.}.\label{ap.1}
\end{eqnarray}
For each pair $\tau_{k-1}$ and $\tau_k$, we take  $\tau_a=\tau_{k-1}$ and $\tau_b=\tau_k$ in (\ref{a.1}). By making use
of (\ref{ap.1}) we deduce
\begin{eqnarray}{\mathbb{E}}\left(\sup_{t\in[\tau_{k-1},\tau_k]}X+\int_{\tau_{k-1}}^{\tau_k}Yds\right)\leq C{\mathbb{E}}\left(X(\tau_{k-1})+\int_{\tau_{k-1}}^{\tau_k}Zds\right),\label{ap.2}
\end{eqnarray}
Assuming, by induction on $j$, that
\begin{eqnarray}{\mathbb{E}}\left(\sup_{t\in[0,\tau_j]}X+\int_{0}^{\tau_j}Yds\right)\leq C{\mathbb{E}}\left(X(0)+\int_{0}^{\tau_j}Zds\right),\label{ap.3}
\end{eqnarray}
then
\begin{eqnarray}
{\mathbb{E}}\left(\sup_{t\in[0,\tau_{j+1}]}X+\int_{0}^{\tau_{j+1}}Yds\right)\leq&&\!\!\!\!\!\!\!\! C{\mathbb{E}}\left(X(0)+\int_{0}^{\tau_j}Zds+\sup_{t\in[\tau_{j},\tau_{j+1}]}X+\int_{\tau_{j}}^{\tau_{j+1}}Yds\right)
\nonumber \\\leq&&\!\!\!\!\!\!\!\! C{\mathbb{E}}\left(X(0)+\int_{0}^{\tau_{j+1}}Zds+X(\tau_j)\right)
\nonumber \\\leq&&\!\!\!\!\!\!\!\! C{\mathbb{E}}\left(X(0)+\int_{0}^{\tau_{j+1}}Zds\right),
\end{eqnarray}
which implies \eref{a.2}.
\end{proof}

\textbf{Acknowledgments.}
The authors would like to thank the anonymous referees, an Associate
Editor for  very constructive suggestions and valuable comments that significantly improved the quality of the
paper.

Shihu Li is the corresponding author.




\end{document}